\documentclass[letterpaper, 11pt]{amsart}
\usepackage[margin=1.2in]{geometry}
\usepackage{amssymb,latexsym,eufrak,amsmath,amscd, graphics}
\usepackage{xspace,xcolor}
\usepackage[breaklinks,colorlinks,citecolor=teal,linkcolor=teal,urlcolor=teal,pagebackref,hyperindex]{hyperref}
\usepackage[alphabetic]{amsrefs} 
\usepackage{mathrsfs}
\usepackage{amsfonts}
\usepackage{tikz-cd} 
\usepackage{mathtools} 
\usepackage{stmaryrd} 
\usetikzlibrary{calc, intersections} 

\tikzset{symbol/.style={draw=none,every to/.append style={edge node={node [sloped, allow upside down, auto=false]{$#1$}}}}}

\newcommand{\onto}{\twoheadrightarrow}

\tikzset{
    labl/.style={anchor=north, rotate=90, inner sep=1mm}
}

\let\proof\noindentitproof 

\def\Mustata{Mus\-ta\-\c{t}\u{a}\xspace}

\DeclareMathOperator{\Ker}{ker}

\DeclareMathOperator{\Hom}{Hom}

\DeclareMathOperator{\bbc}{\mathbb{C}}
\DeclareMathOperator{\bbr}{\mathbb{R}}
\DeclareMathOperator{\bbz}{\mathbb{Z}}
\DeclareMathOperator{\bbq}{\mathbb{Q}}
\DeclareMathOperator{\bbg}{\mathbb{G}}
\DeclareMathOperator{\bba}{\mathbb{A}}

\DeclareMathOperator{\bbp}{\mathbb{P}}

\DeclareMathOperator{\bfo}{\bf 1}

\DeclareMathOperator{\os}{\mathcal{O}}
\DeclareMathOperator{\vs}{\mathcal{V}}

\DeclareMathOperator{\ds}{\mathcal{D}}

\DeclareMathOperator{\rs}{\mathcal{R}}
\DeclareMathOperator{\ms}{\mathcal{M}}

\DeclareMathOperator{\ps}{\mathcal{P}}

\DeclareMathOperator{\js}{\mathcal{J}}

\DeclareMathOperator{\bff}{\bf{f}}

\DeclareMathOperator{\bfh}{\bf{h}}
\DeclareMathOperator{\bfg}{\bf{g}}

\DeclareMathOperator{\sDer}{\mathcal{D}{\it er}}

\DeclareMathOperator{\num}{\operatorname{num}}

\DeclareMathOperator{\lct}{\operatorname{lct}}

\DeclareMathOperator{\relint}{\operatorname{relint}}

\DeclareMathOperator{\wt}{\operatorname{wt}}

\DeclareMathOperator{\Conv}{\operatorname{Conv}}

\DeclareMathOperator{\supp}{\operatorname{supp}}

\DeclareMathOperator{\Spec}{\operatorname{Spec}}

\DeclareMathOperator{\Gr}{\operatorname{Gr}}

\newtheorem{lemma}{Lemma}[section]
\newtheorem{theorem}[lemma]{Theorem}
\newtheorem{corollary}[lemma]{Corollary}

\newtheorem{claim}[lemma]{Claim}

\theoremstyle{definition}
\newtheorem{definition}[lemma]{Definition}
\newtheorem{example}[lemma]{Example}
\newtheorem{remark}[lemma]{Remark}

\begin{document}

\begin{abstract}
We compute the minimal exponents of semi-quasihomogeneous and Khovanskii non-degenerate complete intersections.
\end{abstract}



\author[Jonghyun Lee]{Jonghyun Lee}
\address{Department of Mathematics, University of Michigan,
Ann Arbor, MI 48109, USA}
\email{{nuyhgnoj@umich.edu}}

\vspace*{-3em}
\title[Minimal exponents]{The minimal exponents of semi-quasihomogeneous and Khovanskii non-degenerate complete intersections}
\maketitle

\vspace{-2em}

\section{Introduction}

Let $X$ be a smooth complex variety. Given a nonempty hypersurface $Z$ in $X$ defined by $f\in \os_X(X)$, the \textit{minimal exponent} $\widetilde{\alpha}(Z)$ was defined by Saito in \cite{Sai93} as follows. Recall that the Bernstein-Sato polynomial of $f$ is the monic polynomial $b_f(s)\in \bbc[s]$ of minimal degree such that 
$$
b_f(s)f^s \in \ds_X[s]\cdot f^{s+1},
$$
where $\ds_X$ is the sheaf of differential operators on $X$, which acts on the formal symbol $f^s$ in the expected way. The roots of $b_f(s)$ are negative rational numbers \cite{Kas76}. Since $Z$ is nonempty, by specializing $s$ to $-1$, it is easy to see that $b_f(-1)=0$. By definition, $\widetilde{\alpha}(Z) = \widetilde{\alpha}(f)$ is the negative of the largest root of $b_f(s)/(s+1)$ (with the convention this is $\infty$ if $b_f(s)=s+1$).

The minimal exponent of a hypersurface is an important invariant of singularities. For example, the \textit{log canonical threshold} of the pair $(X,Z)$ is given by $\lct(X,Z) = \min\{1,\widetilde{\alpha}(Z)\}$ \cite[Theorem 10.6]{Kol97}, and $Z$ has rational singularities if and only if $\widetilde{\alpha}(Z)>1$ \cite{Sai93}. Moreover, the minimal exponent characterizes the higher rational singularities of $Z$ (see \cite{JKSY22} and \cite{MOPW23}) as well as the higher Du Bois singularities (see \cite[Appendix]{FL24} and \cite{MP25}). If $Z$ has isolated singularities, then the minimal exponent coincides with the so-called \textit{complex singularity index} or \textit{Arnold exponent}, which can be described via asymptotic expansions of integrals along vanishing cycles (see \cite{Mal74} and \cite{AGZV12}).

In \cite{CDMO24}, Chen, Dirks, \Mustata, and Olano introduced and studied an extension of the minimal exponent $\widetilde{\alpha}(Z)$ to the case when $Z$ is a local complete intersection in $X$ of pure codimension $r\geq 1$. One of their main results described $\widetilde{\alpha}(Z)$ in terms of the minimal exponent of a hypersurface. 
Suppose that $Z$ is defined in $X$ by $f_1,\dots,f_r\in \os_X(X)$, and consider $g = \sum_{i=1}^r z_i f_i\in \os_{X\times \bba^r}(X\times \bba^r)$, where $z_1,\dots,z_r$ are the coordinates of $\bba^r$. Then we have $\widetilde{\alpha}(Z)=\widetilde{\alpha}(g|_{X\times (\bba^r\setminus \{0\})})$. The minimal exponent can also be described in terms of the Bernstein-Sato polynomial $b_{\bff}(s)$, associated to $\bff=(f_1,\dots,f_r)$, that was introduced in \cite{BMS06}. If $Z$ is nonempty, then $b_{\bff}(-r)=0$ and $\widetilde{\alpha}(Z)$ is the negative of the largest root of $b_{\bff}(s)/(s+r)$ \cite{Dir25}. Similar to the case of hypersurfaces, we have $\lct(X,Z)=\min\{r, \widetilde{\alpha}(Z)\}$, and $Z$ has rational singularities if and only if $\widetilde{\alpha}(Z)>r$ \cite{CDMO24}. Moreover, the minimal exponent characterizes the higher rational and higher Du Bois singularities of local complete intersections (see \cite{MP22} and \cite{CDM24}). 

Two important classes of hypersurfaces where the minimal exponent is explicitly known are weighted homogeneous hypersurfaces with isolated singularities (see \cite{Ste77} or \cite[Theorem 6.18]{Kas03}) and Newton non-degenerate hypersurfaces with isolated singularities (see \cite{Var82}, \cite{EL82}, or \cite{Sai88}), and more generally, semi-quasihomogeneous hypersurfaces (see \cite[Section 2.5]{Sai17}), a natural generalization of the former. To date, beyond the hypersurface case, explicit examples of minimal exponents are known only for certain subclasses of weighted homogeneous complete intersections $Z$ in $\bba^n$ with isolated singularities. The first example, given in \cite[Example 4.23]{CDMO24}, provides an explicit formula for the minimal exponent when $Z$ is defined by homogeneous equations (with respect to the standard grading) of the same degree. The second example, given in \cite[Theorem 1.1]{CDM25}, extends the first example and treats those $Z$ defined by homogeneous equations with possibly different degrees such that the hypersurfaces defined by those equations intersect transversely away from the origin. The third example, given in \cite[Corollary D]{CDO25}, treats those $Z$ with Du Bois singularities such that the expected formula for the minimal exponent is an integer.

The main results of this paper are the formulas for the minimal exponents of semi-quasihomogeneous and Khovanskii non-degenerate complete intersections that satisfy mild transversality conditions. Semi-quasihomogeneous complete intersections, a natural generalization of weighted homogeneous complete intersections with isolated singularities, have been extensively studied in, for example, \cite[]{Dam89}, \cite[]{Giu77}, \cite[]{GH78}, \cite[]{Ran78}, \cite[]{May97}, and \cite[]{BM02}. Also, Khovanskii non-degenerate complete intersections, a natural generalization of Newton non-degenerate hypersurfaces introduced by Khovanskii \cite[]{Kho77}, have been extensively studied in, for example, \cite[]{Var77}, \cite[]{Mor84}, \cite[]{Dam89}, \cite[]{Gaf92}, \cite[]{Oka97}, \cite{BA07}, \cite{ZG09}, \cite{BGZG18}, and \cite[]{Ngu22}.

To state the formula for semi-quasihomogeneous complete intersections, consider positive integers $w_1,\dots,w_n$ and the grading on the polynomial ring $R=\bbc[x_1,\dots,x_n]$ given by $\deg x_i=w_i$ for $1\leq i\leq n$. We refer to homogeneous elements of $R$ with respect to this grading as weighted homogeneous. For $f\in R$, we denote by $\wt(f)$ the smallest degree of a monomial $x^u=x_1^{u_1}\cdots x_n^{u_n}$ that appears with a nonzero coefficient in $f$.

\begin{theorem}\label{weighted minimal exponent}
With the above notation, let $Z \subset \bba^n$ be a closed subscheme defined by $f_1,\dots,f_r \in (x_{1},\dots,x_{n})^{2}$ with $\wt(f_i)=d_i$ such that $d_1\leq \cdots \leq d_r$, and write $f_i = f_{i,d_i}+f_{i,>d_i}$, where $f_{i,d_i}$ is weighted homogeneous of degree $d_i$ and $\wt(f_{i,>d_i})>d_i$. For $i\in \{1,\dots,r-1\}$ with $d_i< d_{i+1}$, and for $i=r$, suppose that $f_1,\dots,f_i$ is semi-quasihomogeneous, that is, the closed subscheme $V(f_{1,d_1},\dots,f_{i,d_i})\subset \bba^n$ is a complete intersection of codimension $i$ with an isolated singularity at $0$. Then we have
$$
\widetilde{\alpha}_0(Z) = \min\{ i + \tfrac{1}{d_i}(w_1+\cdots+w_n - d_1-\cdots - d_i) \ | \ 1\leq i\leq r \}.
$$
\end{theorem}

For the precise definition of the local version $\widetilde{\alpha}_0(Z)$ of the minimal exponent, see Section \ref{the minimal exponent of a local complete intersection}. The above theorem implies the formulas for the minimal exponent in \cite[Example 4.23]{CDMO24} and \cite[Theorem 1.1]{CDM25} by taking $f_1,\dots,f_r$ to be homogeneous with respect to the standard weights $w_{1}= \cdots=w_{n}=1$. We note that if $f_1,\dots,f_r$ are weighted homogeneous with $\sum_{i=1}^r d_i \leq \sum_{i=1}^n w_i$ and we only assume that $V(f_1,\dots,f_r) \subset \bba^n$ is a complete intersection of codimension $r$ with an isolated singularity at the origin, then Chen, Dirks, and Olano obtained the formula in the theorem above when each side is replaced by its floor and when the right-hand-side is an integer \cite[Corollary D]{CDO25}.

Next, we state a pair of formulas for the minimal exponent of Khovanskii non-degenerate complete intersections under different assumptions. Let $R=\bbc[x_1,\dots,x_n]$ and let $Z \subset \bba^n$ be a closed subscheme defined by polynomials $f_1,\dots,f_r\in (x_{1},\dots,x_{n})^{2}$ such that $Z$ is a complete intersection of codimension $r$ in a neighborhood of $0$. We set $g=\sum_{i=1}^{r}z_{i}f_{i}\in R[z_{1},\dots,z_{r}]$. Let us identify the standard basis of $\bbr^{r+n}$ with the variables $z_{1},\dots,z_{r},x_{1},\dots,x_{n}$, and consider the standard projection $\pi_{i}\colon \bbr^{r+n}\to\bbr^{r-1+n}$ obtained by omitting the $z_{i}$-coordinate for $1\leq i\leq r$. For a nonzero polynomial $f\in\bbc[y_1,\dots,y_m]$, let $\supp f$ denote the set of elements $u=(u_{1},\dots,u_{m})\in \bbz^{m}_{\geq 0}$ such that the monomial $y^{u}=y_{1}^{u_{1}}\cdots y_{1}^{u_{1}}$ appears in $f$ with nonzero coefficient, and let $P(f) \subset \bbr^m$ denote the Newton polyhedron of $f$ (recall that $P(f)$ is the convex hull of $\bigcup_{u\in \supp f} (u+\bbr^m_{\geq 0})$).

\begin{theorem}\label{Khovanskii minimal exponent}
With the above notation, suppose that the $|I|$-tuple $(f_{i})_{i\in I}$ is Khovanskii non-degenerate for all subsets $I \subset \{1,\dots,r\}$. Then we have
$$
\widetilde{\alpha}_0(Z) = 1/ \min\bigg\{t>0 \ \big| \ (t,\dots,t) \in \bigcap_{i=1}^r \pi_i^{-1}(\pi_i(P(g)))\bigg\}.
$$
\end{theorem}

For the definition of Khovanskii non-degeneracy, see Section \ref{Non-degeneracy}. Before stating the next theorem, we introduce some additional notation. Let $e_{1},\dots,e_{n}$ denote the standard basis of $\bbr^{n}$ and put
$$
(\bbr^n)^*_{\geq 0} = \{v\in (\bbr^n)^* \ | \ v(e_j)\geq 0 \text{ for $j\in \{1,\dots, n\}$} \},
$$
where $(\bbr^n)^*$ is the dual of $\bbr^{n}$. For an element $v\in (\bbr^n)^*_{\geq 0}$, we write $|v|=\sum_{j=1}^{n}v(e_{j})$ and $d(v,f) =\min\{v(u) \ |\ u\in \supp f\}$ for nonzero $f\in R$. For a nonzero polynomial $f\in R$, recall that the normal fan of the Newton polyhedron $P(f)\subset \bbr^{n}$ is a fan with support $(\bbr^n)^*_{\geq 0}$ whose rays are perpendicular to the facets of $P(f)$ (for a review of normal fans, see Section \ref{Normal fans}).
 
\begin{theorem}\label{nested Khovanskii minimal exponent}
With the above notation, suppose that  the following conditions hold:
\begin{enumerate}
\item[i)] For all $i,j\in \{1,\dots,r\}$, there exists a real number $\epsilon>0$ such that $P(f_{i})\subset \epsilon \cdot P(f_{j})$.
\item[ii)] $P(f_1) \supset\cdots \supset P(f_r)$.
\item[iii)] For all $i\in \{1,\dots,r-1\}$ with $P(f_i)\neq P(f_{i+1})$, and for $i=r$, the $i$-tuple $(f_1,\dots,f_i)$ is Khovanskii non-degenerate. 
\end{enumerate} 
Then we have
$$
\widetilde{\alpha}_0(Z) = 1/ \min\bigg\{t>0 \ \big| \ (t,\dots,t) \in \bigcap_{i=1}^r \pi_i^{-1}(\pi_i(P(g)))\bigg\}.
$$
Moreover, if $\Sigma_{0}$ is any fan that refines the normal fan of $P(f_{i})$ for each $i$, then we have
$$
\widetilde{\alpha}_0(Z) = \min_{v}
\min\{ i + \tfrac{1}{d(v,f_{i})}(|v| - d(v,f_{1})-\cdots - d(v,f_{i})) \ | \ \text{$1\leq i\leq r$ with $d(v,f_{i})\neq 0$} \},
$$
where the outer minimum runs over the primitive ray generators $v$ of $\Sigma_{0}$.
\end{theorem}

We view Theorem \ref{nested Khovanskii minimal exponent} as the Khovanskii non-degenerate analogue of Theorem \ref{weighted minimal exponent}. Condition i) of Theorem \ref{nested Khovanskii minimal exponent} guarantees that the toric morphism $\pi\colon Y \to \bba^{n}$ associated to any fan that refines the normal fan of the Minkowski sum $\sum_{i=1}^{r}P(f_{i})$ without adding new rays is an isomorphism over the complement of $Z$. It is easy to see that condition i) of Theorem \ref{nested Khovanskii minimal exponent} holds when, for each $i\in \{1,\dots,r\}$, $f_{i}$ is \emph{convenient}, that is, $P(f_{i})$ meets each coordinate axis, or when the $P(f_{i})$'s are proportional, that is, there exist positive real numbers $c_{1},\dots,c_{r-1}$ such that $c_{i}\cdot P(f_{i})=P(f_{i+1})$ for each $i\leq r-1$. A fan $\Sigma_{0}$ as in the statement of the theorem above always exists. Namely, the normal fan of $\sum_{i=1}^{r}P(f_{i})$ refines the normal fan of $P(f_{i})$ for each $i$ (see Section \ref{Minkowski sums} for details). If we consider the polynomials 
$$
f_1=x^2+y^2+z^{2}
\ \ \ \ \text{and}
\ \ \ \
f_2 = x^6+2x^4y+x^2y^2+y^6+z^{6},
$$ 
then the closed subscheme $Z=V(f_{1},f_{2})\subset\bba^{3}$ is a complete intersection of codimension $2$ and by Theorem \ref{nested Khovanskii minimal exponent}, we have $\widetilde{\alpha}_{0}(Z)=\tfrac{5}{4}$ (see Example \ref{nested newton polygons example} for details).

The proof of Theorem \ref{Khovanskii minimal exponent} essentially reduces to the case of a Newton non-degenerate hypersurface, since the condition that the $|I|$-tuple $(f_{i})_{i\in I}$ is Khovanskii non-degenerate for all subsets $I \subset \{1,\dots,r\}$ implies that $g_{i}=f_{i}+\sum_{j\neq i}^{}z_{j}f_{j}$ satisfies a form of non-degeneracy in between Newton and global Newton non-degeneracy for each $i$, which we define in Section \ref{Non-degeneracy}. On the other hand, the proof of Theorem \ref{nested Khovanskii minimal exponent} cannot be reduced to the case of a Newton non-degenerate hypersurface, as explained in Example \ref{nested newton polygons example}, and hence requires a different approach.

To prove Theorems \ref{weighted minimal exponent} and \ref{nested Khovanskii minimal exponent}, we follow the strategy of \cite{CDM25}, which established the formula for the minimal exponent of a homogeneous complete intersection (with some transversality conditions) by proving a general upper bound for the minimal exponent in terms of the Newton polyhedra of the defining equations and a general lower bound for the minimal exponent in terms of a strong factorizing log resolution in the sense of Bravo and Villamayor. While the upper bounds in Theorems \ref{weighted minimal exponent} and \ref{nested Khovanskii minimal exponent} follow from \cite[Theorem 1.2]{CDM25}, it is unclear how to apply the lower bound in \cite[Theorem 1.3]{CDM25} to Theorems \ref{weighted minimal exponent} and \ref{nested Khovanskii minimal exponent}, since explicit strong factorizing resolutions are not known in these settings. One of our main results significantly generalizes the lower bound in \cite[Theorem 1.3]{CDM25}, which we use to prove Theorems \ref{weighted minimal exponent} and \ref{nested Khovanskii minimal exponent}. Before stating the result, we introduce the necessary notation and definitions.

For regular functions $f_1,\dots,f_r\in \os_X(X)$ on a smooth variety $X$, the $V$-filtration is a decreasing filtration $(V^{\lambda}B_{\bff})_{\lambda\in \bbq}$ on
$$
B_{\bff} = \bigoplus_{\beta \in \bbz_{\geq 0}^r} \os_X \partial_t^{\beta}\delta_{\bff},
$$
indexed by rational numbers and characterized by a few properties (for details, see Section \ref{V-filtration}). We define the minimal exponent of $f_1,\dots,f_r$ relative to a regular function $g\in \os_X(X)$ as
$$
\widetilde{\alpha}(f_1,\dots,f_r; g) = \left\{
\begin{array}{cl}
\sup\{\gamma>0 \ | \  g\delta_{\bff}\in V^{\gamma}B_{\bff}\}  & \text{if $g\delta_{\bff}\not\in V^rB_{\bff}$} \\[2mm]
\sup\{r-1+q+\gamma \ | \  g\partial_t^{\beta}\delta_{\bff} \in V^{r-1+\gamma}B_{\bff}\}  & \text{if $g\delta_{\bff}\in V^rB_{\bff}$},
\end{array}\right.
$$
where in the latter case, the supremum is over all nonnegative integers $q$ and all rational numbers $\gamma\in (0,1]$ with the property that $g\partial_t^{\beta}\delta_{\bff} \in  V^{r-1+\gamma}B_{\bff}$ for all $\beta=(\beta_1,\dots,\beta_r)\in \bbz_{\geq 0}^r$ such that $\beta_1+\cdots+\beta_r\leq q$. By definition, if $Z\subset X$ is a closed subscheme defined by a regular sequence $f_1,\dots,f_r\in\os_{X}(X)$, then $\widetilde{\alpha}(Z) = \widetilde{\alpha}(f_1,\dots,f_r;1)$. 

Let $X$ be an irreducible smooth affine variety with global coordinates $x_1,\dots,x_n\in \os_X(X)$. Given nonzero regular functions $f_1,\dots,f_r\in \os_X(X)$, we denote by $Z$ the closed subscheme of $X$ defined by the ideal generated by $f_1,\dots,f_r$. Let $\pi\colon Y\to X$ be a projective morphism that is an isomorphism over the complement of $Z$ such that $Y$ is irreducible and has quotient singularities with \'{e}tale cover $p\colon U/G \onto Y$, where $G$ is a finite group acting on a smooth affine variety $U$. Given a finite collection $\{U_i\}_{i\in I}$ of open subsets of $U$ such that each $U_i$ has global coordinates $y_1,\dots,y_n\in \os_U(U_i)$, we denote by $\pi_i$ the morphism $U_i \to X$ and by $g_i=\det ( (\frac{\partial (\pi_i^* x_p)}{\partial y_q})_{pq} )\in \os_U(U_i)$ the Jacobian of $\pi_i$. Finally, we denote the morphism $U\to U/G$ by $q$.

\begin{theorem}\label{minimal exponent lower bound}
With the above notation, the following hold:
\begin{enumerate}
\item[i)] If $Z$ is a complete intersection of codimension $r$ in $X$ and $U=\bigcup_{i\in I}U_i$, then we have
$$
\widetilde{\alpha}(Z) \geq \min_{i\in I} \widetilde{\alpha}(\pi_i^*f_1,\dots,\pi_i^*f_r; g_i).
$$
\item[ii)] If $z\in Z$ is a point such that $Z$ is a complete intersection of codimension $r$ in some open neighborhood of $z$ and $(\pi\circ p\circ q)^{-1}(z) \subset \bigcup_{i\in I} U_i$, then we have
$$
\widetilde{\alpha}_z(Z) \geq \min_{i\in I} \widetilde{\alpha}(\pi_i^*f_1,\dots,\pi_i^*f_r; g_i).
$$
\end{enumerate}
\end{theorem}

The key ingredient in the proof of Theorem \ref{minimal exponent lower bound} is provided by \cite{Lee26}, where we obtained estimates for the roots of Bernstein-Sato polynomials in terms of log resolutions by orbifolds. 

The proofs of Theorems \ref{weighted minimal exponent}, \ref{Khovanskii minimal exponent}, and \ref{nested Khovanskii minimal exponent} also rely on a careful study of the facets of the Newton polyhedron of $g=\sum_{i=1}^r z_if_i$, which we carry out in Sections \ref{The Newton polyhedron associated to several Newton polyhedra} and \ref{Computing the minimal exponent of several Newton polyhedra}. In particular, when the $f_{i}$'s are all monomials, we obtain a complete description of the facets of $g$ (see Theorem \ref{candidates result}).

\subsection{Outline of the paper}

Section \ref{background} covers various background material on $\ds$-modules associated to graph embeddings, the Hodge and $V$-filtrations on $B_{\bff}$, the minimal exponent of a local complete intersection, Newton polyhedra, normal fans, Minkowski sums, and the toric variety associated to a simplicial refinement of the positive  orthant.
In Section \ref{The Newton polyhedron associated to several Newton polyhedra} we introduce and study the Newton polyhedron associated to several Newton polyhedra, which we use in Section \ref{The minimal exponent of several Newton polyhedra} to define and study the minimal exponent of several Newton polyhedra.
In Section \ref{Computing the minimal exponent of several Newton polyhedra} we describe a method for computing the minimal exponent of several Newton polyhedra.
In Section \ref{A combinatorial upper bound for the minimal exponent} we restate the upper bound in \cite[Theorem 1.2]{CDM25} in terms of the minimal exponent of several Newton polyhedra.
In Section \ref{A general lower bound for the minimal exponent} we prove Theorem \ref{minimal exponent lower bound} and establish some basic properties of the relative minimal exponent.
In Section \ref{A combinatorial lower bound for the minimal exponent} we establish a combinatorial lower bound for the relative minimal exponent.
In Section \ref{Semi-quasihomogeneous complete intersections} we prove Theorem \ref{weighted minimal exponent}.
In Section \ref{Non-degeneracy} we recall the definitions of Newton and Khovanskii non-degeneracy and establish a basic result relating the Khovanskii non-degeneracy of the $r$-tuple $(f_{1},\dots,f_{r})$ with the Newton non-degeneracy of $g=\sum_{i=1}^{r}z_{i}f_{i}$. Finally, in Sections \ref{Khovanskii non-degenerate complete intersections} and \ref{Nested non-degenerate complete intersections} we prove Theorems \ref{Khovanskii minimal exponent} and \ref{nested Khovanskii minimal exponent}, respectively.

\subsection{Convention}

All varieties are reduced (and not necessarily irreducible) schemes of finite type over $\bbc$.

\subsection{Acknowledgements}

I am grateful to Mircea \Mustata for many valuable discussions. I would also like to thank Gary Hu for helpful feedback on an earlier version of the introduction.

\section{Background}\label{background}

\subsection{$\ds$-modules and the graph embedding}\label{graph}

Let $X$ be a smooth variety and denote its sheaf of differential operators by $\ds_X$. In this paper all $\ds$-modules will be left $\ds$-modules. For general facts about $\ds$-modules, we refer to \cite[]{HTT08}. 

Given regular functions $f_1,\dots,f_r \in \os_X(X)$, we denote the $r$-tuple $(f_1,\dots,f_r)$ by $\bff$ and consider the graph embedding 
$$
i_{\bff}\colon X \to X \times \bba^r, x \mapsto (x,\bff(x)).
$$
We denote the standard coordinates on $\bba^r$ by $t_1,\dots,t_r$ and the projection $X \times \bba^r \to X$ by $p_X$, and put
$$
\rs_X = p_{X,*}\ds_{X\times\bba^{r}} = \ds_X \langle t_1,\dots,t_r, \partial_{t_1},\dots,\partial_{t_r}\rangle.
$$
For $\beta = (\beta_1,\dots,\beta_r)\in \bbz_{\geq 0}^r$, we put $t^{\beta} = \prod_{i=1}^r t_i^{\beta_i}$ and $\partial_t^{\beta} = \prod_{i=1}^r \partial_{t_i}^{\beta_i}$. 

Given a $\ds_{X}$-module $\ms$, we consider the $\ds$-module theoretic pushforward $i_{\bff,+}\ms$, which is a $\ds_{X\times\bba^{r}}$-module. We also consider the $\rs_X$-module
$$
p_{X,*}(i_{\bff,+}\ms) = \bigoplus_{\beta\in \bbz_{\geq 0}^r} \ms\partial_t^{\beta}\delta_{\bff},
$$
where the actions of $\os_X$ and $\partial_{t_i}$ are the obvious ones, while the actions of $\theta \in \sDer_{\bbc}(\os_X)$ and $t_i$ are given by 
\begin{equation}\label{bf action}
\theta \cdot u\partial_t^{\beta}\delta_{\bff} = \theta( u)\partial_t^{\beta}\delta_{\bff} - \sum_{i=1}^r \theta(f_i) u \partial_t^{\beta+e_i}\delta_{\bff}
\text{      and      }
t_i \cdot u\partial_t^{\beta}\delta_{\bff} = f_i u \partial_t^{\beta}\delta_{\bff} - \beta_i u \partial_t^{\beta-e_i}\delta_{\bff},
\end{equation}
where $u\in\ms$ and $e_1,\dots,e_r$ is the standard basis of $\bbz^r$. 

\subsection{The Hodge and $V$-filtrations}\label{V-filtration}

In this section we recall the definition and some basic properties of the Hodge and $V$-filtrations on the $\ds$-module $B_{\bff}$ discussed in the Introduction. For details, we refer to \cite{Kas83}, \cite[Section 1]{BMS06}, and \cite[Section 2]{CDMO24}. 

Let $X$ be a smooth variety. Given regular functions $f_1,\dots,f_r \in \os_X(X)$, we denote the $r$-tuple $(f_1,\dots,f_r)$ by $\bff$ and the standard coordinates on $\bba^r$ by $t_1,\dots,t_r$, and set
$$
s = -\sum_{i=1}^r \partial_{t_i}t_i.
$$
We define
$$
B_{\bff} := p_{X,*}(i_{\bff,+}\os_X) = \bigoplus_{\beta\in \bbz_{\geq 0}^r} \os_X \partial_t^{\beta}\delta_{\bff},
$$
whose $\rs_{X}$-module structure is given by \eqref{bf action}. The \textit{Hodge filtration} $(F_pB_{\bff})_{p\in \bbz}$ on $B_{\bff}$ is the increasing filtration given by
$$
F_pB_{\bff} = \bigoplus_{|\beta|\leq p}\os_X\partial_t^{\beta} \delta_{\bff}.
$$

Now consider the decreasing filtration $(V^k\rs_X)_{k\in \bbz}$ on $\rs_X$ given by
$$
V^k\rs_X = \bigoplus_{|\alpha|-|\beta|\geq k} \ds_X t^{\alpha}\partial_t^{\beta}.
$$
The $V$-filtration on $B_{\bff}$ is a decreasing, exhaustive filtration $(V^{\lambda}B_{\bff})_{\lambda\in \bbq}$ indexed by rational numbers, which is \textit{discrete and left-continuous}\footnote{This means that there is a positive integer $\ell$ such that $V^{\lambda}B_{\bff}$ has constant value for all $\lambda$ in an interval of the form $(\frac{i-1}{\ell}, \frac{i}{\ell}]$ with $i\in \bbz$.} and satisfies the following properties:
\begin{enumerate}
\item[i)] $V^{\lambda}B_{\bff}$ is a coherent $V^0\rs_X$-submodule of $B_{\bff}$ for all $\lambda\in \bbq$.
\item[ii)] $t_i \cdot V^{\lambda}B_{\bff} \subset V^{\lambda+1}B_{\bff}$ and $\partial_{t_i} \cdot V^{\lambda}B_{\bff} \subset V^{\lambda-1}B_{\bff}$ for all $1\leq i\leq r$ and $\lambda\in \bbq$.
\item[iii)] $V^1\rs_X \cdot V^{\lambda}B_{\bff} = V^{\lambda+1}B_{\bff}$ for all $\lambda \gg 0$.
\item[iv)] For all $\lambda \in \bbq$, if we put $V^{>\lambda}B_{\bff}=\bigcup_{\lambda' > \lambda}V^{\lambda'}B_{\bff}$, then $s+\lambda$ acts nilpotently on
$$
\Gr_V^{\lambda}B_{\bff} = V^{\lambda}B_{\bff}/ V^{>\lambda}B_{\bff}.
$$
\end{enumerate}
By the theory of Kashiwara \cite{Kas83}, there exists a unique such $V$-filtration.

Given a section $u\in B_{\bff}$, the $b$-function $b_u(s)$ of $u$ is the monic generator of the ideal
$$
\{b(s)\in \bbc[s] \ | \ b(s)u \in V^1\rs_X \cdot u \}.
$$
It follows from the results in \cite{Kas83} (see also \cite{BMS06}) that $b_u(s)$ is nonzero and all its roots are rational. The $V$-filtration can then be described as
\begin{equation}\label{V roots}
V^{\lambda}B_{\bff} = \{u\in B_{\bff}\ | \ \text{all roots of $b_u(s)$ are $\leq -\lambda$}\}.
\end{equation}

For the remainder of this section, suppose that $X$ is irreducible and each $f_{i}$ is nonzero. Let $\frak{a} \subset \os_X$ be the ideal generated by $f_1,\dots,f_r$, and denote by $\js(\frak{a}^{\lambda})$ the multiplier ideal of $\frak{a}$ with exponent $\lambda$ (for the definition and basic properties of multiplier ideals, we refer to \cite[Chapter 9]{Laz04}). By \cite[Theorem 1]{BMS06}, we have
\begin{equation}\label{multiplier ideals}
\{g\in \os_X \ | \ g\delta_{\bff} \in V^{\lambda}B_{\bff}\} = \js(\frak{a}^{\lambda-\epsilon})
\end{equation}
for all $\lambda\in \bbq_{>0}$ and $0<\epsilon \ll 1$. Given $g\in \os_X(X)$, if we define
$$
\lct_g(\frak{a}) := \sup\{ \lambda >0 \ | \ g\in \js(\frak{a}^{\lambda}) \},
$$
then by \eqref{V roots} and \eqref{multiplier ideals}, we have
\begin{equation}\label{lct b}
\lct_g(\frak{a})=\min\{\lambda\in\bbq \ |\ b_{g\delta_{\bff}}(-\lambda)=0 \}.
\end{equation}
We denote the log canonical threshold of $\frak{a}$ by $\lct(\frak{a})$ (we also write this as $\lct(X,Z)$, where $Z\subset X$ is the closed subscheme defined by $\frak{a}$), which is given by $\lct(\frak{a})=\lct_{1}(\frak{a})$. When $r=1$, we write $\lct(f_{1})$ for $\lct(\frak{a})$.

\subsection{The minimal exponent of a local complete intersection}\label{the minimal exponent of a local complete intersection}

Throughout this section, let $X$ be a smooth variety. In this section we recall the definition and some properties of the minimal exponent of a local complete intersection. For details, we refer to \cite{CDMO24}. 

Given a local complete intersection $Z$ in $X$ of pure codimension $r\geq 1$, the minimal exponent $\widetilde{\alpha}(Z)$, which we write as $\widetilde{\alpha}(X,Z)$ whenever the ambient variety is not clear from the context, is defined as follows. If $Z$ is globally a complete intersection, that is, there are $f_1,\dots,f_r\in \os_X(X)$ such that $Z$ is defined by the ideal generated by $f_1,\dots,f_r$, then
$$
\widetilde{\alpha}(Z) = \left\{
\begin{array}{cl}
\sup\{\gamma>0 \ | \  \delta_{\bff}\in V^{\gamma}B_{\bff}\}  & \text{if $\delta_{\bff}\not\in V^rB_{\bff}$} \\[2mm]
\sup\{r-1+q+\gamma \ | \  F_qB_{\bff}\subset V^{r-1+\gamma}B_{\bff}\}  & \text{if $\delta_{\bff}\in V^rB_{\bff}$},
\end{array}\right.
$$
where in the latter case, the supremum is over all nonnegative integers $q$ and all rational numbers $\gamma\in (0,1]$ with the property that $F_qB_{\bff} \subset V^{r-1+\gamma}B_{\bff}$. In general, we can find open sets $U_1,\dots,U_N$ of $X$ with $Z \subset \bigcup_{i=1}^NU_i$ such that $Z\cap U_i$ is nonempty and defined in $U_i$ by an ideal generated by $r$ regular functions on $U_i$, so that each $\widetilde{\alpha}(U_i,Z\cap U_i)$ is well-defined. Then we have
$$
\widetilde{\alpha}(Z)= \min_{1\leq i\leq N} \widetilde{\alpha}(U_i,Z\cap U_i).
$$

Next, we recall the local version of the minimal exponent discussed in the Introduction. If $z\in Z$ is a point, then for every open neighborhood $U$ of $z\in X$, we have $\widetilde{\alpha}(Z \cap U) \geq \widetilde{\alpha}(Z)$, and $\widetilde{\alpha}(Z \cap U)$ is constant if $U$ is small enough. This constant value is denoted by $\widetilde{\alpha}_z(Z)=\widetilde{\alpha}_z(X,Z)$ and called the minimal exponent of $Z$ at $z$. By \cite[Remark 4.17]{CDMO24}, the set $\{ \widetilde{\alpha}_z(Z)\ |\ z\in Z\}$ is finite   and we have
$$
\widetilde{\alpha}(Z) = \min_{z\in Z} \widetilde{\alpha}_z(Z),
$$
where the minimum runs over the points $z$ of $Z$. If $Z$ is a hypersurface defined by $f$, we also write $\widetilde{\alpha}_z(f)$ for $\widetilde{\alpha}_z(Z)$. 

More generally, given a closed subscheme $Z$ in $X$ and a point $z\in Z$ such that $Z$ is a local complete intersection of codimension $r$ in some open neighborhood $U$ of $z$, the minimal exponent of $Z$ at $z$ is defined as $\widetilde{\alpha}_z(Z)=\widetilde{\alpha}_z(U,Z \cap U)$.

We conclude this section with the following lemma that describes the minimal exponent and its local version in terms of minimal exponents of hypersurfaces.

\begin{lemma}\label{minimal exponent in terms of hypersurface}
Given regular functions $f_1,\dots,f_r\in \os_X(X)$ on a smooth variety $X$, consider, for each $i\in \{1,\dots,r\}$, the regular function
$$
g_i = z_1f_1+\cdots + z_{i-1}f_{i-1} + f_i + z_{i+1}f_{i+1}+\cdots + z_rf_r \in \os_{X\times U_{i}}(X\times U_{i}),
$$
where $z_1,\dots,\widehat{z}_i,\dots,z_r$ are the standard coordinates on $U_{i}=\bba^{r-1}$. If we denote by $Z$ the closed subscheme defined by the ideal generated by $f_1,\dots,f_r$, then the following hold:
\begin{enumerate}
\item[i)] If $Z$ is a complete intersection of codimension $r$, then we have
$$
\widetilde{\alpha}(Z) = \min_{1\leq i\leq r} \widetilde{\alpha}(g_i).
$$
\item[ii)] If $z$ is a closed point of $Z$ such that $Z$ is a complete intersection of codimension $r$ in some neighborhood of $z$, then we have
$$
\widetilde{\alpha}_{z}(Z)=\min_{1\leq i\leq r}\min_{u_{i}\in U_{i}} \widetilde{\alpha}_{(z,u_{i})}(g_i),
$$
where, for each fixed $i$, the inner minimum runs over the closed points $u_{i}$ of $U_{i}$.
\end{enumerate}
\end{lemma}
\proof Let $g= \sum_{i=1}^r z_if_i\in \os_{X\times \bba^r}(X\times \bba^r)$, where $z_1,\dots,z_r$ are the standard coordinates on $\bba^r$. We regard $\os_{X\times \bba^r}(X\times \bba^r)=\os_{X}(X)[z_1,\dots,z_r]$ as a graded ring with $\os_{X}(X)$ lying in degree $0$ and $\deg z_{i}=1$ for all $i$. Then $g$ is a homogeneous element of $\os_{X\times \bba^r}(X\times \bba^r)$, and we denote by $H\subset X\times\bbp^{r-1}$ the closed subscheme defined by $g$, where $z_1,\dots,z_r$ are the standard homogeneous coordinates on $\bbp^{r-1}$. Note that $\bbp^{r-1}=\bigcup_{i =1}^{r}U_{i}$ and that the closed subscheme $H\cap (X\times U_{i})$ of $X\times U_{i}$ is defined by $g_{i}$.

We first prove part i). Assume that $Z$ is a complete intersection of codimension $r$. By \cite[Theorem 1.1]{CDMO24}, we have $\widetilde{\alpha}(Z) = \widetilde{\alpha}(g|_{X\times (\bba^r \setminus \{0\})})$. For each $i\in \{1,\dots,r\}$, if we put $D(z_{i}) = (z_i\neq 0) \subset \bba^r\setminus \{0\}$, then $\widetilde{\alpha}(g|_{X\times D(z_{i})}) = \widetilde{\alpha}(g_i)$ (as in the proof of \cite[Theorem 1.2]{CDM25}, we use \cite[Proposition 4.12]{CDMO24}, which says that the minimal exponent does not change under pullback along a smooth surjective morphism). Thus we have
$$
\widetilde{\alpha}(Z) = \widetilde{\alpha}(g|_{X\times (\bba^r \setminus \{0\})}) = \min_{1\leq i \leq r} \widetilde{\alpha}(g|_{X\times U_i})= \min_{1\leq i\leq r} \widetilde{\alpha}(g_i).
$$
In particular, part i) implies that
\begin{equation}\label{part one consequence}
\widetilde{\alpha}(Z\cap U)=\widetilde{\alpha}(H\cap (U\times \bbp^{r-1}))
\end{equation}
for all open subsets $U$ of $X$, which we will use in the proof of part ii).

We now prove part ii). Let $z$ be a closed point of $Z$ such that $Z$ is a complete intersection of codimension $r$ in some neighborhood of $z$. By shrinking $X$, we may assume that $Z$ is a complete intersection of codimension $r$ and that $\widetilde{\alpha}_{z}(Z)=\widetilde{\alpha}(Z)$. By part i), we have
$$
\widetilde{\alpha}_{z}(Z)=\widetilde{\alpha}(Z)\leq \min_{1\leq i\leq r}\min_{u_{i}\in U_{i}} \widetilde{\alpha}_{(z,u_{i})}(g_i),
$$
hence it suffices to prove the reverse inequality. Fix $i\in \{1,\dots, r\}$. For each closed point $u_{i}\in U_{i}$, let $W_{u_{i}}$ be an open neighborhood of $u_{i}$ such that $\widetilde{\alpha}(g_{i}|_{W_{u_{i}}})=\widetilde{\alpha}_{u_{i}}(g_{i})$, and set $W_{i}=\bigcup_{u_{i}\in U_{i}}W_{u_{i}}$. Then $W_{i}$ is an open neighborhood of $\{z\}\times U_{i}\subset X\times U_{i}$ such that
$$
\widetilde{\alpha}(g_{i}|_{W_{i}})=\min_{u_{i}\in U_{i}} \widetilde{\alpha}_{(z,u_{i})}(g_i).
$$
Let $W=\bigcup_{i=1}^{r}W_{i}\subset X\times\bbp^{r-1}$. Then we have
$$
\widetilde{\alpha}(W,H\cap W)=\min_{1\leq i\leq r}\widetilde{\alpha}(W_{i},H\cap W_{i})=\min_{1\leq i\leq r}\widetilde{\alpha}(g_{i}|_{W_{i}})=\min_{1\leq i\leq r}\min_{u_{i}\in U_{i}} \widetilde{\alpha}_{(z,u_{i})}(g_i),
$$
so it suffices to show that $\widetilde{\alpha}(Z)\geq \widetilde{\alpha}(W,H\cap W)$. Since $W$ is an open neighborhood of $\{z\}\times\bbp^{r-1}$, if we denote the projection $X\times \bbp^{r-1}\to X$ by $\pi$ and put $U=X\setminus \pi((X\times \bbp^{r-1})\setminus W)$, then $U$ is an open neighborhood of $z\in X$ such that $U\times\bbp^{r-1}\subset W$. Thus we obtain
$$
\widetilde{\alpha}_{z}(Z)= \widetilde{\alpha}(Z\cap U)=\widetilde{\alpha}(H\cap (U\times\bbp^{r-1}))\geq \widetilde{\alpha}(W,H\cap W),
$$
where the first equality follows from the fact that $\widetilde{\alpha}_{z}(Z)=\widetilde{\alpha}(Z)$ and the second equality follows from \eqref{part one consequence}. \qed

\subsection{Newton polyhedra}\label{Newton polyhedra}

In this section we review basic definitions and properties of Newton polyhedra, and prove a simple lemma that will be used later.

We denote the convex hull of a subset $A\subset\bbr^m$ by $\Conv(A)\subset\bbr^{m}$. Recall that the Newton polyhedron $P(A)$ of a nonempty subset $A\subset\bbz_{\geq 0}^m$ is defined as $\Conv(\bigcup_{a \in A} (a + \bbz_{\geq 0}^m))$. We say that a subset $P \subset \bbr^m$ is a Newton polyhedron if there exists a nonempty subset $A \subset \bbz_{\geq 0}^m$ such that $P=P(A)$. The Newton polyhedron of a nonzero polynomial $f \in \bbc[y_1,\dots,y_m]$ is defined as $P(f)=P(\supp f)$, where $\supp f$ is the set of elements $u=(u_{1},\dots,u_{m})\in \bbz_{\geq 0}^m$ such that the monomial $y^u=y_1^{u_1}\cdots y_m^{u_m}$ appears in $f$ with nonzero coefficient. The Newton polyhedron of a nonzero monomial ideal $\frak{a}\subset \bbc[y_1,\dots,y_m]$ is defined as $P(\frak{a})=P(\{u\in \bbz_{\geq 0}^m \ | \ y^u\in \frak{a}\})$.

Consider $\bbr^m$ and its dual $(\bbr^m)^*$. We denote the standard basis of $\bbr^m$ by $e_1,\dots,e_m$, and put
$$
(\bbr^m)^*_{\geq 0} = \{v\in (\bbr^m)^* \ | \ v(e_i)\geq 0 \text{ for $i\in \{1,\dots, m\}$} \}
$$
and
$$
(\bbr^m)^*_{> 0} = \{v\in (\bbr^m)^* \ | \ v(e_i)> 0 \text{ for $i\in \{1,\dots, m\}$} \}.
$$
For a subset $I\subset \{1,\dots, m\}$, we put
\begin{equation}\label{global real}
(\bbr^m)^*_{I} = \{v\in (\bbr^m)^*_{\geq 0} \ | \ v(e_i)> 0 \text{ for $i\notin I$} \}.
\end{equation}For an element $v\in (\bbr^m)^*_{\geq 0}$, we put
\begin{equation}\label{coordinate sum}
|v| =\sum_{i=1}^{m}v(e_{i})\in\bbr_{\geq 0}.
\end{equation}

Given an element $v\in (\bbr^m)^*_{\geq 0}$ and a nonempty subset $S \subset \bbr_{\geq 0}^m$, define
$$
d(v,S) =  \inf\{v(u) \ | \ u\in S\} \in \bbr_{\geq 0}.
$$
For every element $v\in (\bbr^m)^*_{\geq 0}$ and nonempty subset $A \subset \bbz_{\geq 0}^m$, observe that
\begin{equation}\label{vertices}
d(v,P(A)) = d(v,A).
\end{equation}

For the remainder of the section, we fix a Newton polyhedron $P\subset \bbz^{m}_{\geq 0}$. Note that if $P=P(A)$ for some subset $A\subset \bbz^{m}_{\geq 0}$, then there exists a finite subset $B\subset A$ such that $P=P(B)$. Indeed, if we identify $\bbz_{\geq 0}^m$ with the set of monomials of the polynomial ring $\bbz[y_{1},\dots,y_{m}]$, then $B$ is given by any finite set of monomial generators of the ideal generated by $A$. Together with \eqref{vertices}, this implies that the function $(\bbr^m)^*_{\geq 0} \to \bbr, v\mapsto d(v,P)$ is continuous and piecewise linear.

For an element $v\in (\bbr^m)^*_{\geq 0}$, recall that \textit{the first meet locus} of $P$ along $v$ is defined as
$$
F(v, P) = \{ u \in P \ | \ v(u) = d(v,P) \}.
$$
The set of faces of $P$ is equal to the set $\{F(v,P) \ | \ v\in (\bbr^m)^*_{\geq 0}\}$, and the set of bounded faces of $P$ is equal to the set $\{F(v,P) \ | \ v\in (\bbr^m)^*_{> 0}\}$. Recall that a vertex is a face of dimension zero, and a facet is a proper face of maximal dimension $m-1$. If $P=P(A)$ for some $A\subset \bbz^{m}_{\geq 0}$, then $A$ contains all the vertices of $P$ by \eqref{vertices}. Note that a face $Q$ of $P$ is unbounded if and only if $Q+e_i \subset Q$ for some $i$.

\begin{definition}\label{}
Given a facet $Q$ of $P$ that is not contained in any coordinate hyperplane, we define $L_Q$ as the unique element $L$ of $(\bbr^m)^*_{\geq 0}$ such that $d(L,P)=1$ and $F(L,P)=Q$. 
\end{definition}

\begin{remark}\label{numerical facet characterization}
For an element $L\in (\bbr^m)^*_{\geq 0}$, it is easy to see that $v=L_{Q}$ for some facet $Q$ of $P$ that is not contained in any coordinate hyperplane if and only if $d(L,P)=1$ and $\dim F(L,P)=m-1$.
\end{remark}

Note that $P$ has a facet not contained in any coordinate hyperplane if and only if $0\notin P$, in which case
\begin{equation}\label{half space intersection}
P = \bigcap_{Q} \{u\in \bbr^m_{\geq 0} \ | \ L_Q(u)\geq 1\},
\end{equation}
where the intersection runs over the facets $Q$ of $P$ that are not contained in any coordinate hyperplane.

\begin{remark}\label{unbounded facets}
Let $Q$ be a facet of $P$ that is not contained in any coordinate hyperplane, and fix an index $i\in \{1,\dots,m\}$. Then it is easy to see that $L_{Q}(e_{i})=0$ if and only if $Q+e_{i}\subset Q$.
\end{remark}

\begin{remark}\label{unbounded face correspondence}
Let $I\subset \{{1},\dots,{m}\}$ be a subset. We denote by $\pi\colon\bbr^m \to \bbr^{m-|I|}$ the standard projection obtained by omitting the coordinates indexed by $I$. Note that $\pi(P)\subset \bbr^{m-|I|}$ is a Newton polyhedron since $\pi(P)=P(\pi(A))$ if $P=P(A)$ for some $A \subset\bbz_{\geq 0}^m$. There is a bijection between faces $Q$ of $P$ such that $Q +e_{i}\subset Q$ for each $i\in I$ and faces $Q'$ of $\pi(P)$ given by $Q'=\pi(Q)$ with inverse $Q=\pi^{-1}(Q') \cap P$ (see \cite[Remark 3.2]{BMScomb}). Under this bijection, $Q$ is a facet that is not contained in any coordinate hyperplane if and only if $Q'$ has the same property, in which case $L_Q = \pi\circ L_{Q'}$. Thus by (\ref{half space intersection}) and Remark \ref{unbounded facets}, if $0\notin \pi(P)$, then we have
$$
\pi^{-1}(\pi(P))\cap \bbr^m_{\geq 0} = \bigcap_Q \{u\in \bbr^m_{\geq 0} \ | \ L_Q(u)\geq 1\},
$$
where the intersection runs over the facets $Q$ of $P$ that are not contained in any coordinate hyperplane such that $L_{Q}(e_{i})=0$ for all $i\in I$.
\end{remark}

We now prove the following characterization of the linear functions associated to facets of Newton polyhedra.

\begin{lemma}\label{facet uniqueness}
Given a subset $A\subset \bbz_{\geq 0}^m$, let $Q$ be a facet of $P(A)$ that is not contained in any coordinate hyperplane. Then $L_Q$ is the unique element $L$ of $(\bbr^m)^*_{\geq 0}$ such that the following properties hold:
\begin{enumerate}
\item[a)] $\{u\in A \ | \ L(u)=1\} =Q \cap A$.
\item[b)] For $i\in \{1,\dots,m\}$, we have $L(e_i)= 0$ if and only if $Q+e_i\subset Q$.
\end{enumerate}
\end{lemma}
\proof It is clear that $L_Q$ satisfies these two properties, so let $L$ be an element of $(\bbr^m)^*_{\geq 0}$ satisfying properties a) and b). We need to show that $L=L_Q$. First assume that $Q$ is bounded. Let $H\subset \bbr^m$ denote the affine span of $Q$, which is an affine hyperplane since $Q$ is a facet. Note that $H$ does not contain the origin since $L_Q(H)=L_Q(Q)=1$, hence the linear span of $H$ is equal to $\bbr^m$. Since $Q$ is bounded, $H$ is the affine span of the vertices of $Q$. By \eqref{vertices}, the vertices of $Q$ are contained in $Q \cap A$. Thus the span of $Q \cap A$ is equal to $\bbr^m$. Since $L(Q \cap A)=L_Q(Q \cap A)$ by property a), we have $L=L_Q$. 

Now assume $Q$ is unbounded. Let $I = \{i\in\{1,\dots,m\} \ | \ Q+e_i \subset Q\}$, and consider the standard projection $\pi\colon \bbr^m \to \bbr^{m-|I|}$ obtained by omitting the coordinates indexed by $I$. Write $A'=\pi(A)$ and $Q'=\pi(Q)$. By Remark \ref{unbounded face correspondence}, $Q'$ is a facet of $P(A')$ that is not contained in any coordinate hyperplane such that $L_Q = L_{Q'} \circ \pi$. Note that $Q'$ is bounded because $L_{Q'}(\pi(e_i))>0$ for $i\not\in I$. Consider the element $L' \in (\bbr^{m-|I|})^*_{\geq 0}$ given by $L'(\pi(e_i))=L(e_i)>0$ for $i\not\in I$. By property b), we have $L=L'\circ \pi$, which implies that
$$
\{u\in A' \ | \ L'(u)=1\} = \pi(\{u\in A \ | \ L(u)=1\} ) = \pi(Q \cap A) 
$$
$$
= \pi(\pi^{-1}(Q') \cap P \cap A) = \pi(\pi^{-1}(Q')  \cap A) = Q' \cap A',
$$
where the second equality follows from property a). By the bounded case, we then obtain $L'=L_{Q'}$, hence $L=L'\circ \pi = L_{Q'} \circ \pi = L_Q$. \qed

\subsection{Cones and fans}

In this section, we recall some basic definitions and properties of cones and fans. All cones in this paper are strongly convex rational polyhedral cones. For general facts about cones and fans, we refer to \cite{CLS11}. 

Let $e_1,\dots,e_m$ denote the standard basis of $\bbr^m$. We consider cones and fans in $(\bbr^m)^*$ with respect to the lattice $\{v\in (\bbr^m)^* \ | \ v(e_1),\dots,v(e_m)\in \bbz \}$. Given a cone $\sigma$ in $(\bbr^m)^*$, recall that the relative interior $\relint\sigma$ of $\sigma$ is the interior of $\sigma$ in its $\bbr$-span. 

Recall that rays are one-dimensional cones. The primitive generator of a ray $\rho$ in $(\bbr^m)^*$ is the unique vector $v\in \rho$ such that $v(e_1),\dots,v(e_m)$ are integers generating the unit ideal in $\bbz$. We define the set of primitive ray generators of a fan $\Sigma$ (resp. of a cone $\sigma$) to be the set of primitive generators of the rays of $\Sigma$ (resp. of the one-dimensional faces of $\sigma$). If $\sigma$ is a cone with primitive ray generators $v_1,\dots,v_l$, then by \cite[Proposition 2.44]{RW98}, we have
\begin{equation}\label{relative interior}
\relint \sigma = \bigg\{\sum_{i=1}^l a_i v_i \ \big| \ a_1,\dots,a_l\in \bbr_{>0}\bigg\}.
\end{equation}

\subsection{Normal fans}\label{Normal fans}

In this section, we recall some basic properties of the normal fan of a Newton polyhedron. For general facts about normal fans, we refer to \cite[\S 7.1]{CLS11}.

Given a Newton polyhedron $P \subset \bbr^n$, let $\Sigma_P$ denote its normal fan, whose support is given by $|\Sigma_P|=(\bbr^n)^*_{\geq 0}$. The cones of $\Sigma_P$ are in natural bijection with the faces of $P$. Given a face $Q$ of $P$, we denote the corresponding cone of $\Sigma_P$ under this bijection by $\sigma_Q$, whose relative interior is given by
\begin{equation}\label{relative interior characterization}
\relint \sigma_Q = \{v\in (\bbr^m)^*_{\geq 0} \ | \ F(v,P)=Q\}.
\end{equation}
Since $d(v,P)$ is continuous in $v$, we have
\begin{equation}\label{closure}
d(v,P) = v(u) \ \ \ \text{for $v\in \sigma_Q$ and $u\in Q$},
\end{equation}
or equivalently, $Q \subset F(v,P)$ for $v\in \sigma_Q$.

Let $e_1,\dots,e_n$ denote the standard basis of $\bbr^n$ and let $e^{*}_1,\dots,e^{*}_n\in (\bbr^n)^*$ denote the corresponding dual basis. The set of rays of $\Sigma_P$ is given by
\begin{equation}\label{rays}
\{\bbr_{\geq 0} e_i^* \ | \ 1\leq i\leq n\} \cup \bigg\{\bbr_{\geq 0}L_Q \ \big| \ 
\begin{array}{l} \text{$Q$ is a facet of $P$ that is not contained} \\
\text{in any coordinate hyperplane}
\end{array}
\bigg\},
\end{equation}
where $\bbr_{\geq 0} v = \{tv \ | \ t\in \bbr_{\geq 0}\}$ for $v\in (\bbr^n)^*_{\geq 0}$. 

\begin{lemma}\label{closure lemma}
Let $P\subset \bbr^{n}$ be a Newton polyhedron and let $\Sigma$ be a fan that refines the normal fan of $P$. Given a cone $\sigma\in \Sigma$, if we denote the primitive ray generators of $\sigma$ by $v_{1},\dots,v_{l}$, then the following hold:
\begin{enumerate}
\item[i)] For all real numbers $\lambda_{1},\dots,\lambda_{l}\geq 0$, we have $d(\sum_{j=1}^{l}\lambda_{j}v_{j},P)=\sum_{j=1}^{l}\lambda_{j}d(v_{j},P)$.
\item[ii)] For all $v\in\relint \sigma$, we have $F(v,P)=\bigcap_{j=1}^{l} F(v_j,P)$.
\end{enumerate}
\end{lemma}
\proof Let $\Sigma_{P}$ denote the normal fan of $P$, and let $Q$ denote the face of $P$ such that $\sigma_{Q}$ is the smallest cone of $\Sigma_{P}$ containing $\sigma$, which implies that $\relint \sigma\subset \relint \sigma_{Q}$. To prove part i), observe that if we choose an element $u\in Q$, then by \eqref{closure}, we have 
$$
d\bigg(\sum_{j=1}^{l}\lambda_{j}v_{j},P\bigg)=\sum_{j=1}^{l}\lambda_{j}v_{j}(u)=\sum_{j=1}^{l}\lambda_{j}d(v_{j},P).
$$  
For part ii), note that $Q=F(v,P)$ by \eqref{relative interior characterization} and $Q \subset \bigcap_{j=1}^l F(v_j, P)$ by \eqref{closure}, so it suffices to prove the reverse inclusion. Let $u\in \bigcap_{j=1}^l F(v_j, P)$, and let us write $v=\sum_{i=1}^l \lambda_i v_i$ for some $\lambda_1,\dots,\lambda_{l}\geq 0$. Then we have
$$
v(u)=\sum_{j=1}^{l}\lambda_{j}v_{j}(u)=\sum_{j=1}^{l}\lambda_{j}d(v_{j},P)=d(v,P),
$$
where the last equality follows from part i), which implies that
$\bigcap_{i=1}^l F(v_i, P)\subset F(v,P)$. \qed

\subsection{Minkowski sums}\label{Minkowski sums}

In this section, we recall some basic properties of the normal fan of the Minkowski sum of several Newton polyhedra. Recall that the Minkowski sum of finitely many subsets $S_1,\dots,S_r \subset \bbr^n$ is defined as 
$$
\sum_{i=1}^r S_i = \{u_1+\cdots +u_r \ | \ u_i\in S_i\} \subset \bbr^n.
$$

Let $P_1,\dots,P_r \subset \bbr^n$ be Newton polyhedra. We consider the Minkowski sum $P=\sum_{i=1}^r P_i$, which is a Newton polyhedron since $P=P(\sum_{i=1}^r A_i)$ if $P_i=P(A_i)$ with $A_i \subset \bbz_{\geq 0}^n$. We denote the normal fans of $P_1,\dots,P_r,P$ by $\Sigma_{1},\dots,\Sigma_{r},\Sigma_{P}$, respectively. For $v\in (\bbr^n)^*_{\geq 0}$, note that $d(v,P) = \sum_{i=1}^r d(v,P_i)$ and $F(v,P) = \sum_{i=1}^r F(v,P_i)$. If $v,w\in (\bbr^n)^*_{\geq 0}$ are elements such that $F(v,P)=F(w,P)$, then it it easy to see that $F(v,P_i)=F(w,P_i)$ for each $i$ (see also \cite[Lemma 2.7]{Dam89}). Thus, given any face $Q$ of $P$, if we set $Q_i=F(v,P_i)$ for some $v\in \relint \sigma_Q$, then $Q_i$ is independent of the choice of $v$, hence by \eqref{relative interior characterization}, we have
\begin{equation}\label{cones normal fan sum}
\relint \sigma_Q = \bigcap_{i=1}^r\relint\sigma_{Q_i}.
\end{equation}
In particular, $\Sigma_{P}$ is a common refinement of $\Sigma_{1},\dots,\Sigma_{r}$. See also \cite[Remark 4]{ZG09}.

\subsection{Simplicial toric resolutions}\label{simplicial toric resolutions}

In this section, we recall some basic properties of the toric variety associated to a simplicial refinement of the positive orthant. For details on toric varieties, we refer to \cite{CLS11}.

Consider $M=\bbz^n$ and its dual $N=\Hom_{\bbz}(M,\bbz)$. We put $M_{\bbr} =M\otimes_{\bbz} \bbr=\bbr^n$ and $N_{\bbr}=N\otimes_{\bbz} \bbr=(\bbr^n)^*$. Let $e_1,\dots,e_n$ denote the standard basis of $M$ and let $e^{*}_1,\dots,e^{*}_n\in N$ denote the corresponding dual basis. Let $\bbc[\bbz_{\geq 0}^n]=\bbc[x_1,\dots,x_n]$, where $u\in \bbz_{\geq 0}^n$ corresponds to the monomial $x^u = \prod_{i=1}^n x_i^{u_i}$.

We fix a simplicial fan $\Sigma$ with support $(\bbr^n)^*_{\geq 0} \subset N_{\bbr}$, and let $Y$ be the toric variety associated to $\Sigma$. Since $|\Sigma| = (\bbr^n)^*_{\geq 0}$, we have the induced toric morphism
$$
\pi\colon Y \to \bba^n=\Spec \bbc[x_1,\dots,x_n],
$$
which is proper and birational. By \cite[Lemma 3.3.21]{CLS11}, $\pi$ is an isomorphism over the complement of
\begin{equation}\label{isom locus}
\bigcup_v V(x_i \ | \ \text{$i\in \{1,\dots,n\}$ with $v(e_i)>0$}),
\end{equation}
where $v$ runs over the primitive ray generators of $\Sigma$ other than $e_1^*,\dots,e_n^*$. The following lemma is an immediate consequence of \eqref{isom locus}.

\begin{lemma}\label{isom lemma}
With the above notation, if $f\in\bbc[x_{1},\dots,x_{n}]$ is a nonzero polynomial such that $d(v,P(f))>0$ for all primitive ray generators $v$ of $\Sigma$ other than $e_1^*,\dots,e_n^*$, then $\pi$ is an isomorphism over the complement of $V(f)\subset \bba^{n}$.
\end{lemma}

For the remainder of this section, unless stated otherwise, we fix an $n$-dimensional cone $\sigma \in \Sigma$ with primitive ray generators $v_1,\dots,v_n$. Let 
$$
U_{\sigma} = \Spec \bbc[\sigma^{\vee} \cap M].
$$
We denote by $v_1^*,\dots,v_n^*\in M_{\bbr}$ the dual basis corresponding to the basis $v_1,\dots,v_n\in N_{\bbr}$. We set $N_{\sigma} =\sum_{i=1}^n \bbz v_i \subset N$ and
$$
M_{\sigma} = \Hom_{\bbz}(N_{\sigma},\bbz)= \sum_{i=1}^n \bbz v_i^* \subset M_{\bbr}.
$$
 Note that $\sigma^{\vee} \cap M_{\sigma} = \sum_{i=1}^n \bbz_{\geq 0} v_i^*$. Let $\bbc[\sigma^{\vee} \cap M_{\sigma}]=\bbc[y_1,\dots,y_n]$, where $\sum_{i=1}^n u_i v_i^*$ with $u_i\in \bbz_{\geq 0}$ corresponds to the monomial $\prod_{i=1}^ny_i^{u_i}$, and put
$$
U'_{\sigma} = \Spec \bbc[\sigma^{\vee} \cap M_{\sigma}]=\Spec \bbc[y_1,\dots,y_n].
$$
By \cite[Proposition 1.3.18]{CLS11}, the finite group $G_{\sigma} = N/N_{\sigma}$ acts on $U'_{\sigma}$ and we have
$$
U'_{\sigma}/G_{\sigma} = U_{\sigma}.
$$
We denote the composition $U'_{\sigma} \to U_{\sigma} \to \bba^n$ by $\pi_{\sigma}$. Since $u =\sum_{j=1}^n v_j(u) v^*_j$ for $u\in M$, we have
\begin{equation}\label{toric monomial}
\pi_{\sigma}^*(x^u )= \prod_{j=1}^{n}y_j^{v_j(u)}.
\end{equation}
It is easy to see that
\begin{equation}\label{jacobian}
\pi_{\sigma}^*(dx_1\wedge \cdots \wedge dx_n) =\det(v_j(e_i))_{ij} \prod_{j=1}^n y_j^{|v_j|-1}dy_1\wedge \cdots \wedge dy_n,
\end{equation}
where $|v|$ is as defined in \eqref{coordinate sum} and $\det(v_j(e_i))_{ij}$ is a nonzero integer because $v_1,\dots,v_n$ are linearly independent. We record the following technical lemma.

\begin{lemma}\label{Newton pull}
With the above notation, let $f=\sum_u c_{u}x^{u}\in \bbc[x_{1},\dots,x_{n}]$ be a nonzero polynomial, let $\Sigma$ be a refinement of the normal fan of $P(f)$, and let $l\leq n$ be a nonnegative integer. If we put
$$
\widetilde{f}(y_1,\dots,y_n) =  \sum_{u\in \bbz^{n}_{\geq 0}} c_{u} \prod_{j>l} y_j^{v_j(u)}  \prod_{j=1}^l y_j^{v_j(u)-d(v_j, P(f))}\in \bbc[y_1,\dots,y_n]
$$
and denote by $\tau$ the face of $\sigma$ with primitive ray generators $v_{1},\dots,v_{l}$, then for every element $v\in \relint\tau$, we have
$$
\pi^{*}_{\sigma}\bigg(\sum_{u\in F(v,P(f))} c_{u}x^{u}\bigg)=\widetilde{f}(0,\dots,0, y_{l+1},\dots,y_{n})\prod_{j=1}^l y_j^{d(v_j, P(f))}.
$$
\end{lemma}
\proof The proof is essentially the same as the proofs for the corresponding special cases in \cite[Lemma 2.13]{Var77}, \cite[Lemma 4]{Den95}, and the equation at the beginning of page 24 of \cite[]{Oka90}. We include the details for completeness. Let $\Sigma_{f}$ denote the normal fan of $P(f)$. By \eqref{toric monomial}, we have
$$
\pi^{*}_{\sigma}\bigg(\sum_{u\in F(v,P(f))} c_{u}x^{u}\bigg)=\prod_{j=1}^l y_j^{d(v_j, P(f))}\sum_{u\in F(v,P(f))} c_{u} \prod_{j>l} y_j^{v_j(u)}\prod_{j=1}^l y_j^{v_j(u)-d(v_j, P(f))}.
$$
Its suffices to show that for an element $u\in P(f)$, we have $u\in F(v,P(f))$ if and only if $v_j(u)=d(v_j, P(f))$ for all $j\in \{1,\dots,l\}$, or equivalently, that $F(v,P(f))=\bigcap_{j=1}^{l} F(v_j,P(f))$. This follows from Lemma \ref{closure lemma}. \qed

For a face $\tau \subset \sigma$, if $\{v_i\}_{i\in I}\subset\{v_{1},\dots,v_{n}\}$ is the set of primitive ray generators of $\tau$, then we write
$$
O(\tau) = \Spec \bbc[y_1,\dots,y_n][(\textstyle\prod_{i\not\in I} y_i)^{-1}]/(y_i \ | \ i\in I)
$$
for the torus orbit in $U'_{\sigma}$ corresponding to $\tau$, so that
$$
U'_{\sigma} = \coprod_{\tau \subset \sigma} O(\tau),
$$
where the disjoint union runs over the faces $\tau$ of $\sigma$.

\begin{lemma}\label{exceptional}
For every subset $I\subset \{1,\dots,n\}$, we have
$$
\pi_{\sigma}^{-1}(V(x_{i} \ | \ i\notin I)) = \coprod_{\relint(\tau)\cap  (\bbr^n)^*_{I}\neq \emptyset}O(\tau),
$$
where the disjoint union runs over the faces $\tau$ of $\sigma$ such that $\relint(\tau)\cap  (\bbr^n)^*_{I}\neq \emptyset$, with $(\bbr^n)^*_{I}$ as defined in \eqref{global real}. 
\end{lemma}
\proof Note that $(\bbr^n)^*_{I}$ is the union of $\relint(\gamma)$
where $\gamma$ runs over the faces of the cone $(\bbr^n)^*_{\geq 0}$ such that $\{ e^{*}_{i}\ |\ i\notin I \}$ is contained in the set of primitive ray generators of $\gamma$. Thus, for a face $\tau$ of $\sigma$, if $\gamma$ is the smallest face of $(\bbr^n)^*_{\geq 0}$ containing $\tau$, then we have $\relint(\tau)\cap  (\bbr^n)^*_{I}\neq \emptyset$ if and only if $\{ e^{*}_{i}\ |\ i\notin I \}$ is contained in the set of primitive ray generators of $\gamma$. The lemma then follows from \cite[Lemma 3.3.21]{CLS11}, which implies that $\pi_{\sigma}^{-1}(V(x_{i} \ | \ i\notin I))$ is the union of orbits $O(\tau)$ where $\tau$ runs over the faces of $\sigma$ such that, if $\gamma$ is the smallest face of $(\bbr^n)^*_{\geq 0}$ containing $\tau$, then $\{ e^{*}_{i}\ |\ i\notin I \}$ is contained in the set of primitive ray generators of $\gamma$. \qed

We now allow $\sigma$ to vary. Recall that
$$
Y = \bigcup_{\sigma} U_{\sigma},
$$
where $\sigma$ runs over the $n$-dimensional cones of $\Sigma$. If we put $U = \coprod_{\sigma} U'_{\sigma}$ and $G = \prod_{\sigma} G_{\sigma}$, where the disjoint union and the product both run over the $n$-dimensional cones $\sigma$ of $\Sigma$, then $U/G = \coprod_{\sigma} U_{\sigma} \onto Y$ is an \'{e}tale cover. For later use, we denote the morphisms $U\to U/G$ and $U/G\to Y$ by $q$ and $p$, respectively.

\section{The Newton polyhedron associated to several Newton polyhedra}\label{The Newton polyhedron associated to several Newton polyhedra}

Throughout this section, let $A_1,\dots,A_r$ be nonempty subsets of $\bbz_{\geq 0}^n$. We set $P_i = P(A_i)$ for each $i$. We denote the standard bases of $\bbr^r$, $\bbr^n$, and $\bbr^{r+n}$ by $\{e'_1,\dots,e'_r\}$, $\{e_1,\dots,e_n\}$, and $\{z_1,\dots,z_r,x_1,\dots,x_n\}$, respectively. 
For $1\leq i\leq r$, we denote by $\pi_i\colon\bbr^{r+n} \to \bbr^{r-1+n}$ the standard projection obtained by omitting the $z_i$-coordinate. We set 
$$
A = \bigcup_{i=1}^r (e'_i\times A_i) \subset \bbz_{\geq 0}^{r+n}.
$$
We define the Newton polyhedron $G=G(P_1,\dots,P_r)\subset\bbr^{r+n}$ associated to the Newton polyhedra $P_1,\dots,P_r\subset\bbr^{n}$ as
$$
G=\Conv\bigg(\bigcup_{i=1}^r ((e'_i \times P_i) + \bbr_{\geq 0}^{r+n})\bigg),
$$ 
which is indeed a Newton polyhedron since $G=P(A)$. The goal of this section is to prove certain properties of the linear functionals $L_{Q}$ for all facets $Q$ of $G$ not contained in any coordinate hyperplane. 

We begin by considering what we will call the trivial facet. Let $\pi\colon\bbr^{r+n} \to \bbr^{r}$ denote the standard projection onto the first $r$ coordinates. Note that $\pi(G)=P(\{e'_{1},\dots,e'_{r}\})$ has a unique facet not contained in any coordinate hyperplane, hence by Remark \ref{unbounded face correspondence}, $G$ has a unique facet $Q$ not contained in any coordinate hyperplane such that $Q+x_{j}\subset Q$ for all $1\leq j\leq n$.

\begin{definition}\label{}
We define the trivial facet of $G$ as the unique facet $Q$ of $G$ not contained in any coordinate hyperplane such that $Q+x_{j}\subset Q$ for all $1\leq j\leq n$. We refer to all other facets of $G$ as nontrivial.
\end{definition}

Note that $G$ has nontrivial facets not contained in any coordinate hyperplane if and only if $0\notin P_{i}$ for some $i$. Indeed, if $0\in P_{i}$ for all $i$, then $G=P(\{z_{1},\dots,z_{r}\})$. Conversely, suppose that $0\notin P_{i}$ for some $i$. Then $0\notin \pi_{i}(G)$, so $\pi_{i}(G)$ has a facet $Q'$ not contained in any coordinated hyperplane. Thus by Remark \ref{unbounded face correspondence}, $Q=\pi_{i}^{-1}(Q')\cap G$ is a facet of $G$ not contained in any coordinated hyperplane that is nontrivial since $Q+z_{i}\subset Q$.

Before we state the main result of this section, we introduce the following definition.

\begin{definition}\label{linear functional}
Given $v\in (\bbr^n)^*_{\geq 0}$, we define $L_{v,P_1,\dots,P_r}$ as the element of $(\bbr^{r+n})^*_{\geq 0}$ given by $L_{v,P_1,\dots,P_r}(x_j)=v(e_j)$ for $1\leq j\leq n$ and 
$$
L_{v,P_1,\dots,P_r}(z_i) = 
\begin{cases}
1-d(v,P_i) & \text{if $d(v,P_i)\leq 1$} \\
0 & \text{if $d(v,P_i)>1$}
\end{cases}
$$
for $1\leq i\leq r$. 
\end{definition}

For example, if we denote the trivial facet of $G$ by $Q_0$, then $L_{Q_0}=L_{0,P_1,\dots,P_r}$. Note that for all $v\in (\bbr^n)^*_{\geq 0}$, we have $d(L_{v,P_1,\dots,P_r},G)\geq 1$. 

\begin{theorem}\label{structure of facets}
Let $Q$ be a nontrivial facet of $G$ that is not contained in any coordinate hyperplane. If we let $v\in (\bbr^n)^*_{\geq 0}$ be the element given by $v(e_j)=L_{Q}(x_j)$ for $1\leq j\leq n$, then $L_Q=L_{v,P_1,\dots,P_r}$ and $d(v,P_i)=1$ for some $i$.
\end{theorem}

Before proving the theorem, we make the following remark, which will be used in the proof.

\begin{remark}\label{face example}
Given $L\in (\bbr^{r+n})^*_{\geq 0}$ such that $d(L,G)\geq 1$, if we let $v\in (\bbr^n)^*_{\geq 0}$ be the element given by $v(e_j)=L(x_j)$ for $j\leq n$, then $L(z_i)+d(v,P_i)=d(L,e'_i\times P_i)\geq 1$ for $i\leq r$ and
$$
\{a\in A \ | \ L(a)=1\} = \bigcup_{\substack{ 1\leq i\leq r, \\ L(z_i) +d(v,P_i)= 1 }} (e'_i \times (F(v,P_i)\cap A_i) ).
$$
\end{remark}

\noindent{\it Proof of Theorem \ref{structure of facets}}. Throughout this proof, we write $L_v = L_{v,P_1,\dots,P_r}$ for $v\in (\bbr^n)^*_{\geq 0}$. We denote the dual basis of $z_1,\dots,z_r,x_1,\dots,x_n\in \bbr^{r+n}$ by $z_1^*,\dots,z_r^*,x_1^*,\dots,x_n^*\in (\bbr^{r+n})^*$. Given $L \in (\bbr^{n+r})^*_{\geq 0}$, we define the following subset of $\{z_1,\dots,z_r,x_1,\dots,x_n\}$:
$$
\supp L = \{z_i \ | \ L(z_i)>0\} \cup \{x_j \ | \ L(x_j)>0\}.
$$
By Lemma \ref{facet uniqueness}, $L=L_Q$ if and only if $\{a\in A \ | \ L(a)=1\} = Q \cap A$ and $\supp L = \supp L_Q$. Note that by Remark \ref{face example}, we have $L_Q(z_i)+d(v,P_i)\geq 1$ for $i\leq r$.

We first show that $L_Q=L_v$. For the sake of contradiction, assume that $L_Q\neq L_v$. Then there exists $i$ such that either both $d(v,P_i)>1$ and $L_Q(z_i)>0$ hold, or both $d(v,P_i)\leq 1$ and $L_Q(z_i)>1-d(v,P_i)$ hold. In either case,  if we put $L = L_Q + t z_i^* \in (\bbr^{n+r})^*_{\geq 0}$ for any $t \in \bbr_{>0}$, then we have $\{a\in A \ | \ L(a)=1\} = Q \cap A$ by Remark \ref{face example} and $\supp L = \supp L_Q$, hence $L_Q=L$, which is a contradiction.

We are done once we show that $d(v,P_i)=1$ for some $i$. For the sake of contradiction, assume that $d(v,P_i)\neq 1$ for every $1\leq i\leq r$. Since $Q$ is a nontrivial facet, we have $L_Q\neq L_0$, so that $v\neq 0$. If we put $L=L_{(1+\epsilon)v}$ for $\epsilon >0$ sufficiently small, which is not equal to $L_{Q}=L_v$ since $v\neq 0$, then we have $\{a\in A \ | \ L(a)=1\} = Q \cap A$ by Remark \ref{face example} and $\supp L = \supp L_Q$, hence $L_Q=L$, which is a contradiction.  \qed

By combining Theorem \ref{structure of facets} with Remark \ref{unbounded face correspondence}, we obtain the following corollary.

\begin{corollary}\label{rational poly intersection}
Suppose that $0\notin P_{i}$ for some $i$. Then we have
$$
\bigcap_{i=1}^r \pi_i^{-1}(\pi_i(G)) \cap \bbr^{r+n}_{\geq 0}=  \bigcap_v \{u\in \bbr^{r+n}_{\geq 0} \ | \ L_{v,P_1,\dots,P_r}(u)\geq 1\}=  \bigcap_Q \{u\in \bbr^{r+n}_{\geq 0} \ | \ L_{Q}(u)\geq 1\},
$$
where the second intersection runs over the elements $v\in (\bbr^n)^*_{\geq 0}$ such that $d(v,P_i)=1$ for some $i$, and the third intersection runs over the nontrivial facets  $Q$ of $G$ that are not contained in any coordinate hyperplane. Moreover, we have
$$
G=\{ u\in\bbr^{r +n}_{\geq 0}\ |\ L_{0,P_1,\dots,P_r}(u)\geq 1 \}\cap\bigcap_{i=1}^r \pi_i^{-1}(\pi_i(G))\cap \bbr^{r+n}_{\geq 0}. 
$$
\end{corollary}

\begin{example}\label{}
Let $r=2$ and $n=1$. Consider the subsets $A_{1}=\{2\}$ and $A_{2}=\{3\}$ of $\bbz_{\geq 0}$, and set $P_{i}=P(A_{i})$ for each $i$. Let $G\subset \bbr^{2+1}$ be the Newton polyhedron associated to $P_{1}$ and $P_{2}$, and note that $G=P(\{(1,0,2),(0,1,3)\})$. In the diagram below, we shaded the facets of $G$ that are not contained in any coordinate hyperplane, where we used a darker shade for the trivial facet.
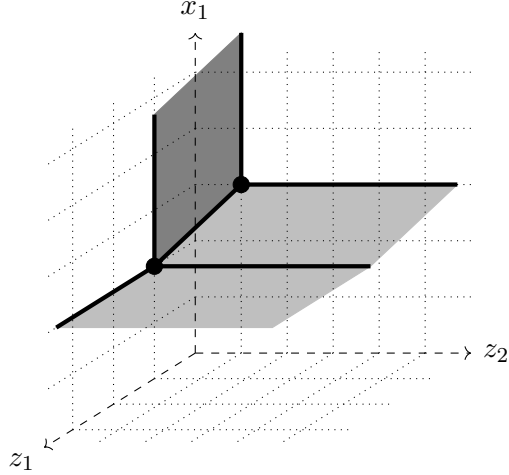
\[
\begin{tikzpicture}[
    scale=0.675,
    x={(-0.4cm,-0.25cm)},
    y={(0.9cm,0cm)},
    z={(0cm,1.1cm)}
]


\filldraw[lightgray]
    (2,0,2) --
    (0,1,3) --
    (0,5.7,3) --
    (2,4.7,2) ; 
\filldraw[gray]
    (2,0,2) --
    (0,1,3) --
    (0,1,5.7) --
    (2,0,4.7); 
\filldraw[lightgray]
    (2,0,2) --
    (6.8,0,2) --
    (6.8,4.7,2) --
    (2,4.7,2);

\draw[->,dashed]
    (0,0,0) -- (7.4,0,0); 
\draw[->,dashed]
    (0,0,0) -- (0,6,0); 
\draw[->,dashed]
    (0,0,0) -- (0,0,5.7); 

\node at (8.5,0,0) {$z_{1}$};
\node at (0,6.55,0) {$z_{2}$};
\node at (0,0,6.1) {$x_{1}$};


\foreach \y in {1,...,5}
    \draw[dotted]
        (0,\y,0) -- (0,\y,5.4);

\foreach \z in {1,...,5}
    \draw[dotted]
        (0,0,\z) -- (0,6,\z);

\foreach \y in {1,...,5}
    \draw[dotted]
        (0,\y,0) -- (7,\y,0);

\foreach \x in {2,4,6}
    \draw[dotted]
        (\x,0,0) -- (\x,6,0);

\foreach \x in {2,4,6}
    \draw[dotted]
        (\x,0,0) -- (\x,0,5.4);

\foreach \z in {1,...,5}
    \draw[dotted]
        (0,0,\z) -- (7.3,0,\z);


\draw[ultra thick]
    (2,0,2) -- (0,1,3);

\draw[ultra thick]
    (2,0,2) -- (6.8,0,2);

\draw[ultra thick]
    (2,0,2) -- (2,4.7,2);

\draw[ultra thick]
    (0,1,3) -- (0,5.7,3);

\draw[ultra thick]
    (0,1,3) -- (0,1,5.7);

\draw[ultra thick]
    (2,0,2) -- (2,0,4.7);


\filldraw[black]
    (2,0,2) circle (4.5pt);

\filldraw[black]
    (0,1,3) circle (4.5pt);

\end{tikzpicture}
\]
If we denote the two nontrivial facets that are not contained in any coordinate hyperplane by $Q_{1}$ and $Q_{2}$, with $Q_{1}$ parallel to the $z_{1}z_{2}$-plane, then
$L_{Q_{1}}=L_{\frac{1}{2},P_{1},P_{2}}$, $L_{Q_{2}}=L_{\frac{1}{3},P_{1},P_{2}}$, and $d(\frac{1}{2},P_{1})=d(\frac{1}{3},P_{2})=1$.
\end{example}

\section{The minimal exponent of several Newton polyhedra}\label{The minimal exponent of several Newton polyhedra}

In this section we define and study the minimal exponent of several Newton polyhedra. Throughout this section, we fix a positive integer $r$ and a nonnegative integer $n$. We denote the standard basis of $\bbr^{r+n}$ by $\{z_1,\dots,z_r,x_1,\dots,x_n\}$, and for $1\leq i\leq r$, we denote by $\pi_i\colon\bbr^{r+n} \to \bbr^{r-1+n}$ the standard projection obtained by omitting the $z_i$-coordinate. For $m \geq 0$, we denote the $m$-dimensional all ones vector by $\bfo_m=(1,\dots,1)\in \bbz^m$ (with the convention that $\bfo_0=0$). Note that $|v| = v(\bfo_m)$ for $v\in (\bbr^m)^*_{\geq 0}$, where $|v|$ is as defined in \eqref{coordinate sum}.

Motivated by the formula for the minimal exponent of a Newton non-degenerate hypersurface (see Remark \ref{Newton minimal exponent} below), we make the following definition for the minimal exponent of a single Newton polyhedron.

\begin{definition}\label{minimal exponent single newton poly}
Given a Newton polyhedron $P \subset \bbr^m$ and an element $u\in \bbz_{\geq 0}^m$, we define the minimal exponent of $P$ relative to $u$ as
$$
\widetilde{\alpha}(P; u) :=  \sup\{t\in \bbr_{>0} \ | \ t^{-1}(\bfo_m+u) \in P\} \in \bbq_{>0}\cup\hspace{0.15em} \{\infty\}.
$$
We then define the minimal exponent of $P$ as $\widetilde{\alpha}(P) := \widetilde{\alpha}(P; 0)$.
\end{definition}

\begin{remark}\label{}
With the notation of Definition \ref{minimal exponent single newton poly}, we have $\widetilde{\alpha}(P; u)<\infty$ if and only if $0\notin  P$.
\end{remark}

\begin{remark}\label{Newton minimal exponent}
Let $f\in\bbc[x_{1},\dots,x_{n}]$ be a polynomial such that $f\in(x_{1},\dots,x_{n})^{2}$. If $f$ is Newton non-degenerate (see Definition \ref{Newton}), then we have
$$
\widetilde{\alpha}_{0}(f)=\widetilde{\alpha}(P(f)).
$$
When $f$ has an isolated singularity at $0\in\bba^{n}$, this is proved in \cite{Var82}, \cite{EL82}, \cite{Sai88}. In general, the upper bound $\widetilde{\alpha}_{0}(f)\leq\widetilde{\alpha}(P(f))$ follows from \cite[Proposition 2.1]{CDM25} (in the form of Lemma \ref{minimal exponent hypersurface upper bound} below), while the lower bound $\widetilde{\alpha}_{0}(f)\geq\widetilde{\alpha}(P(f))$ follows from  \cite[Corollary D]{MP20} applied to a standard log resolution of $f$ described in \cite{Var76} (see also \cite[Section 8.2]{AGZV12}); alternatively, Corollary \ref{Newton minimal exponent lower bound} provides a self-contained proof of the lower bound. Similarly, if $f$ is globally Newton non-degenerate (see Definition \ref{Newton}), then $\widetilde{\alpha}(f)=\widetilde{\alpha}(P(f))$.
\end{remark}

\begin{remark}\label{howald}
Given a nonzero monomial ideal $\frak{a} \subset \bbc[y_1,\dots,y_m]$ and an element $u\in \bbz_{\geq 0}^m$, we have $\lct_{y^u}(\frak{a}) = \widetilde{\alpha}(P(\frak{a}); u)$ by Howald's characterization of multiplier ideals of monomial ideals \cite{How01}.
\end{remark}

Motivated by Lemma \ref{minimal exponent in terms of hypersurface}, we make the following definition for the minimal exponent of several Newton polyhedra.

\begin{definition}\label{minimal exponent several newton poly}
Given Newton polyhedra $P_1,\dots,P_r \subset \bbr^n$ and an element $u\in \bbz^n_{\geq 0}$, we define the minimal exponent of $P_1,\dots,P_r$ relative to $u$ as
$$
\widetilde{\alpha}(P_1,\dots,P_r; u) := \min_{1\leq i\leq r} \widetilde{\alpha}(\pi_i(G); \pi_i(0\times u)) \in \bbq_{>0}\cup\hspace{0.15em} \{\infty\},
$$
where $G\subset \bbr^{r+n}$ denotes the Newton polyhedron associated to $P_1,\dots,P_r$, as defined in Section \ref{The Newton polyhedron associated to several Newton polyhedra}. We then define the minimal exponent of $P_1,\dots,P_r$ as $\widetilde{\alpha}(P_1,\dots,P_r):=\widetilde{\alpha}(P_1,\dots,P_r; 0)$. 
\end{definition}

\begin{definition}\label{minimal exponent number def}
Given $a_1,\dots,a_r,k \in \bbz_{\geq 0}^n$, we define the minimal exponent of $a_1,\dots,a_r$ relative to $k$ as
$$
\widetilde{\alpha}(a_1,\dots,a_r; k) := \widetilde{\alpha}(P(\{a_1\}),\dots,P(\{a_r\}); k) \in \bbq_{>0} \cup\hspace{0.15em} \{\infty\}.
$$
\end{definition}

\begin{remark}\label{comb conn}
Let $R=\bbc[x_1,\dots,x_n]$. Given nonzero polynomials $f_1,\dots,f_r \in R$, let $G \subset \bbr^{r+n}$ be the Newton polyhedron associated to $P(f_1),\dots,P(f_r)$. Consider the polynomials $g=\sum_{i=1}^r f_iz_i  \in R[z_1,\dots,z_r]$ and $g_i = f_i + \sum_{j\neq i} f_jz_j \in R[z_1,\dots,\widehat{z}_i,\dots,z_r]$ for $1\leq i\leq r$, and observe that $G = P(g)$ and $\pi_i(G) = P(g_i)$ for each $i$. Thus we have
$$
\widetilde{\alpha}(P(f_1),\dots,P(f_r)) = \min_{1\leq i\leq r} \widetilde{\alpha}(P(g_i)).
$$
\end{remark}

\begin{remark}\label{finite minimal exponent several newton poly}
With the notation of Definition \ref{minimal exponent several newton poly}, we have $\widetilde{\alpha}(P_1,\dots,P_r;u)<\infty$ if and only if $0\notin  P_{i}$ for some $i$, in which case
$$
\widetilde{\alpha}(P_1,\dots,P_r;u)=1/ \min\bigg\{t\in \bbr_{>0} \ \big| \ t(\bfo_{r+n} + (0\times u)) \in \bigcap_{i=1}^r \pi_i^{-1}(\pi_i(G))\bigg\}.
$$
\end{remark}

\begin{remark}\label{rational poly}
With the notation of Definition \ref{minimal exponent several newton poly}, it is immediate that
$$
\widetilde{\alpha}(P_1,\dots,P_r;u)=\sup\bigg\{t\in \bbr_{>0} \ \big| \ t^{-1}(\bfo_{r+n} + (0\times u)) \in \bigcap_{i=1}^r \pi_i^{-1}(\pi_i(G))\bigg\}.
$$
By Corollary \ref{rational poly intersection}, we then obtain
$$
\widetilde{\alpha}(P_1,\dots,P_r; u) = \inf_{v} L_{v,P_1,\dots,P_r}(\bfo_{r+n} + (0\times u)),
$$
where the infimum runs over the elements $v\in (\bbr^n)^*_{\geq 0}$ such that $d(v,P_i)=1$ for some $i$, with the convention that the infimum is $\infty$ if there are no such $v$, and $L_{v,P_1,\dots,P_r}$ is as defined in Definition \ref{linear functional}. Moreover, if $0 \notin  P_{i}$ for some $i$, then there exists a nontrivial facet $Q$ of $G$ that is not contained in any coordinate hyperplane such that the infimum is attained at the element $v\in (\bbr^n)^*_{\geq 0}$ given by $v(e_j)=L_Q(e_j)$ for $1\leq j\leq n$, where $e_{1},\dots,e_{n}$ is the standard basis of $\bbr^{n}$. Also, we have
$$
\widetilde{\alpha}(G) =\min\{r,\widetilde{\alpha}(P_{1},\dots,P_{r})\}.
$$
\end{remark}

We now introduce the following definition, which we will use to give an alternative expression for the minimal exponent of several Newton polyhedra.

\begin{definition}\label{minimal exponent numbers}
Let $c, d_1,\dots,d_r \in \bbr$. If $d_1\leq \cdots \leq d_r$, we define
$$
\widetilde{\alpha}_{\num}(d_1,\dots,d_r; c) := \min\{i+ \tfrac{1}{d_i} ( c+1 - d_1- \cdots -  d_i) \ | \ \text{$1\leq i\leq r$ with $d_i\neq 0$} \},
$$
with the convention that the minimum is $\infty$ if $d_1=\cdots=d_r=0$. In general, we define
$$
\widetilde{\alpha}_{\num}(d_1,\dots,d_r; c):=\widetilde{\alpha}_{\num}(d_{\sigma(1)},\dots,d_{\sigma(r)}; c),
$$
where $\sigma \in S_r$ is any permutation such that $d_{\sigma(1)} \leq \cdots \leq d_{\sigma(r)}$.
\end{definition}

For $c, d_1,\dots,d_r,t \in \bbr$ with $t>0$, note that
\begin{equation}\label{scaling}
\widetilde{\alpha}_{\num}(td_1,\dots,td_r; tc-1)=\widetilde{\alpha}_{\num}(d_1,\dots,d_r; c-1).
\end{equation}

\begin{theorem}\label{minimal exponent in terms of numbers}
With the notation of Definition \ref{minimal exponent several newton poly}, we have
$$
\widetilde{\alpha}(P_1,\dots,P_r;u) = \inf_{v\in (\bbr^n)^*_{\geq 0}} \widetilde{\alpha}_{\num}(d(v,P_1),\dots,d(v,P_r);v(u)+|v|-1).
$$
Moreover, if $0 \notin  P_{i}$ for some $i$, then the infimum is attained at an element $v\in (\bbr^n)^*_{\geq 0}$ given by $v(e_j)=L_Q(e_j)$ for $1\leq j \leq n$ for some nontrivial facet $Q$ of $G$ that is not contained in any coordinate hyperplane.
\end{theorem}
\proof Given $v\in (\bbr^n)^*_{\geq 0}$, let $d_i = d(v,P_i)$ for $1\leq i \leq r$, and let $\sigma\in S_r$ be a permutation such that $d_{\sigma(1)}\leq \cdots \leq d_{\sigma(r)}$. For $w\in (\bbr^n)^*_{\geq 0}$, we write $L_w = L_{w,P_1,\dots,P_r}$ (see Definition \ref{linear functional}). If $d_{\sigma(i)}>0$, then observe that
$$
L_{d_{\sigma(i)}^{-1}v} ( \bfo_{r+n} + (0\times u)) = d_{\sigma(i)}^{-1}v(\bfo_n + u) + \sum_{j=1}^i (1-d_{\sigma(i)}^{-1}d_{\sigma(j)})
$$
$$
= i + \tfrac{1}{d_{\sigma(i)}}(v(u) + |v| - d_{\sigma(1)}-\cdots -d_{\sigma(i)}).
$$
Thus, with the convention that $\min \emptyset = \infty$, we have
$$
\min\{L_{d_i^{-1}v} ( \bfo_{r+n} + (0\times u)) \ | \ \text{$1\leq i\leq r$ with $d_i>0$} \} = \widetilde{\alpha}_{\num}(d_1,\dots,d_r ;v(u) + |v| -1).
$$
The equation in the theorem statement then follows from Remark \ref{rational poly} since the set of elements $L_v$ with $v\in (\bbr^n)^*_{\geq 0}$ such that $d(v,P_i)=1$ for some $i$ is equal to the set of elements $L_{d(v,P_i)^{-1}v}$ with $v\in (\bbr^n)^*_{\geq 0}$ and $1\leq i\leq r$ such that $d(v,P_i)>0$. The final assertion on where the infimum is attained also follows from Remark \ref{rational poly}. \qed

\begin{remark}\label{dimension one}
For $a_1,\dots,a_r,k\in \bbz_{\geq 0}$, we have $\widetilde{\alpha}(a_1,\dots,a_r;k)=\widetilde{\alpha}_{\num}(a_1,\dots,a_r;k)$ by \eqref{scaling} and Theorem \ref{minimal exponent in terms of numbers}.
\end{remark}

The next theorem expresses the minimal exponent of several Newton polyhedra in terms of Definition \ref{minimal exponent number def} and any refinement of the normal fan of their Minkowski sum.

\begin{theorem}\label{minimal exponent fan}
Given Newton polyhedra $P_1,\dots,P_r \subset \bbr^n$, let $\Sigma$ be any fan that refines the normal fan of $P_{i}$ for each $i$. For each cone $\sigma\in\Sigma$, let $[\sigma]$ denote the set of primitive ray generators of $\sigma$, equipped with a fixed ordering $[\sigma]=\{v_1,\dots,v_l\}$ so that, for each $i$, $(d(v,P_i))_{v\in [\sigma]}$ and $(|v|-1)_{v\in [\sigma]}$ may be regarded as elements of $\bbz_{\geq 0}^{l}$. Then we have
$$
\widetilde{\alpha}(P_1,\dots,P_r) = \min_{\sigma\in \Sigma} \widetilde{\alpha}( (d(v,P_1))_{v\in [\sigma]} ,\dots, (d(v,P_r))_{v\in [\sigma]}; (|v|-1)_{v\in [\sigma]} ).
$$
Moreover, the minimum can be taken over the $n$-dimensional cones $\sigma$ of $\Sigma$.
\end{theorem}
\proof For each cone $\sigma\in\Sigma$, let
$$
c_{\sigma}=\widetilde{\alpha}( (d(v,P_1))_{v\in [\sigma]} ,\dots, (d(v,P_r))_{v\in [\sigma]}; (|v|-1)_{v\in [\sigma]} ).
$$
If $\tau$ is a face of a cone $\sigma\in\Sigma$, then by Theorem \ref{minimal exponent in terms of numbers}, it is easy to see that $c_{\sigma}\leq c_{\tau}$. So it suffices to show that
$$
\widetilde{\alpha}(P_1,\dots,P_r) = \min_{\dim\sigma=n}c_{\sigma}.
$$
To see this, given an element $\lambda \in (\bbr^n)^*_{\geq 0}$ and an $n$-dimensional cone $\sigma\in\Sigma$, consider the element
$$
v=\sum_{j=1}^l \lambda(e_j)v_j\in  (\bbr^n)^*_{\geq 0},
$$
where $[\sigma]=\{v_1,\dots,v_l\}$ and $e_{1},\dots,e_{l}$ is the standard basis of $\bbr^{l}$. Note that $v$ ranges over all elements of $(\bbr^n)^*_{\geq 0}$ as $\lambda$ and $\sigma$ vary. For each $i$, put $P'_i = P(\{(d(v_j,P_i))_{j=1}^l\}) \subset \bbr^l$, and consider the elements $L_v=L_{v,P_1,\dots,P_r}$ and $L'_{\lambda}=L'_{\lambda,P'_1,\dots,P'_r}$ in $(\bbr^{r+l})^*_{\geq 0}$, as defined in Definition \ref{linear functional}. Observe that
$$
d(\lambda,P'_i) = \lambda\big( (d(v_j,P_i))_{j=1}^l)\big) = \sum_{j=1}^l \lambda_j d(v_j,P_i) = d(v,P_i)
$$
for each $i$, where the last equality follows from Lemma \ref{closure lemma}. Thus we have
$$
L'_{\lambda}(\bfo_{r+l} + (0\times (|v_j|-1)_{j=1}^l))
= L'_{\lambda}((\bfo_r \times 0) + (0\times (|v_j|)_{j=1}^l)
=L_v(\bfo_{r+l}).
$$
The theorem then follows from Remark \ref{rational poly}. \qed

We conclude this section with the following result.

\begin{lemma}\label{explicit lct}
With the notation of Definition \ref{minimal exponent several newton poly}, we have
$$
\widetilde{\alpha}(G; \beta \times u) \geq \min\{ r+|\beta|, \widetilde{\alpha}(P_1,\dots,P_r;u)\}
$$
for every $\beta\in \bbz_{\geq 0}^r$.
\end{lemma}
\proof If $0\in P_{i}$ for all $i$, then the lemma follows since this implies $G=P(\{z_{1},\dots,z_{r}\})$, $\widetilde{\alpha}(G; \beta \times u)=r+|\beta|$, and $\widetilde{\alpha}(P_1,\dots,P_r;u)=\infty$. So assume that $0\notin P_{i}$ for some $i$. Let $Q_0$ denote the trivial facet of $G$, and note that $L_{Q_0}(\bfo_{r+n}+(\beta \times u)) = r+|\beta|$. Also, for every $L\in (\bbr^{r+n})^*_{\geq 0}$, note that 
$$
L(\bfo_{r+n}+(\beta \times u))
=L(\bfo_{r+n}+(0 \times u)) + L(\beta \times 0)
 \geq L(\bfo_{r+n}+(0 \times u)).
$$ 
So by Corollary \ref{rational poly intersection}, we have
$$
\widetilde{\alpha}(G; \beta \times u) = \min_Q L_Q(\bfo_{r+n}+(\beta \times u)) 
\geq \min\{r+|\beta|, \min_{Q\neq Q_0} L_Q(\bfo_{r+n}+(0 \times u))\}
$$
$$
= \min\{r+|\beta|, \widetilde{\alpha}(P_1,\dots,P_r;u)\},
$$
where the first minimum runs over the facets $Q$ of $G$ that are not contained in any coordinate hyperplane, the third minimum runs over the nontrivial facets $Q$ of $G$ that are not contained in any coordinate hyperplane, and the second equality follows from Remark \ref{rational poly}. \qed

\section{Computing the minimal exponent of several Newton polyhedra}\label{Computing the minimal exponent of several Newton polyhedra}

Throughout this section, fix a positive integer $r$ and a nonnegative integer $n$. Let $e_1,\dots,e_n$ denote the standard basis of $\bbr^n$ and let $e^{*}_1,\dots,e^{*}_n\in (\bbr^n)^{*}$ denote the corresponding dual basis. We denote the standard bases of $\bbr^r$ and $\bbr^{r+n}$ by $\{e'_1,\dots,e'_r\}$ and $\{z_1,\dots,z_r,x_1,\dots,x_n\}$, respectively. 
For $1\leq i\leq r$, we denote by $\pi_i\colon\bbr^{r+n} \to \bbr^{r-1+n}$ the standard projection obtained by omitting the $z_i$-coordinate. Also, for each subset $J\subset\{1,\dots,n\}$, we denote by $p_{J}\colon\bbr^{n}\to\bbr^{|J|}$ the standard projection onto the coordinates indexed by $J$.

The goal of this section is to introduce a method for computing the minimal exponent of several Newton polyhedra. If $P_{1},\dots,P_{r}\subset \bbr^{n}$ are Newton polyhedra and $G\subset \bbr^{r+n}$ is the Newton polyhedron associated to $P_{1},\dots,P_{r}$, then determining the value of $\widetilde{\alpha}(P_{1},\dots,P_{r})$ directly from the definition or from Theorem \ref{minimal exponent in terms of numbers} requires an explicit description of the facets of $G$, which is often difficult to obtain in general. However, Theorem \ref{minimal exponent fan} says that if we can find a fan $\Sigma$ that refines the normal fan of $P_{i}$ for each $i$, then we have
$$
\widetilde{\alpha}(P_1,\dots,P_r) = \min_{\dim\sigma =n} \widetilde{\alpha}( (d(v,P_1))_{v\in [\sigma]} ,\dots, (d(v,P_r))_{v\in [\sigma]}; (|v|-1)_{v\in [\sigma]} ),
$$
where the minimum runs over the $n$-dimensional cones $\sigma$ of $\Sigma$. Therefore, the problem of computing $\widetilde{\alpha}(P_{1},\dots,P_{r})$ reduces to the simpler problem of computing $\widetilde{\alpha}(a_{1},\dots,a_{r};k)$ for elements $a_{1},\dots,a_{r},k\in \bbz^{n}$. This, in turn, requires an explicit description of the facets of $G$ when $P_{i}=P(\{a_{i}\})$ for each $i$, which we obtain after the following definition and remark.

\begin{definition}\label{}
Let $a_{1},\dots,a_{r}\in\bbz^{n}_{\geq 0}$. We define $\vs(a_{1},\dots,a_{r})\subset (\bbr^{n})_{\geq 0}^{*}$ as the subset consisting of elements $v\in (\bbr^{n})_{\geq 0}^{*}$ such that if we set 
$$
I=\{i \in \{1,\dots,r\}  \ | \ v(a_i)= 1\}
\ \ \
\text{and}
\ \ \
J=\{j \in \{1,\dots,n\}  \ | \ v(e_j)>0\},
$$
then $I$ and $J$ are both nonempty and $ \dim\Conv(\{p_{J}(a_i)\}_{i\in I})=|J|-1$.
\end{definition}

The following remark provides an alternative description of $\vs(a_{1},\dots,a_{r})$, which is often easier to compute in practice.

\begin{remark}\label{computing in practice}
Let $a_{1},\dots,a_{r}\in\bbz^{n}_{\geq 0}$. Let $\ps(a_{1},\dots,a_{r})$ denote the set of all pairs $(I,J)$ with $\emptyset\neq I \subset \{1,\dots,r\}$ and $\emptyset\neq J \subset\{1,\dots,n\}$ such that the affine span of $\{p_{J}(a_{i})\}_{i\in I}$ in $\bbr^{|J|}$ has dimension $|J|-1$, intersects each coordinate axis at a point different from the origin, and does not contain $p_{J}(a_{i})$ for each $i\notin I$. For $(I,J)\in \ps(a_{1},\dots,a_{r})$, let $\phi(I,J)$ denote the unique element $v$ of $(\bbr^n)^*_{\geq 0}$ such that for $i \in \{1,\dots,r\}$, we have $v(a_i)= 1$ if and only if $i\in I$, and for $j \in \{1,\dots,n\}$, we have $v(e_{j})> 0$ if and only if $j\in J$. Then $\phi(I,J)\in \vs(a_{1},\dots,a_{r})$ and it is easy to see that the map $\phi\colon \ps(a_{1},\dots,a_{r})\to \vs(a_{1},\dots,a_{r})$ is a bijection.
\end{remark}

\begin{theorem}\label{candidates result}
Let $a_{1},\dots,a_{r}\in \bbz^{n}_{\geq 0}$, set $P_{i}=P(\{a_{i}\})\subset \bbr^{n}$ for each $i$, and consider the Newton polyhedron $G\subset \bbr^{r+n}$ associated to $P_{1},\dots,P_{r}$. Then for $v\in (\bbr^n)^*_{\geq 0}$, we have $L_{v,P_{1},\dots,P_{r}}=L_{Q}$ for some nontrivial facet $Q$ of $G$ that is not contained in any coordinate hyperplane if and only if $v\in \vs(a_{1},\dots,a_{r})$, where $L_{v,P_{1},\dots,P_{r}}$ is as defined in Definition \ref{linear functional}.
\end{theorem}

Before proving the theorem, we begin with the following lemma.

\begin{lemma}\label{face dimension count}
Let $a_{1},\dots,a_{r}\in \bbz^{n}_{\geq 0}$, set $P_{i}=P(\{a_{i}\})\subset \bbr^{n}$ for each $i$, and consider the Newton polyhedron $G\subset \bbr^{r+n}$ associated to $P_{1},\dots,P_{r}$. Given an element $v\in (\bbr^{n})_{\geq 0}^{*}$, set 
$$
I=\{i \in \{1,\dots,r\}  \ | \ v(a_i)= 1\}
\ \ \
\text{and}
\ \ \
J=\{j \in \{1,\dots,n\}  \ | \ v(e_j)>0\}.
$$
If $I$ and $J$ are both nonempty, then $d(L_{v,P_{1},\dots,P_{r}},G)= 1$ and
$$
\dim F(L_{v,P_{1},\dots,P_{r}},G) = r+n-|J| + \dim\Conv(\{p_{J}(a_i)\}_{i\in I}).
$$
\end{lemma}
\proof Let $A=\{ (e'_{i}, a_{i})\ |\ 1\leq  i\leq r \}$, so that $G =P(A)$, and write $L_{v}=L_{v,P_{1},\dots,P_{r}}$. Note that $d(L_{v},G)\geq 1$ and $L_{v}((e'_{i},a_{i}))=1$ for $i\in I$, which implies that $d(L_{v},G)= 1$ since $I$ is nonempty. We set
$$
s=|\{i \in \{1,\dots,r\}  \ | \ v(a_i)< 1\}|.
$$
By reindexing $a_{1},\dots,a_{r}$, if necessary, we may assume that $v(a_i)< 1$ for $i\leq s$. We denote by $q\colon \bbr^{r}\to\bbr^{s}$ the standard projection onto the first $s$ coordinates and by $e''_{1},\dots ,e''_{s}$ the standard basis of $\bbr^{s}$. Also, we set $\pi=q\times p_{J}\colon\bbr^{r+n}\to\bbr^{s+|J|}$. For $y\in\{z_{1},\dots, z_{r},x_{1},\dots,x_{n}\}$, note that $L_{v}(y)>0$ if and only if $y\in\{z_{1},\dots, z_{s}\}\cup\{x_{j}\}_{j\in J}$. Thus, if we consider the element $L\in (\bbr^{s+|J|})_{>0}^{*}$ given by $L_{v}=L\circ\pi$, then $\pi(F(L_v,G))=F(L,\pi(G))$, which implies that $\pi(F(L_v,G))$ is a bounded face of $\pi(G)$, and $F(L_v,G)=\pi^{-1}(F(L,\pi(G)))\cap G$,
which implies that $\pi(F(L_v,G))\cap\pi(A)=\pi(F(L_v,G)\cap A)$. Since bounded faces are the convex hull of their vertices and $\pi(G)=P(\pi(A))$,  we have
$$
\pi(F(L_v,G))=\Conv(\pi(F(L_v,G))\cap\pi(A)).
$$
Observe that $F(L_v,G)\cap A=\{ (e'_{i}, a_{i})\ |\ i\in \{1,\dots,s\} \cup I \}$ and
$$
\pi(F(L_v,G))\cap\pi(A)=\pi(F(L_v,G)\cap A)= 
\{ (e''_i,p_{J}(a_{i}))\ | \ 1\leq i\leq s \}
\cup
\{ (0,p_{J}(a_{i}))\ | \ i\in I \}.
$$
Thus we have $\dim\pi(F(L_v,G))=s+ \dim \Conv(\{p_{J}(a_i)\}_{i\in I})$, which implies that
$$
\dim F(L_v,G)=
r+n-s-|J|+\dim\pi(F(L_v,G))
$$
$$
=r+n-|J|+ \dim \Conv(\{p_{J}(a_i)\}_{i\in I}).\qed
$$

\noindent{\it Proof of Theorem \ref{candidates result}}. Fix an element $v\in (\bbr^n)^*_{\geq 0}$, and set 
$$
I=\{i \in \{1,\dots,r\}  \ | \ v(a_i)= 1\}
\ \ \
\text{and}
\ \ \
J=\{j \in \{1,\dots,n\}  \ | \ v(e_j)>0\}.
$$
If $L_{v}=L_{Q}$ for some nontrivial facet $Q$ of $G$ that is not contained in any coordinate hyperplane, then $I$ is nonempty by Theorem \ref{structure of facets} and $J$ is nonempty since $Q$ is nontrivial. Thus we may assume that $I$ and $J$ are both nonempty. It suffices to prove that $L_{v}=L_{Q}$ for some nontrivial facet $Q$ of $G$ that is not contained in any coordinate hyperplane if and only if $\dim\Conv(\{p_{J}(a_i)\}_{i\in I})=|J|-1$. By Lemma \ref{face dimension count}, we have $d(L_{v},G)= 1$ and
$$
\dim F(L_{v},G) = r+n-|J| + \dim\Conv(\{p_{J}(a_i)\}_{i\in I}).
$$
Thus $\dim\Conv(\{p_{J}(a_i)\}_{i\in I})=|J|-1$ if and only if $\dim F(L_v,G)=n+r-1$, which, by Remark \ref{numerical facet characterization}, is equivalent to saying that $L_{v}=L_{Q}$ for some nontrivial facet $Q$ of $G$ that is not contained in any coordinate hyperplane. \qed

\begin{corollary}\label{minimal exponent singleton}
Given $a_1,\dots,a_r,k\in \bbz^n_{\geq 0}$, we have
$$
\widetilde{\alpha}(a_1,\dots,a_r;k) = \min_{v\in S} \widetilde{\alpha}_{\num}(v(a_1),\dots,v(a_r); v(k)+|v|-1)
$$
for any subset $S\subset (\bbr^{n})^{*}_{\geq 0}$ with $\vs(a_1,\dots,a_r)\subset S$.
\end{corollary}
\proof This follows immediately from Theorems \ref{minimal exponent in terms of numbers} and \ref{candidates result}.\qed

\begin{corollary}\label{nested minimal exponent singleton}
Given $a_1,\dots,a_r,k\in \bbz^n_{\geq 0}$ with $a_{i}=(a_{ij})_{j=1}^{n}$ and $k=(k_{j})_{j=1}^{n}$, suppose that $a_{1j}\leq \cdots\leq  a_{rj}$ for each $j\in \{1,\dots,n\}$. Then we have
$$
\widetilde{\alpha}(a_1,\dots,a_r;k) = \min_{1\leq j\leq n} 
\min\{ i + \tfrac{1}{a_{ij}}(k_{j}+1 - a_{1j}-\cdots - a_{ij}) \ | \ 1\leq i\leq r\text{ with }a_{ij}\neq 0 \}.
$$
\end{corollary}
\proof Let $v\in \vs(a_1,\dots,a_r)$, and set 
$$
I=\{i \in \{1,\dots,r\}  \ | \ v(a_i)= 1\}
\ \ \
\text{and}
\ \ \
J=\{j \in \{1,\dots,n\}  \ | \ v(e_j)>0\}.
$$
Let $w$ denote the unique element of $(\bbr^{|J|})^{*}_{>0}$ given by $v=w\circ p_{J}$. Since $w(p_{J}(a_{i}))=1$ for each $i\in I$ and $a_{1j}\leq \cdots\leq  a_{rj}$ for each $j\in \{1,\dots,n\}$, it follows that $p_{J}(a_{i})=p_{J}(a_{i'})$ for all $i,i'\in I$, hence we obtain $|J|-1=\dim\Conv(\{p_{J}(a_i)\}_{i\in I})=0$ and $|J| =1$. If we write $J=\{j\}$ for some $j$, then $v= \lambda e^{*}_{j}$ for some $\lambda>0$. Thus, by Corollary \ref{minimal exponent singleton}, if we put
$$
S=\{ \lambda e^{*}_{j}\ |\ \lambda>0,1\leq j\leq n\},
$$
then we have
$$
\widetilde{\alpha}(a_1,\dots,a_r;k) = \min_{v\in S} \widetilde{\alpha}_{\num}(v(a_1),\dots,v(a_r); v(k)+|v|-1)
$$
$$
= \min_{\lambda>0,1\leq j\leq n} \widetilde{\alpha}_{\num}(\lambda ea_{j1},\dots,\lambda a_{jr}; \lambda k_{j}+\lambda -1)
$$
$$
 = \min_{1\leq j\leq n} 
\min\{ i + \tfrac{1}{a_{ij}}(k_{j}+1 - a_{1j}-\cdots - a_{ij}) \ | \ 1\leq i\leq r\text{ with }a_{ij}\neq 0 \}.\qed
$$

We conclude this section with the following formula for the minimal exponent of nested Newton polyhedra.

\begin{theorem}\label{minimal exponent nested polyhedra}
Let $P_{1},\dots,P_r\subset \bbr^{n}$ be Newton polyhedra with $P_{1}\supset \cdots\supset P_r$. If $\Sigma$ is any fan that refines the normal fan of $P_{i}$ for each $i$, then we have
$$
\widetilde{\alpha}(P_1,\dots,P_r) = \min_{v}
\min\{ i + \tfrac{1}{d(v,P_i)}(|v| - d(v,P_1)-\cdots - d(v,P_i)) \ | \ 1\leq i\leq r,d(v,P_i)\neq 0 \},
$$
where the outer minimum runs over the primitive ray generators $v$ of $\Sigma$.
\end{theorem}
\proof By Theorem \ref{minimal exponent in terms of numbers}, we have
$$
\widetilde{\alpha}(P_1,\dots,P_r) = \min_{\sigma \in \Sigma} \widetilde{\alpha}( (d(v,P_1))_{v\in [\sigma]} ,\dots, (d(v,P_r))_{v\in [\sigma]}; (|v|-1)_{v\in [\sigma]} ).
$$ 
Since $P_{1}\supset \cdots\supset P_r$, we have $d(v,P_1)\leq \cdots\leq  d(v,P_r)$ for all $v\in (\bbr^{n})^{*}_{\geq 0}$, hence the corollary follows from Lemma \ref{nested minimal exponent singleton}.\qed

\section{A combinatorial upper bound for the minimal exponent}\label{A combinatorial upper bound for the minimal exponent}

In this section we restate the upper bounds for the minimal exponent in \cite[Theorem 1.2, Proposition 2.1]{CDO25} in terms of the minimal exponent of several Newton polyhedra.

\begin{lemma}\label{minimal exponent hypersurface upper bound}
Given a nonzero polynomial $f\in \bbc[y_1,\dots,y_m]$ with $f\in (y_1,\dots,y_m)^2$, we have
$$
\widetilde{\alpha}_0(f) \leq \widetilde{\alpha}(P(f)).
$$
\end{lemma}
\proof Let $e_1,\dots,e_m$ denote the standard basis of $\bbr^m$. By \eqref{half space intersection}, there exists a facet $Q$ of $P(f)$ that is not contained in any coordinate hyperplane such that $\widetilde{\alpha}(P(f))=\sum_{i=1}^m L_Q(e_i)$. For $\epsilon>0$, define $L_{\epsilon} \in (\bbr^m)^*_{>0}$ by $L_{\epsilon}(e_i) = L_Q(e_i)+\epsilon$. By \cite[Proposition 2.1]{CDM25}, we have 
$$
\widetilde{\alpha}_0(f)\leq \frac{L_{\epsilon}(e_1)+\cdots +L_{\epsilon}(e_m)}{d(L_{\epsilon}, P(f))}.
$$
Taking the limit as $\epsilon$ goes to $0$, we obtain $\widetilde{\alpha}_0(f)\leq\sum_{i=1}^m L_Q(e_i)=\widetilde{\alpha}(P(f))$, as desired.

\begin{corollary}\label{minimal exponent upper bound}
Let $f_1,\dots,f_r\in \bbc[x_1,\dots,x_n]$ be polynomials with $f_1,\dots,f_r \in (x_1,\dots,x_n)^2$ such that the subvariety $Z \subset \bba^n$ defined by $f_1,\dots,f_r$ is a local complete intersection of codimension $r$ in some open neighborhood of $0$. Then we have
$$
\widetilde{\alpha}_0(Z)  \leq \widetilde{\alpha}(P(f_1),\dots,P(f_r)).
$$
\end{corollary}
\proof Let $R=\bbc[x_1,\dots,x_n]$ and $g_i = f_i+\sum_{j\neq i}z_i f_i \in R[z_1,\dots,\widehat{z}_i,\dots,z_r]$ for each $i$. Then we have
$$
\widetilde{\alpha}_0(Z) \leq \min_{1\leq i\leq r} \widetilde{\alpha}_{(0,0)}(g_i) \leq \min_{1\leq i\leq r} \widetilde{\alpha}(P(g_i)) = \widetilde{\alpha}(P(f_1),\dots,P(f_r)),
$$
where the first inequality follows from  Lemma \ref{minimal exponent in terms of hypersurface}, the second equality follows from Lemma \ref{minimal exponent hypersurface upper bound}, and the equality follows from Remark \ref{comb conn}. \qed

\section{A general lower bound for the minimal exponent}\label{A general lower bound for the minimal exponent}

The goal of this section is to prove Theorem \ref{minimal exponent lower bound} and establish some basic properties of the relative minimal exponent, which was discussed in the Introduction and is redefined below in terms of the Hodge filtration.

\begin{definition}\label{relative minimal exponent}
Let $X$ be a smooth variety, let $g,f_1,\dots,f_{r}\in \os_X(X)$ be regular functions, and denote the $r$-tuple $(f_1,\dots,f_{r})$ by $\bff$. We define the minimal exponent of $f_1,\dots,f_{r}$ relative to $g$ as
$$
\widetilde{\alpha}(f_1,\dots,f_{r}; g) = \left\{
\begin{array}{cl}
\sup\{\gamma>0 \ | \  g\delta_{\bff}\in V^{\gamma}B_{\bff}\}  & \text{if $g\delta_{\bff}\not\in V^rB_{\bff}$} \\[2mm]
\sup\{r-1+q+\gamma \ | \  gF_qB_{\bff}\subset V^{r-1+\gamma}B_{\bff}\}  & \text{if $g\delta_{\bff}\in V^rB_{\bff}$},
\end{array}\right.
$$
where in the latter case, the supremum is over all nonnegative integers $q$ and all rational numbers $\gamma\in (0,1]$ with the property that $gF_qB_{\bff} \subset V^{r-1+\gamma}B_{\bff}$. We also write $\widetilde{\alpha}(\bff; g)$ for $\widetilde{\alpha}(f_1,\dots,f_{r}; g)$.
\end{definition}

\begin{remark}\label{}
With the notation of Definition \ref{relative minimal exponent}, it is clear that
$$
\widetilde{\alpha}(f_{\sigma(1)},\dots,f_{\sigma(r)}; g)=\widetilde{\alpha}(f_1,\dots,f_{r}; g)
$$
for all permutations $\sigma\in S_{r}$.
\end{remark}

\begin{lemma}\label{increment}
Let $g,f_1,\dots,f_r\in \os_X(X)$ be regular functions on a smooth variety $X$. If $q$ is an integer and $\lambda$ is a rational number such that $\lambda \not\in (r-1,r]$ and $gF_qB_{\bff} \subset V^{\lambda}B_{\bff}$, then we have $gF_{q-1}B_{\bff} \subset V^{\lambda+1}B_{\bff}$.
\end{lemma}
\proof The proof is identical to the one in \cite[Remark 4.6]{CDMO24}, which treats the case $g=1$, and which we recall as follows. Fix $\beta\in \bbz^r_{\geq 0}$ with $|\beta|\leq q-1$ and set $u=g\partial_t^{\beta}\delta_{\bff}$. We need to show that $u \in V^{\lambda+1}B_{\bff}$. If we put $\gamma = \sup\{\eta \ | \ u\in V^{\eta}B_{\bff}\}$, then by left continuity of the $V$-filtration, we have $u\in V^{\gamma}B_{\bff}$. So it suffices to show that $\gamma\geq \lambda+1$. For the sake of contradiction, assume that $\gamma < \lambda+1$. 

Let $\overline{u}$ denote the image of $u$ in $\Gr^{\gamma}_VB_{\bff}$, which is nonzero by definition of $\gamma$. Since $\lambda \not\in (r-1,r]$ and $\lambda\leq\gamma < \lambda+ 1$, we have $\gamma\neq r$. Thus $s+r$ acts invertibly on $\Gr^{\gamma}_VB_{\bff}$, so that $(s+r)\overline{u}\neq 0$,
Since $gF_qB_{\bff} \subset V^{\lambda}B_{\bff}$, we have $\partial_{t_i} u \in V^{\lambda}B_{\bff}$ for all $i$ and thus $t_i\partial_{t_i} u \in V^{\lambda+1}B_{\bff}$. We conclude that
$$
(s+r)u = - \sum_{i=1}^r t_i \partial_{t_i} u \in V^{\lambda+1}B_{\bff},
$$
which contradicts the fact that $(s+r)\overline{u}\neq 0$. \qed

\begin{corollary}\label{cdmo remark}
Let $g,f_1,\dots,f_{r}\in \os_X(X)$ be regular functions on a smooth variety $X$ and let $\ell$ be an integer with $\ell\geq r$. If $q_1$ and $q_2$ are nonnegative integers and $\gamma_1,\gamma_2\in (0,1]$ are rational numbers such that $q_1+\gamma_1\geq q_2+\gamma_2$ and $gF_{q_1}B_{\bff} \subset V^{\ell-1+\gamma_1}B_{\bff}$, then we have $gF_{q_2}B_{\bff} \subset V^{\ell-1+\gamma_2}B_{\bff}$.
\end{corollary}
\proof The corollary is clear if $q_{1}=q_{2}$, so assume this is not the case. Our hypothesis then implies that $q_{1}-1\geq q_{2}$.  If $\ell=r$, then since $gF_{q_1}B_{\bff} \subset V^{r-1+\gamma_1}B_{\bff}\subset V^{r-1}B_{\bff}$, we have $gF_{q_1-1}B_{\bff} \subset V^{r}B_{\bff}$ by Lemma \ref{increment}, hence $gF_{q_2}B_{\bff}\subset gF_{q_1-1}B_{\bff} \subset V^{r}B_{\bff}\subset V^{r-1+\gamma_2}B_{\bff}$. If $\ell>r$, then by Lemma \ref{increment}, we have $gF_{q_1-1}B_{\bff} \subset V^{\ell+\gamma_1}B_{\bff}\subset V^{\ell-1+\gamma_2}B_{\bff}$. \qed

\begin{remark}\label{smaller open set}
With the notation of Definition \ref{relative minimal exponent}, if $X=\bigcup_{i}U_{i}$ is a finite open cover, then by Corollary \ref{cdmo remark}, we have
$$
\widetilde{\alpha}(f_1,\dots,f_{r}; g) =\min_{i} \widetilde{\alpha}(f_1|_{U_{i}},\dots,f_{r}|_{U_{i}}; g|_{U_{i}}).
$$
\end{remark}

The following lemma allows us to express the local setup in part ii) of Theorem \ref{minimal exponent lower bound} as the global setup in part i).

\begin{lemma}\label{local setup}
Let $\pi\colon Y \to X$ be a proper morphism of varieties, let $p\colon U/G \onto Y$ be a surjective \'etale morphism, where $G$ is a finite group acting on a smooth affine variety $U$, and denote the morphism $U\to U /G$ by $q$. Given a point $z\in Z$, an open neighborhood $W'$ of $z$, and an open neighborhood $U'$ of $(\pi\circ p\circ q)^{-1}(z)$, there exists an affine open neighborhood $W$ of $z$ and an open neighborhood $V$ of $(\pi\circ p)^{-1}(z)$ such that the following hold:
\begin{enumerate}
\item[i)] $W\subset W'$.
\item[ii)] $p(V)=\pi^{-1}(W)$.
\item[iii)] $q^{-1}(V)\subset U'$.
\end{enumerate} 
In particular, if $U'=\bigcup_{i\in I} U_i$ and $V=\bigcup_{j\in J}V_{j}$ are finite open covers such that each $V_{j}$ is affine, and we set $U_{ij}=U_{i}\cap q^{-1}(V_{j})$ for all $i\in I,j\in J$ and $U''=\coprod_{j\in J}q^{-1}(V_{j})$, then $U''$ is a smooth affine variety with finite open cover $U''=\bigcup_{i,j}U_{ij}$ such that the induced morphism
$$
U''/G=\coprod_{j\in J}V_{j}\to \pi^{- 1}(W)
$$
is surjective and \'etale.
\end{lemma}
\proof If we set $V'=(U/G)\setminus q(U\setminus U')$, then $V'$ is an open neighborhood of $(\pi\circ p)^{-1}(z)$ such that $q^{-1}(V')\subset U'$. Let $W$ be an affine open neighborhood of $z$ such that 
$$
W\subset W'\cap (X\setminus \pi(Y\setminus p(V'))).
$$
Then $\pi^{-1}(W)  \subset p(V')$, so the lemma follows by setting $V=V'\cap(\pi\circ p)^{-1}(W)$. \qed

\noindent{\it Proof of Theorem \ref{minimal exponent lower bound}}. Since part ii) follows from part i) by Remarks \ref{smaller open set} and \ref{local setup}, it suffices to prove part i). So assume that $Z$ is a complete intersection of codimension $r$ in $X$ and $U=\bigcup_{i\in I}U_i$. We set
$$
\eta = \min_{i\in I} \widetilde{\alpha}(\pi^{*}_{i}f_{1},\dots,\pi^{*}_{i}f_{r};g_{i}).
$$

We first assume that $\delta_{\bff} \not\in V^rB_{\bff}$. By left-continuity of the $V$-filtration and \eqref{V roots}, we have $\widetilde{\alpha}(Z)<r$ and that the maximum root of $b_{\delta_{\bff}}(s)$ is $-\widetilde{\alpha}(Z)$. If we denote the maximum root of $b_{g_i\delta_{\pi_i^*\bff}}(s)$ by $r_{i}\in\bbq$ for each $i\in I$, then by \cite[Theorem 1.2]{Lee26}, we have 
$$
\min\{ r_{i}\ |\ i \in I \}\leq \widetilde{\alpha}(Z)<r.
$$ 
Thus, if we let $i_{0}\in I$ be an index such that $r_{i_{0}}=\min\{ r_{i}\ |\ i \in I \}$,  then by \eqref{V roots}, we have $g_{i_{0}}\delta_{\pi_{i_{0}}^*\bff}\not\in V^rB_{\pi_{i_{0}}^*\bff}$, which implies that $r_{i_{0}}=\widetilde{\alpha}(\pi^{*}_{i_{0}}f_{1},\dots,\pi^{*}_{i_{0}}f_{r};g_{i_{0}})$. We then obtain
$$
\widetilde{\alpha}(Z)\geq r_{i_{0}}= \min_{i\in I}\widetilde{\alpha}(\pi^{*}_{i_{0}}f_{1},\dots,\pi^{*}_{i_{0}}f_{r};g_{i_{0}})\geq\eta.
$$

Next, we assume that $\delta_{\bff} \in V^rB_{\bff}$. If $g_i\delta_{\pi_i^*\bff}\not\in V^rB_{\pi_i^*\bff}$ for some $i\in I$, then
$$
\widetilde{\alpha}(Z)\geq r>\widetilde{\alpha}(\pi^{*}_{i}f_{1},\dots,\pi^{*}_{i}f_{r}; g_i) \geq\eta
$$
Thus we may assume that $g_i\delta_{\pi_i^*\bff}\in V^rB_{\pi_i^*\bff}$ for all $i\in I$. By definition of $\eta$, for every $i\in I$, there exist a nonnegative integer $q_i$ and a rational number $\gamma_i\in (0,1]$ such that
$$
g_i F_{q_i}B_{\pi_i^*\bff} \subset V^{r-1+\gamma_i}B_{\pi_i^*\bff}
$$
and $\eta \leq r-1+q_i+\gamma_i$.
Let $i_0\in I$ be an index such that $q_{i_0}+\gamma_{i_0} = \min_{i\in I} (q_i+\gamma_i)$. By Corollary \ref{cdmo remark}, we have
$$
g_i F_{q_{i_0}}B_{\pi_i^*\bff} \subset V^{r-1+\gamma_{i_0}}B_{\pi_i^*\bff}
$$
for all $i\in I$. Thus, by \eqref{V roots} and \cite[Theorem 1.2]{Lee26}, for all $\beta\in\bbz^{r}$ with $|\beta|\leq q_{i_0}$, we have that all roots of $b_{\partial_t^{\beta}\delta_{\bff}}(s)$ are $\leq -(r-1+\gamma_{i_0})$, which implies that
$$
F_{q_{i_0}}B_{\bff} \subset V^{r-1+\gamma_{i_0}}B_{\bff}
$$
and $\widetilde{\alpha}(Z)\geq r-1+q_{i_0} + \gamma_{i_0}\geq \eta$. \qed

The remainder of this section is devoted to establishing some basic properties of the relative minimal exponent.

Let $X=\Spec R$ be a smooth affine variety and let $I\subset R$ be an ideal. For a fixed regular function $g\in R$, the value of $\widetilde{\alpha}(f_1,\dots,f_{r}; g)$ is not independent of the choice of generators $f_1,\dots,f_{r}$ of the ideal $I$ (see Example \ref{generator dependence example} below). However, we have the following lemma, which shows that for a fixed regular function $g\in R$ and a fixed positive integer $r$, the value of $\widetilde{\alpha}(f_1,\dots,f_{r}; g)$ for regular functions $f_1,\dots,f_{r}\in R$ only depends on $f_1,\dots,f_{r'}$ for any positive integer $r'$ with $r'\leq r$ such that $f_{i}$ is contained in the ideal generated by $f_1,\dots,f_{r}$ for all $i>r'$.

\begin{lemma}\label{redundant}
Let $X$ be a smooth variety, let $r'$ and $r''$ be positive integers. Given regular functions $g,f'_1,\dots,f'_{r'},f''_1,\dots,f''_{r''}\in \os_X(X)$ such that $f''_1,\dots,f''_{r''}$ are contained in the ideal generated by $f'_1,\dots,f'_{r'}$, if we denote the $(r'+r'')$-tuple $(f'_{1},\dots,f'_{r'},f''_{1},\dots, f''_{r''})$ by $\bff$ and the $r'$-tuple $(f'_{1},\dots,f'_{r'})$ by $\bff'$, then we have
$$
\widetilde{\alpha}(\bff; g)  = \left\{
\begin{array}{cl}
\sup\{\gamma>0 \ | \  g\delta_{\bff'}\in V^{\gamma}B_{\bff'}\}  & \text{if $g\delta_{\bff'}\not\in V^{r'+r''}B_{\bff'}$} \\[2mm]
\sup\{r'+r''-1+q+\gamma \ | \ gF_qB_{\bff'}\subset V^{r'+r''-1+\gamma}B_{\bff'}\}  & \text{if $g\delta_{\bff'}\in V^{r'+r''}B_{\bff'}$},
\end{array}\right.
$$
where in the latter case, the supremum is over all nonnegative integers $q$ and all rational numbers $\gamma\in (0,1]$ with the property that $gF_qB_{\bff'}\subset V^{r'+r''-1+\gamma}B_{\bff'}$. In particular, we have
$$
\widetilde{\alpha}(f'_{1},\dots,f'_{r'},f''_{1},\dots, f''_{r''};g)=\widetilde{\alpha}(f'_{1},\dots,f'_{r'},h_{1},\dots, h_{r''};g)
$$
for all regular functions $h_{1},\dots,h_{r''}$ contained in the ideal generated by $f'_1,\dots,f'_{r'}$.
\end{lemma}
\proof As in Remark \ref{smaller open set}, it follows from Lemma \ref{increment} that the right-hand-side of the equation in the lemma can be computed as a minimum over any finite open cover. Thus we may assume that $X$ is affine.

We set $r=r'+r''$. We denote the standard coordinates on $\bba^{r'}$ and $\bba^{r''}$ by $t'_1,\dots,t'_{r'}$ and $t''_1,\dots,t''_{r''}$, respectively. For $1\leq i\leq r''$, let us write $f''_{i}=\sum_{j=1}^{r'}a_{ij}f'_{j}$ with $a_{ij}\in \os_X(X)$ and set $g_{i}=\sum_{j=1}^{r'}a_{ij}t'_{j}\in \os_{X \times \bba^{r'}}(X\times \bba^{r'})$.  We denote the $r''$-tuple $(g_{1},\dots,g_{r''})$ by $\bfg$. Finally, we denote by $p_{X}$ the projection $X\times\bba^{r}=X\times\bba^{r'}\times\bba^{r''}\to X$. 

If we denote by $\frak{a} \subset \os_X$ the ideal generated by $f'_1,\dots,f'_{r'}$, then by (\ref{V roots}) and (\ref{multiplier ideals}), we have
$$
\sup\{\gamma>0 \ | \  g\delta_{\bff}\in V^{\gamma}B_{\bff}\}
=\lct_{g}(\frak{a})=
\sup\{\gamma>0 \ | \  g\delta_{\bff'}\in V^{\gamma}B_{\bff'}\}.
$$
Hence we have $g\delta_{\bff}\not\in V^rB_{\bff}$ if and only if $g\delta_{\bff'}\not\in V^rB_{\bff'}$, in which case the lemma holds.

Thus we assume that $g\delta_{\bff}\in V^rB_{\bff}$, which implies that $g\delta_{\bff'}\in V^rB_{\bff'}$. We need to show that
$$
\widetilde{\alpha}(f'_1,\dots,f'_{r'},f''_1,\dots,f''_{r''}; g)
=\sup\{r-1+q+\gamma \ | \  gF_qB_{\bff'}\subset V^{r-1+\gamma}B_{\bff'}\}. 
$$
Since $i_{\bff}=i_{\bfg}\circ i_{\bff'}$, we have the following natural isomorphism of $\ds_{X \times \bba^{r}}$-modules:
$$
i_{\bff,+}\os_{X}\simeq i_{\bfg,+}(i_{\bff',+}\os_{X}).
$$
Applying $p_{X,*}$ then yields the following natural isomorphism of $p_{X,*}\ds_{X \times \bba^{r}}$-modules:
$$
B_{\bff}\simeq \bigoplus_{\beta\in\bbz^{r''}}B_{\bff'} \partial_{t''}^{\beta}\delta_{\bfg},
$$
where $h\delta_{\bff}$ correspond to $(h \delta_{\bff'})\delta_{\bfg}$ for $h\in \os_{X}$. By functoriality of filtered direct image (deduced from the corresponding statement for right $\ds$-modules \cite[Section 2.3.7]{HM}), we have
$$
F_{p}B_{\bff}\simeq \bigoplus_{|\beta|\leq p}F_{p-|\beta|}B_{\bff'} \partial_{t''}^{\beta}\delta_{\bfg}
$$
(this can also be obtained directly from \eqref{bf action}). On the other hand, by \cite[Proposition 2.2]{BMS06}, we have
$$
V^{\lambda}B_{\bff}\simeq \bigoplus_{\beta\in\bbz^{r''}}V^{\lambda+|\beta|}B_{\bff'} \partial_{t''}^{\beta}\delta_{\bfg}.
$$
Since $r''>0$, for $q\in \bbz_{\geq 0}$ and a rational number $\gamma\in (0,1]$, we have that $gF_qB_{\bff}\subset V^{r-1+\gamma}B_{\bff}$ if and only if $gF_{q-i}B_{\bff'}\subset V^{r-1+\gamma+ i}B_{\bff'}$ for all $i\leq q$, which by Lemma \ref{increment}, is equivalent to $gF_{q}B_{\bff'}\subset V^{r-1+\gamma}B_{\bff'}$. Thus the lemma follows. \qed

\begin{example}\label{generator dependence example}
If $f_{1},\dots,f_{r}\in \os_{X}(X)$ are nonzero regular functions on an irreducible smooth variety $X$ such that $f_{2},\dots,f_{r}$ are contained in the ideal generated by $f_{1}$, then by Lemma \ref{redundant}, we have 
$$
\widetilde{\alpha}(f_{1},\dots,f_{r} ;1)=
\begin{cases}
\widetilde{\alpha}(f_{1}) & \text{if $r=1$} \\
\lct(f_{1}) & \text{if $r>1$}.
\end{cases}
$$
Indeed, it follows from \eqref{multiplier ideals} that $\delta_{\bff}\notin V^{r}B_{\bff}$ for all $r >1$.
\end{example}

\begin{lemma}\label{smooth pull}
Given a smooth morphism $\pi\colon  Y \to X$ of smooth varieties and regular functions $g,f_1,\dots,f_{r}\in\os_{X}(X)$, we have $\widetilde{\alpha}(\pi^*f_1,\dots,\pi^*f_r; \pi^* g)\geq\widetilde{\alpha}(f_1,\dots,f_r; g)$, with equality if $\pi$ is surjective.
\end{lemma}
\proof The proof is identical to that of \cite[Proposition 4.12]{CDMO24}, which we recall as follows. Since $\pi$ is smooth, we have an isomorphism $B_{\pi^*\bff} \simeq \pi^*B_{\bff}$ such that $\pi^{*}(h)F_pB_{\pi^*\bff} \simeq \pi^*(hF_pB_{\bff})$ and $V^{\lambda} B_{\pi^*\bff} \simeq \pi^* V^{\lambda}B_{\bff}$ for all $h\in\os_{X}(X)$, $p\in\bbz$, and $\lambda\in\bbq$. Thus for $p\in\bbz$ and $\lambda\in\bbq$, if $gF_pB_{\bff}\subset V^{\lambda}B_{\bff}$, then we have $\pi^{*}(g)F_pB_{\pi^*\bff}\subset V^{\lambda} B_{\pi^*\bff}$, and the converse holds if $\pi$ is surjective. The lemma then follows directly from the definition of the relative minimal exponent. \qed

\begin{lemma}\label{units}
Given a smooth variety $X$, regular functions $g, f_1,\dots,f_r\in \os_X(X)$, and invertible functions $u, u_1,\dots,u_r \in \os_X(X)$, we have $\widetilde{\alpha}(u_1f_1,\dots,u_rf_r; u g) = \widetilde{\alpha}(f_1,\dots,f_r; g)$.
\end{lemma}
\proof We follow the proof in \cite[Remark 4.8]{CDMO24}. Let us write $h_i=u_if_i$ for $1\leq i\leq r$ and denote the projection $X\times\bba^{r}\to X$ by $p_{X}$. Consider the isomorphism 
$$
\phi\colon  X \times \bba^r \xrightarrow{\sim} X \times \bba^r, (x,t_1,\dots,t_r) \mapsto (x, u_1(x)t_1,\dots,u_r(x)t_r).
$$
Since $\phi \circ i_{\bff} = i_{\bfh}$, we have the following isomorphism of $\ds_{X \times \bba^r}$-modules:
$$
i_{\bfh,+}\os_{X}=\phi_{+}(i_{\bff,+}\os_{X})=\phi_{*}(i_{\bff,+}\os_{X}).
$$ 
Applying $p_{X,*}$ yields an isomorphism of $\os_{X}$-modules $\tau\colon  B_{\bfh}\xrightarrow{\sim}B_{\bff}$, which clearly has the property that $\tau(ugF_{p}B_{\bfh})=ugF_{p}B_{\bff}=gF_{p}B_{\bff}$ for $p\in\bbz$. Since $\phi(X\times\{0\})=X\times\{0\}$, it follows from \cite[Proposition 2.2]{BMS06} that we have $\tau(V^{\lambda}B_{\bfh})=V^{\lambda}B_{\bff}$ for $\lambda\in\bbq$. Thus we have $ugF_{q}B_{\bfh} \subset V^{r-1+\gamma}B_{\bfh}$ if and only if $gF_{q}B_{\bff} \subset V^{r-1+\gamma}B_{\bff}$, from which the lemma follows. \qed

\section{A combinatorial lower bound for the minimal exponent}\label{A combinatorial lower bound for the minimal exponent}

The goal of this section is to prove the following theorem,  which provides a combinatorial lower bound for the relative minimal exponent in the monomial setting. Before stating the theorem, we introduce the following notation. If $X$ is a smooth variety with global coordinates $z_{1}, \dots, z_{r},x_{1}, \dots,x_{n}\in\os_{X}(X)$, then we put $x^{\alpha}=\prod_{i=1}^{n}x_{i}^{\alpha_{i}}$ for $\alpha=(\alpha_{1},\dots,\alpha_{n})\in\bbz^{n}_{\geq 0}$. Also, for $a_{1},\dots,a_{r},k\in \bbz^{n}_{\geq 0}$, recall that $\widetilde{\alpha}(a_{1},\dots,a_{r};k)$ was defined in Definition \ref{minimal exponent number def}.

\begin{theorem}\label{combinatorial lower bound theorem}
Let $g,f_{1},\dots,f_{r}\in\os_{X}(X)$ be regular functions on a smooth variety $X$. Assume that $X$ has global coordinates $z_1,\dots,z_{r}, x_1,\dots,x_n\in\os_{X}(X)$ such that $g = u x^k$ and $f_i=u_i z_i x^{a_i}$ for $1\leq i\leq r$ for some invertible functions $u,u_1, \dots,u_{r}\in \os_X(X)$ and elements $k,a_{1},\dots,a_{r}\in\bbz^{n}_{\geq 0}$. Then we have 
$$
\widetilde{\alpha}(f_1,\dots,f_r; g) \geq \widetilde{\alpha}(a_1,\dots,a_r; k).
$$
\end{theorem}

Before proving the theorem, we record two of its main consequences.

\begin{corollary}\label{combinatorial lower bound first corollary}
Let $g,f_{1},\dots,f_{r}\in\os_{X}(X)$ be regular functions on a smooth variety $X$ and let $r'$ be a nonnegative integer with $r'\leq r$. Assume that $X$ has global coordinates $z_1,\dots,z_{r'}, x_1,\dots,x_n\in\os_{X}(X)$ such that $g = u x^k$, $f_i=u_i z_i x^{a_i}$ for $i\leq r'$, and $f_{i}=u_{i} x^{a_{i}}$ for $i>r'$ for some invertible functions $u,u_1, \dots,u_{r}\in \os_X(X)$ and elements $k,a_{1},\dots,a_{r}\in\bbz^{n}_{\geq 0}$. Then we have 
$$
\widetilde{\alpha}(f_1,\dots,f_r; g) \geq \widetilde{\alpha}(a_1,\dots,a_r; k).
$$
\end{corollary}
\proof By Lemma \ref{units}, we may assume that $u =u_1 \cdots=u_{r}=1$, and by Lemma \ref{smooth pull}, we may assume that $X=\bba^{r'+n}$, whose standard coordinates are given by $z_1,\dots,z_{r'}, x_1,\dots,x_n$. We consider $\bba^{r-r'}=\Spec \bbc[z_{r'+1} ,\dots, z_{r}]$ and its open subset $U=\Spec \bbc[z^{\pm}_{r'+1} ,\dots, z^{\pm}_{r}]$, and denote the projections $X\times\bba^{r-r'}\to X$ and $X\times U\to X$ by $p$ and $q$, respectively. Observe that
$$
\widetilde{\alpha}(f_1,\dots,f_r; g)=\widetilde{\alpha}(q^{*}f_1,\dots,q^{*}f_r; q^{*}g)
=\widetilde{\alpha}(q^{*}f_1,\dots,q^{*}f_{r'},z_{r'+1}q^{*}f_{r'+1},\dots,z_{r}q^{*}f_{r}; q^{*}g)
$$
$$
\geq\widetilde{\alpha}(p^{*}f_1,\dots,p^{*}f_{r'},z_{r'+1}p^{*}f_{r'+1},\dots,z_{r}p^{*}f_{r}; p^{*}g),
$$
where the first equality follows from Lemma \ref{smooth pull}, the second equality follows from Lemma \ref{units}, and the inequality follows from Remark \ref{smaller open set}. The corollary then follows from Theorem \ref{combinatorial lower bound theorem}.\qed 

\begin{corollary}\label{combinatorial lower bound second corollary}
Let $g,f_{1},\dots,f_{r}\in\os_{X}(X)$ be regular functions on a smooth variety $X$ and let $r'$ be a nonnegative integer with $r'< r$ such that $f_{i}$ is contained in the ideal generated by $f_{1},\dots,f_{r'}$ for all $i>r'+1$. Assume that $X$ has global coordinates $z_1,\dots,z_{r'}, x_1,\dots,x_n\in\os_{X}(X)$ such that $g = u x^k$, $f_i=u_i z_i x^{a_i}$ for $i\leq r'$, and $f_{r'+1}=u_{r'+1} x^{a_{r'+1}}$ for some invertible functions $u,u_1, \dots,u_{r'+1}\in \os_X(X)$ and elements $k,a_{1},\dots,a_{r'+1}\in\bbz^{n}_{\geq 0}$. Then we have 
$$
\widetilde{\alpha}(f_1,\dots,f_r; g) \geq \widetilde{\alpha}(a_1,\dots,a_r; k)
$$
for all elements $a_{r'+2},\dots,a_{r}\in\bbz^{n}_{\geq 0}$ that are component-wise greater than or equal to $a_{r'}$.
\end{corollary}
\proof By Lemma \ref{redundant}, we may assume that $f_{i}=x^{a_i}$ for $i>r'+1$. The corollary then follows from Corollary \ref{combinatorial lower bound first corollary}. \qed

The remainder of this section is devoted to the proof of Theorem \ref{combinatorial lower bound theorem}. We begin with two auxiliary lemmas that will be used in the proof.

\begin{lemma}\label{min thing}
Given regular functions $g,f_1,\dots,f_r \in \os_X(X)$ on a smooth variety $X$ and a rational number $\eta$ such that $g\partial_t^{\beta}\delta_{\bff} \in V^{\min(r,\eta-|\beta|)}B_{\bff}$ for all $\beta\in \bbz_{\geq 0}^r$, we have
$$
\widetilde{\alpha}(f_1,\dots,f_r; g)\geq \eta.
$$
\end{lemma}
\proof First assume that $g\delta_{\bff} \not\in V^rB_{\bff}$. Since $g\delta_{\bff} \in V^{\min(r,\eta)}B_{\bff}$, we have $\eta<r$ and
$$
\widetilde{\alpha}_r(f_1,\dots,f_r;g) \geq\min(r,\eta)=\eta.
$$

Now assume that $\delta_{\bff} \in V^rB_{\bff}$, which implies that $r\leq \widetilde{\alpha}(f_1,\dots,f_r; g)$. If $\eta \leq r$, then we are done, so assume that $\eta \geq r$. Let us write $\eta = r-1+\gamma+q$ with $q$ a nonnegative integer and $\gamma \in (0,1]$ a rational number. Then for all $\beta\in \bbz^r_{\geq 0}$ such that $|\beta|\leq q$, we have $\eta - q \leq r$, hence $r-1+\gamma  = \min(r, \eta-q) \leq \min(r, \eta-|\beta|)$ and $g\partial_t^{\beta}\delta_{\bff} \in V^{\min(r,\eta-|\beta|)}B_{\bff} \subset V^{r-1+\gamma}B_{\bff}$. Thus
$$
\widetilde{\alpha}(f_1,\dots,f_r; g)\geq r-1+\gamma +q = \eta. \qed
$$

\begin{lemma}\label{microlocal}
Let $g,g_1,\dots,g_r\in \os_S(S)$ be regular functions on a smooth variety $S$, and set $X=S \times \bba^r$ and $f_i=z_i g_i \in \os_X(X)$ for $1\leq i \leq r$, where $z_1,\dots,z_r$ are the standard coordinates on $\bba^{r}$. Then for every $\beta \in \bbz^r_{\geq 0}$, the $b$-functions of the sections $g \partial_t^{\beta}\delta_{\bff}$ and $g z^{\beta} \delta_{\bff}$ of $B_{\bff}$ are related as follows:
$$
b_{g \partial_t^{\beta}\delta_{\bff}}(s) = b_{g z^{\beta}\delta_{\bff}}(s-|\beta|).
$$
\end{lemma}
\proof Let $\bba^{r}_{z}=\Spec\bbc[z_1,\dots,z_r]$ and $\bba^{r}_{y}=\Spec\bbc[y_1,\dots,y_r]$, so that $X=S\times\bba^{r}_{z}$. We set $Y=X \times \bba^r_y$ and $h = \sum_{i=1}^r y_if_i  \in \os_{Y}(Y)$, and denote the projection $Y=X \times \bba^r_y\to X$ by $p_{X}$. As in \cite[Section 2]{CDMO24}, we consider the sheaf of rings $\widetilde{\rs}=\ds_Y\langle w,\partial_w,\partial_w^{-1}\rangle$ and the left $\widetilde{\rs}$-module
$$
\widetilde{B}_h= \bigoplus_{j \in \bbz} \os_{Y}\partial_w^j\delta_h,
$$
where the actions of $\os_Y$, $\partial_{w}$, and $\partial^{-1}_{w}$ are the obvious ones, while the actions of derivations and $w$ are given by the following analogue of \eqref{bf action}: for all $\theta \in \sDer_{\bbc}(\os_Y)$, $p\in\os_{Y}$, and $j\in \bbz$, we have
$$
\theta \cdot p\partial_w^{j}\delta_{h} = \theta(p)\partial_w^{j}\delta_{h} - \theta(h) p \partial_w^{j+1}\delta_{h}
\ \
\text{ and }
\ \
w \cdot p\partial_w^{j}\delta_{h} =  h p \partial_t^{j}\delta_{h} - j p \partial_w^{j-1}\delta_{h}.
$$
Also consider the decreasing filtration $(V^{k}\widetilde{\rs})_{k\in \bbz}$ on $\widetilde{\rs}$ given by
$$
V^{k}\widetilde{\rs}= \bigoplus_{i-j\geq k}\ds_{Y}w^{i}\partial_{w}^{j},
$$
where the direct sum runs over $i\in\bbz_{\geq 0}$ and $j\in\bbz$ such that $i-j\geq k$. Recall that for a section $u\in \widetilde{B}_h$, the microlocal $b$-function $\widetilde{b}_u(s)$ of $u$ is the monic generator of the ideal
$$
\{b(s)\in \bbc[s] \ | \ b(s)u \in V^1\widetilde{\rs} \cdot u \}.
$$
We note that \cite[Section 3]{CDMO24} suppresses the pushforward $p_{X,*}$ and simply writes $\widetilde{B}_h$ for $p_{X,*}\widetilde{B}_h$. We will keep this distinction explicit. With this notation, we equip the $p_{X,*}\widetilde{\rs}$-module
$$
p_{X,*}\widetilde{B}_h= \bigoplus_{j \in \bbz}\bigoplus_{\alpha \in \bbz^{r}_{\geq 0}} \os_{X}y^{\alpha}\partial_w^j\delta_h
$$
with the natural $\bbz$-grading introduced in \cite[Section 3]{CDMO24}, given by $p_{X,*}\widetilde{B}_h=\bigoplus_{m\in \bbz}\widetilde{B}^{(m)}_h$, where
$$
\widetilde{B}^{(m)}_h= \bigoplus_{\alpha \in \bbz^{r}_{\geq 0}} \os_{X}y^{\alpha}\partial_w^{|\alpha|-m}\delta_h
$$
for each $m\in \bbz$. Finally, let $\phi\colon  p_{X,*}\widetilde{B}_h \to B_{\bff}$ be the unique $\os_X$-linear map such that
$$
\phi(y^{\alpha}\partial_w^j\delta_h) = \partial_t^{\alpha}\delta_{\bff}
$$
for all $\alpha\in \bbz_{\geq 0}^r$ and $j\in \bbz$, as defined in \cite[(17)]{CDMO24}. By \cite[Proposition 3.4]{CDMO24}, for every $m\in \bbz$ and every $u\in \widetilde{B}^{(m)}_h$, we have
\begin{equation}\label{Proposition 3.4}
\widetilde{b}_{u}(s-m)=b_{\phi(u)}(s).
\end{equation}

Now consider the sections $gy^{\beta}\delta_{h},gz^{\beta}\delta_{h}\in\Gamma(Y,\widetilde{B}_h)$. Because $h=\sum_{i=1}^{r}y_{i}z_{i}g_{i}$ is invariant under swapping the $y$- and $z$-variables, we have
$$
\widetilde{b}_{gy^{\beta}\delta_{h}}(s)=\widetilde{b}_{gz^{\beta}\delta_{h}}(s).
$$
Since $gy^{\beta}\delta_h\in \widetilde{B}^{(|\beta|)}_h$ and $\phi(gy^{\beta}\delta_h) = g\partial_t^{\beta}\delta_{\bff}$, by \eqref{Proposition 3.4}, we have $\widetilde{b}_{gy^{\beta}\delta_h}(s-|\beta|)=b_{g \partial_t^{\beta}\delta_{\bff}}(s)$. Also, since $gz^{\beta}\delta_h\in \widetilde{B}^{(0)}_h$ and $\phi(gz^{\beta}\delta_h) = gz^{\beta}\delta_{\bff}$, by \eqref{Proposition 3.4}, we have $\widetilde{b}_{gz^{\beta}\delta_h}(s)  = b_{g     z^{\beta}\delta_{\bff}}(s)$. Thus we obtain
$$
b_{g \partial_t^{\beta}\delta_{\bff}}(s) = \widetilde{b}_{gy^{\beta}\delta_h}(s-|\beta|)= \widetilde{b}_{gz^{\beta}\delta_h}(s-|\beta|)  = b_{g     z^{\beta}\delta_{\bff}}(s-|\beta|). \qed
$$

\noindent{\it Proof of Theorem \ref{combinatorial lower bound theorem}}. By Lemma \ref{units}, we may assume that $u =u_1 \cdots=u_{r}=1$, and by Lemma \ref{smooth pull}, we may assume that $X=\bba^{r+n}$, whose standard coordinates are given by $z_1,\dots,z_{r}, x_1,\dots,x_n$, so that $g = x^k$ and $f_i= z_i x^{a_i}$ for $1\leq i\leq r$. Consider the monomial ideal $\frak{a}=(f_1,\dots,f_r)$ and its Newton polyhedron $P(\frak{a})\subset\bbr^{r+n}$. Note that $P(\frak{a})$ is the Newton polyhedron associated to $P(a_1),\dots,P(a_r)\subset\bbr^{n}$, as defined in Section \ref{The Newton polyhedron associated to several Newton polyhedra}. Given $\beta\in\bbz^{r}_{\geq 0}$, we have $b_{g\partial_t^{\beta}\delta_{\bff}}(s)=b_{gz^{\beta}\delta_{\bff}}(s-|\beta|)$ by Lemma \ref{microlocal}. Also, by Remark \ref{howald} and Lemma \ref{explicit lct}, we have
$$
\lct_{gz^{\beta}}(\frak{a})=\widetilde{\alpha}(P(\frak{a});\beta\times k)\geq\min\{ r+|\beta|, \widetilde{\alpha}(a_1,\dots,a_r;k)\}.
$$
So by \eqref{lct b}, all the roots of $b_{g\partial_t^{\beta}\delta_{\bff}}(s)$ are 
$\leq-\min\{ r, \widetilde{\alpha}(a_1,\dots,a_r;k)-|\beta|\}$. The theorem then follows from \eqref{V roots} and Lemma \ref{min thing}. \qed

\section{Semi-quasihomogeneous complete intersections}\label{Semi-quasihomogeneous complete intersections}

Throughout this section, let $R=\bbc[x_1,\dots,x_n]$ and let $w_1,\dots,w_n$ be positive integers. We consider the grading on $R$ given by $\deg x_i=w_i$ for $1\leq i\leq n$, and we refer to homogeneous elements of $R$ with respect to this grading as weighted homogeneous. The goal of this section is to prove Theorem \ref{weighted minimal exponent}. Before we do so, we begin with the following lemma.

\begin{lemma}\label{transverse}
Let $f_1,\dots,f_r\in R$ be weighted homogeneous polynomials such that the closed subscheme $Z=V(f_1,\dots,f_r)\subset\bba^{n}$ is a complete intersection of codimension $r$ with an isolated singularity at the origin. Then the ideal generated by
$$
f_1(x_1,\dots,x_{n-1},1),\dots, f_r(x_1,\dots,x_{n-1},1) \in \bbc[x_1,\dots,x_{n-1}]
$$
defines a smooth subvariety of codimension $r$ in $\bba^{n-1}=\Spec \bbc[x_1,\dots,x_{n-1}]$.
\end{lemma}
\proof All ideals in this proof are ideals of $R$. For a matrix $A$ with entries in $R$, let $I_r(A)$ denote the ideal of $r\times r$ minors of $A$. We denote the Jacobian matrix of $f_{1},\dots,f_{r}$ by
$$
J_f = \bigg( \frac{\partial f_i}{\partial x_j} \bigg)_{1\leq i\leq r, 1\leq j\leq n}
$$
and consider the matrix $J_f'=J_f(x_1,\dots,x_{n-1},1)$ obtained from $J_f$ by setting $x_{n}=1$. We put $g_i = f_i(x_1,\dots,x_{n-1},1)\in R$ for $1\leq i\leq r$ and denote the Jacobian matrix of $g_{1},\dots,g_{r}$ in the first $n-1$ variables by
$$
J_g = \bigg( \frac{\partial g_i}{\partial x_j} \bigg)_{1\leq i\leq r, 1\leq j\leq n-1}.
$$

It suffices to show that 
$$
(g_1,\dots,g_r, I_r(J_g),x_n-1)=R.
$$ 
To begin, note that $Z \setminus \{0\}$ is smooth by weighted homogeneity. Since $f_1,\dots,f_r$ form a regular sequence, $Z\setminus \{0\} \subset\bba^{n}\setminus\{0\}$ is then a smooth subvariety of codimension $r$ and hence $(x_1,\dots,x_n)\subset\sqrt{(f_1,\dots,f_r,I_r(J_f))}$, which implies that $(f_1,\dots,f_r,I_r(J_f),x_n-1) = R$. Since $(f_1,\dots,f_r,I_r(J_f),x_n-1)=(g_1,\dots,g_r,I_r(J'_f),x_n-1)$, it thus suffices to show that
$$
(g_1,\dots,g_r, I_r(J_g))=(g_1,\dots,g_r,I_r(J'_f)).
$$
To see this, let $d_{i}=\deg f_{i}$ and $h_i=(\partial_{x_n} f_i)(x_1,\dots,x_{n-1},1)$ for $1\leq i\leq r$. Substituting $x_{n}=1$ in the equation $d_if_i = \sum_{j=1}^n w_j x_j \partial_{x_j} f_i$ yields $d_i g_i = w_n h_i + \sum_{j=1}^{n-1} w_j x_j \partial_{x_j}g_i$ for each $1\leq i\leq r$, hence the last column of $J_{f}'$ is equal to
$$
(h_i)_{i=1}^r =  (w_n^{-1}d_i g_i)_{i=1}^r - \sum_{j=1}^{n-1} (w_n^{-1}w_j x_j \partial_{x_j} g_i)_{i=1}^r.
$$
Since the submatrix consisting of the first $n-1$ columns of $J'_{f}$ is equal to $J_g$, it is then easy to see that $(g_1,\dots,g_r, I_r(J_g))=(g_1,\dots,g_r,I_r(J'_f))$, as desired. \qed

\noindent{\it Proof of Theorem \ref{weighted minimal exponent}}. Since the statement of Theorem \ref{weighted minimal exponent} is invariant under scaling the weights, we may assume that $\gcd(w_1,\dots,w_n)=1$. As in the proof of \cite[Theorem 1.2]{BBV23}, we consider the weighted blow up of $\bba^{n}$ with respect to the weights $w_1,\dots,w_n$. Let $e_1,\dots,e_n$ denote the standard basis of $\bbr^{n}$ and let $e^{*}_1,\dots,e^{*}_n\in(\bbr^{n})^{*}$ denote the corresponding dual basis. We set $w=\sum_{i=1}^n w_i e_i^*$. For each $1\leq i\leq n$, let $\sigma_{i}\subset(\bbr^{n})^{*}$ denote the cone with primitive ray generators $e^{*}_1,\dots,e^{*}_{i-1},w,e^{*}_{i+1},\dots,e^{*}_n$. Let $\Sigma$ be the fan with support $|\Sigma|=(\bbr^{n})^{*}_{\geq 0}$ whose $n$-dimensional cones are given by $\sigma_{1},\dots,\sigma_{n}$. In other words, $\Sigma$ is the star subdivision of $(\bbr^n)_{\geq 0}^*$ at $w$, as defined in \cite[page 515]{CLS11}. 

For the remainder of this proof, we use the notation introduced in Section \ref{simplicial toric resolutions} for the simplicial fan $\Sigma$. In particular, recall the setup
$$
U\xrightarrow{q} U/G\xrightarrow{p} Y\xrightarrow{\pi} \bba^{n},
$$
where $U=\coprod_{i=1}^{n}U'_{\sigma_{i}}$, $p$ is surjective and \'etale, and $\pi$ is a proper birational toric morphism. Also, recall that $\pi_{\sigma_{i}}$ denotes the morphism $U'_{\sigma_{i}}\to  \bba^{n}$ for $1\leq i\leq n$. By \cite[Proposition 11.1.6]{CLS11}, $\pi$ is projective, and by \eqref{isom locus}, $\pi$ is an isomorphism over the complement of the origin $0\in \bba^n$.

For the following claim, we say that a reindexing of $f_{1},\dots,f_{r}$ is degree-preserving if it is induced by a permutation $\rho \in S_{r}$ such that $\deg f_{\rho(i)}= d_{i}$ for each $i$.

\begin{claim}\label{weighted claim}
Given a closed point $P \in (\pi\circ p\circ q)^{-1}(0)$, there exists an open neighborhood $U_{P}$ of $P$ with global coordinates $z_{1},\dots,z_{r'},t_{1},\dots,t_{n-r'}\in\os_{U_{P}}(U_{P})$ for some nonnegative integer $r'\leq r$ such that if we denote by $\pi_{P}$ the morphism $U_{P}\to \bba^{n}$ and by $g$ the Jacobian of $\pi_{P}$, then one of the following cases holds, depending on the value of $r'$:
\begin{enumerate}
\item[i)] $r'=r$. We have $\pi_{P}^{*}f_{i}=z_{i}t_{1}^{d_{i}}$ for all $i\leq r$ and $g=ut_{1}^{w_1+\cdots+w_n-1}$ for some invertible function $u\in\os_{U_{P}}(U_{P})$. 
\item[ii)] $r'<r$. After a degree-preserving reindexing of $f_{1},\dots,f_{r}$, if necessary, we have $\pi_{P}^{*}f_{i}=z_{i}t_{1}^{d_{i}}$ for all $i\leq r'$, $\pi_{P}^{*}f_{r'+1}=vt_{1}^{d_{r'+1}}$, and $g=ut_{1}^{w_1+\cdots+w_n-1}$ for some invertible functions $u,v\in\os_{U_{P}}(U_{P})$. Moreover, $\pi_{P}^{*}f_{i}$ is contained in the ideal generated by $\pi_{P}^{*}f_{r'+1}$ for all $i> r'+1$.
\end{enumerate}
Furthermore, in both cases, $t_{1}$ (set-theoretically) defines $(\pi\circ p\circ q)^{-1}(0)\cap U_{P}$.
\end{claim}
\noindent{\it Proof of Claim \ref{weighted claim}}. We may assume without loss of generality that $P\in \pi_{\sigma_n}^{-1}(0)$. Recall that $U'_{\sigma_{n}}=\Spec \bbc[y_{1},\dots,y_{n}]$ and by \eqref{toric monomial}, the $\bbc$-algebra map $\pi_{\sigma_n}^*: R \to \bbc[y_{1},\dots,y_{n}]$ is given by $\pi_{\sigma_n}^*(x_j) = y_j y_n^{w_j}$ for $j\leq n-1$ and $\pi_{\sigma_n}^*(x_n)= y_n^{w_n}$, which implies that $y_{n}$ (set-theoretically) defines $\pi_{\sigma_n}^{-1}(0)$ and hence $y_{n}(P)=0$. For $1\leq i\leq r$, note that
$$
\pi_{\sigma_n}^* f_i(x_1,\dots,x_n) = y_n^{d_i}\widetilde{f}_i(y_1,\dots,y_n),
$$
where $\widetilde{f}_i(y_1,\dots,y_n) = f_{i,d_i}(y_1,\dots,y_{n-1},1) + y_n h_i$ for some $h_i\in \bbc[y_1,\dots,y_n]$. We define $r'$ as follows. Let $d_{0}=0$. If $\widetilde{f}_1(P)=\cdots= \widetilde{f}_{r}(P)=0$, then define $r'=r$. If $\widetilde{f}_1(P)=\cdots= \widetilde{f}_{i}(P)=0$ and $\widetilde{f}_{i+1}(P)\neq 0$ for some $i<r$, then define $r'$ as the largest nonnegative integer $j\leq i$ such that $d_{j} \neq d_{i+1}$, in which case we also perform a degree-preserving reindexing of $f_{1},\dots,f_{r}$ by swapping the indices $r'+1$ and $i+1$. If $r'<r$, then it is clear there exists an open neighborhood $B_{P}$ of $P$ such that $\pi_{P}^{*}f_{i}|_{B_{P}}$ is contained in the ideal generated by $\pi_{P}^{*}f_{r'+1}|_{B_{P}}$ for all $i> r'+1$. 

By the assumptions of the theorem, the closed subscheme $V(f_{1,d_{1}},\dots,f_{r',d_{r'}})\subset \bba^{n}$ is a complete intersection of codimension $r'$ with an isolated singularity at $0$. Thus by Lemma \ref{transverse}, the ideal generated by $\widetilde{f}_1,\dots,\widetilde{f}_{r'},y_{n}$ defines a smooth subvariety in $U_{\sigma_{n}}$ of codimension $r'+1$. This implies that there exists an open neighborhood $U_{P}$ of $P$ with global coordinates $z_{1},\dots,z_{r'},t_{1},\dots,t_{n-r'}\in\os_{U_{P}}(U_{P})$ such that $z_{i}=\widetilde{f}_{i}$ for $i\leq r'$ and $t_{1}=y_{n}$, and with $U_{P}\subset B_{P}$ if $r'<r$. Finally, by \eqref{jacobian}, we have
\begin{equation}\label{weighted Jacobian}
\pi_{\sigma_n}^*dx_1\wedge \cdots \wedge dx_n = w_n y_n^{w_1+\cdots+w_n-1} dy_1\wedge \cdots \wedge dy_n,
\end{equation}
which, after possibly shrinking $U_{P}$, implies the assertions concerning the Jacobian $g$ in both cases of the claim. \qed

Let $\widetilde{Z}$ denote the closure of $(\pi\circ p\circ q)^{-1}(Z\setminus\{0\})$ in $U$. By Claim \ref{weighted claim}, $\widetilde{Z}$ is a smooth subvariety of pure codimension $r$ in some neighborhood of $(\pi\circ p\circ q)^{-1}(0)$. Thus by Remark \ref{local setup}, there exists an open neighborhood $W$ of $0\in \bba^{n}$ and an open neighborhood $V$ of $(\pi\circ p)^{-1}(0)$ such that $p(V)=\pi^{-1}(W)$ and $\widetilde{Z}\cap q^{-1}(V)$ is a smooth subvariety of pure codimension $r$ in $q^{-1}(V)$. Note that $\pi\circ p\circ q$ is \'etale over $\bba^{n}\setminus\{0\}$ by \eqref{weighted Jacobian}. Together with the fact that $\widetilde{Z}\cap q^{-1}(V)$ surjects onto $Z\cap W$, it follows that the closed subscheme $Z\cap W\subset W$ is a complete intersection with an isolated singularity at $0$ (see \cite[Theorem 5.1]{GH78} for an analytic proof), hence $\widetilde{\alpha}_{0}(Z)$ makes sense. 

By \cite[Proposition 2.1]{CDM25}, we have the upper bound
$$
\widetilde{\alpha}_0(Z) \leq \min\{ i + \tfrac{1}{d_i}(w_1+\cdots+w_n - d_1-\cdots - d_i) \ | \ 1\leq i\leq r \}.
$$
Therefore, it suffices to prove the reverse inequality. Let $P \in (\pi\circ p\circ q)^{-1}(0)$ be a closed point. If we let $U_{P}$ be an open neighborhood of $P$ as in Claim \ref{weighted claim}, then by Corollary \ref{combinatorial lower bound first corollary} in case i) and Corollary \ref{combinatorial lower bound second corollary} in case ii), we have
$$
\widetilde{\alpha}(\pi_{P}^{*}f_{1},\dots,\pi_{P}^{*}f_{r};g)\geq \widetilde{\alpha}(d_{1},\dots,d_{r};w_1+\cdots+w_n-1).
$$
Thus by Theorem \ref{minimal exponent lower bound}, we have
$$
\widetilde{\alpha}_{0}(Z)\geq\widetilde{\alpha}(d_{1},\dots,d_{r};w_1+\cdots+w_n-1).
$$
The desired lower bound then follows from Remark \ref{dimension one}, which implies that
$$
\widetilde{\alpha}(d_{1},\dots,d_{r};w_1+\cdots+w_n-1)=\min\{ i + \tfrac{1}{d_i}(w_1+\cdots+w_n - d_1-\cdots - d_i) \ | \ 1\leq i\leq r \}. \qed
$$

\section{Newton and Khovanskii non-degeneracy}\label{Non-degeneracy}

In this section we recall the definitions of Newton and Khovanskii non-degeneracy. We also introduce notions of Newton non-degeneracy that interpolate between Newton and globally Newton non-degeneracy, in order to make precise the relationship between the Khovanskii non-degeneracy of the $r$-tuple $(f_{1},\dots,f_{r})$ and the Newton non-degeneracy of each polynomial $g_{i}=f_i + \sum_{j\neq i}^{}z_{j}f_{j}$. For $\bba^{d}=\Spec \bbc[y_1,\dots,y_d]$, we denote the standard algebraic torus in $\bba^{d}$ by
$$
\bbg^d_m=\Spec \bbc[y_1^{\pm 1},\dots,y_d^{\pm 1}] \subset \bba^d.
$$
For a polynomial $f = \sum_u c_u y^u\in \bbc[y_1,\dots,y_d]$ and a face $Q$ of the Newton polyhedron $P(f)\subset\bbr^{d}$, we write
$$
f_Q = \sum_{u\in Q} c_u y^u\in \bbc[y_1,\dots,y_d].
$$

\begin{definition}\label{Newton}
Given a nonzero polynomial $f \in \bbc[y_{1},\dots,y_d]$ with $f(0)=0$ and a subset $I\subset\{1,\dots,d\}$, we say that $f$ is $\{y_{i}\}_{i\in I}$-globally Newton non-degenerate if for every $v\in (\bbr^d)_{I}^*$, the closed subscheme $V(f_{F(v,P(f))})\cap\bbg^d_m$ is a smooth hypersurface of $\bbg^d_m$, where $(\bbr^d)_{I}^*$ is as defined in \eqref{global real}. If $I$ is empty (resp. $I=\{1,\dots,d\}$), we simply say that $f$ is Newton non-degenerate (resp. globally Newton non-degenerate).
\end{definition}

\begin{definition}\label{Khovanskii}
Given nonzero polynomials $f_1,\dots,f_r \in \bbc[y_{1},\dots,y_d]$ with $f_{i}(0)=0$ for each $i$, we say that the $r$-tuple $(f_1,\dots,f_r)$ is Khovanskii non-degenerate (resp. globally Khovanskii non-degenerate) if for every $v\in (\bbr^d)_{>0}^*$ (resp. for every $v\in (\bbr^d)_{\geq0}^*$), the closed subscheme
$$
V(f_{1, F(v, P_{1})},\dots, f_{r, F(v, P_r)})\cap\bbg^d_m\subset\bbg^d_m
$$ 
is a smooth subvariety of codimension $r$ (with the convention that the empty subvariety has any codimension), where $P_{i}=P(f_{i})$ for each $i$. 
\end{definition}

Note that a nonzero polynomial $f \in \bbc[y_{1},\dots,y_d]$ is Newton non-degenerate (resp. globally Newton non-degenerate) if and only if the $1$-tuple $(f)$ is Khovanskii non-degenerate (resp. globally Khovanskii non-degenerate).


\begin{lemma}\label{Newton implications}
Let $R=\bbc[x_1,\dots,x_n]$. Given nonzero polynomials $f_{1},\dots,f_{r}\in R$, if we set $g=\sum_{i=1}^{r}f_{i}z_{i}\in R[z_1,\dots,z_r]$ and $g_i = f_i + \sum_{j\neq i}^{}z_{j}f_{j} \in R[z_1,\dots,\widehat{z}_i,\dots,z_r]$
for $1\leq i\leq r$, then the following hold:
\begin{enumerate}
\item[ia)] If the $|I|$-tuple $(f_{i})_{i\in I}$ is Khovanskii non-degenerate for all $I\subset\{1,\dots,r\}$, then $g$ is $\{z_1,\dots,z_r\}$-globally Newton non-degenerate.
\item[ib)] If the $|I|$-tuple $(f_{i})_{i\in I}$ is globally Khovanskii non-degenerate for all $I\subset\{1,\dots,r\}$, then $g$ is globally Newton non-degenerate.
\item[iia)] If $g$ is $\{z_1,\dots,z_r\}$-globally Newton non-degenerate, then $g_{i}$ is $\{z_1,\dots,\widehat{z}_i,\dots,z_r\}$-globally Newton non-degenerate for all $i\in\{1,\dots,r\}$.
\item[iib)] If $g$ is globally Newton non-degenerate, then $g_{i}$ is globally Newton non-degenerate for all $i\in\{1,\dots,r\}$.
\end{enumerate}
\end{lemma}
\proof Let $P_{i}=P(f_{i})$ for each $i$. We denote the standard bases of $\bbr^{n}$ and $\bbr^{r+n}$ by $e_{1},\dots,e_{n}$ and $e'_{1},\dots,e'_{r},e''_{1},\dots,e''_{n}$, respectively. For $1\leq i\leq r$, we denote by $\pi_{i}\colon \bbr^{r+n}\to\bbr^{r-1+n}$ the standard projection obtained by omitting the $e'_{i}$-coordinate. 

To prove part ia), suppose that the $|I|$-tuple $(f_{i})_{i\in I}$ is Khovanskii non-degenerate for all $I\subset\{1,\dots,r\}$. Fix an element $L\in(\bbr^{r+n})^{*}_{\geq  0}$ with $L(e''_{j})>0$ for $1\leq j\leq n$, and set 
$$
D=V(g_{F(L,P(g))})\subset \bba^{r+n}.
$$
We need to show that $D\cap\bbg^{r+n}_{m}\subset\bbg^{r+n}_{m}$ is a smooth hypersurface. Let $v\in(\bbr^{n})^{*}_{> 0}$ be the element given by $v(e_{j})=L(e''_{j})$ for $1\leq j\leq n$. If we put
$$
I=\{ i\in \{1,\dots,r\} \ |\ L(e'_{i})+d(v,P_{i})=d(L ,P(g)) \},
$$
then we have
$$
g_{F(L,P(g))}=\sum_{i\in I}z_{i}f_{i,F(v,P_{i})}.
$$
Consider the closed subscheme $Z=V(f_{i,F(v,P_{i})} \ | \ i\in I)\subset\bba^{n}$. Let $J^{t}$ denote the transpose of the Jacobian matrix $J= (\partial_{x_j}(f_{i,F(v,P_{i})}))_{i\in I,1\leq j\leq n}$. For each closed point $x\in \bba^{n}$, let $W_{x}=\Ker J^{t}(x)$, which we view as a linear subspace of $\bba^{|I|}$. By \cite[Lemma 4.22]{CDMO24}, the singular locus of the hypersurface $D\subset \bba^{r+n}=(\bba^{|I|}\times\bba^{r-|I|})\times \bba^{n}$ is given by
$$
D_{sing} = \coprod_{x\in Z}  (\bba^{r-|I|}\times W_{x})\times \{x\}.
$$
Since the $|I|$-tuple $(f_{i})_{i\in I}$ is Khovanskii non-degenerate, $Z\cap \bbg^{n}_{m}\subset \bbg^{n}_{m}$ is a smooth subvariety of codimension $|I|$, which implies that $W_{x}=0$ for every closed point $x\in Z\cap\bbg^{n}_{m}$. Thus $D_{sing}$ is disjoint from $\bbg^{r+n}_{m}=(\bbg^{|I|}_{m}\times\bbg^{r-|I|}_{m})\times \bbg^{n}_{m}$, hence $D\cap\bbg^{r+n}_{m}\subset\bbg^{r+n}_{m}$ is a smooth hypersurface, as desired. The proof of part ib) is essentially the same as that of part ia), so we omit it.

To prove part iia), suppose that $g$ is $\{z_1,\dots,z_r\}$-globally Newton non-degenerate. It suffices to show that $g_{r}$ is $\{z_1,\dots,z_{r-1}\}$-globally Newton non-degenerate. Fix an element $L\in(\bbr^{r-1+n})^{*}_{\geq 0}$ such that $L(\pi_{r}(e''_{j}))>0$ for $1\leq j\leq n$. We need to show that
$$
V(g_{r,F(L,P(g_{r}))})\cap \bbg^{r-1+n}_{m}\subset \bbg^{r-1+n}_{m}
$$
is a smooth hypersurface. Consider the morphism $p:\bbg^{r+n}_{m} \to\bbg^{r-1+n}_{m}$ given by
$$
p(z_{1},\dots,z_{r},x_{1},\dots,x_{n})= (z_{1}/z_{r},\dots,z_{r-1}/z_{r},x_{1},\dots,x_{n}).
$$
Observe that $p^{*}g_{r,F(L,P(g_{r}))}=g_{F(L\circ \pi_{r},P(g))}/z_{r}$, which defines a smooth hypersurface in $\bbg^{r+n}_{m}$ since $g$ is $\{z_1,\dots,z_r\}$-globally Newton non-degenerate. Because $p$ is smooth and surjective, it follows that $g_{r,F(L,P(g_{r}))}$ defines a smooth hypersurface in $\bbg^{r-1+n}_{m}$, as desired. The proof of part iib) is essentially the same as that of part iia), so we omit it. \qed

\begin{corollary}\label{global nnd}
Let $R=\bbc[x_1,\dots,x_n]$ and let $Z \subset \bba^n$ be a closed subscheme defined by a regular sequence $f_1,\dots,f_r\in (x_{1},\dots,x_{n})^{2}$ such that for all subsets $I \subset \{1,\dots,r\}$, the $|I|$-tuple $(f_{i})_{i\in I}$ is globally Khovanskii non-degenerate. Then we have
$$
\widetilde{\alpha}(Z)=\widetilde{\alpha}(P(f_{1}),\dots,P(f_{r})).
$$
\end{corollary}
\proof By Lemma \ref{Newton implications}, $g_{i}=f_i + \sum_{j\neq i}^{}z_{j}f_{j} \in R[z_1,\dots,\widehat{z}_i,\dots,z_r]$ is globally Newton non-degenerate for each $i$, so by Remark \ref{Newton minimal exponent}, we have $\widetilde{\alpha}(g_{i})=\widetilde{\alpha}(P(g_{i}))$. The corollary then follows from Lemma \ref{minimal exponent in terms of hypersurface} and Remark \ref{comb conn}. \qed

\section{Khovanskii non-degenerate complete intersections}\label{Khovanskii non-degenerate complete intersections}

This section is devoted to the proof of Theorem \ref{Khovanskii minimal exponent}. We begin with the following lemma.

\begin{lemma}\label{Newton minimal exponent lower bound}
Let $R=\bbc[x_1,\dots,x_n]$. Given a subset $I\subset\{1,\dots,n\}$ and a $\{x_{i}\}_{i\in I}$-globally Newton non-degenerate polynomial $f\in R$, we have
$$
\widetilde{\alpha}_{z}(f)\geq \widetilde{\alpha}(P(f))
$$
for every closed point $z\in V(f)\cap V(x_{i}\ |\ i\notin I)$.
\end{lemma}
\proof Let $e_1,\dots,e_n$ denote the standard basis of $\bbr^{n}$ and let $e^{*}_1,\dots,e^{*}_n\in(\bbr^{n})^{*}$ denote the corresponding dual basis. We denote the normal fan of $P(f)$ by $\Sigma$ and the toric variety associated to $\Sigma$ by $Y_{0}$. The induced toric morphism $Y_{0}\to \bba^{n}$ is projective by \cite[Theorem 7.1.10]{CLS11}. By \cite[Proposition 11.1.7]{CLS11}, there exists a simplicial refinement $\Sigma'$ of $\Sigma$ such that $\Sigma'$ and $\Sigma$ have the same set of rays and the induced toric morphism $Y\to Y_{0}$ is projective, where $Y$ is the toric variety associated to $\Sigma'$. Thus, if we denote the composition $Y \to Y_{0}\to\bba^{n}$ by $\pi$, then $\pi$ is projective. Note that $\pi\colon Y\to\bba^{n}$ is the toric morphism associated to $\Sigma'$, as defined in Section \ref{simplicial toric resolutions}. By Lemma \ref{isom lemma}, $\pi$ is an isomorphism over the complement of $V(f)$ since $\Sigma'$ has the same set of rays as $\Sigma$. 

For the remainder of this proof, we use the notation introduced in Section \ref{simplicial toric resolutions} for the simplicial fan $\Sigma'$. In particular, recall the setup
$$
U\xrightarrow{q} U/G\xrightarrow{p} Y\xrightarrow{\pi} \bba^{n},
$$
where the disjoint union $U=\coprod_{\sigma}U'_{\sigma}$ runs over the $n$-dimensional cones $\sigma$ of $\Sigma'$, and $p$ is surjective and \'etale. Also, for an $n$-dimensional cone $\sigma$ of $\Sigma'$, recall that $\pi_{\sigma}$ denotes the morphism $U'_{\sigma}\to \bba^{n}$.

\begin{claim}\label{Newton claim}
Given a closed point $P \in (\pi\circ p\circ q)^{-1}(V(x_{i}\ |\ i\notin I))$, there exists an $n$-dimensional cone $\sigma\in \Sigma'$ with primitive ray generators $v_{1},\dots,v_{n}$, an open neighborhood $U_{P}$ of $P$ with global coordinates $z_{1},\dots,z_{n}\in\os_{U_{P}}(U_{P})$, an integer $\delta\in \{0,1\}$, and a nonnegative integer $l\leq n- \delta$, such that if we denote by $\pi_{P}$ the morphism $U_{P}\to \bba^{n}$ and by $g$ the Jacobian of $\pi_{P}$, then one of the following cases holds, depending on the value of $\delta$:
\begin{enumerate}
\item[i)] $\delta=1$. We have $\pi_{P}^{*}f=z_{1}\prod_{j=2}^{l+1} z_j^{d(v_j,P_i)}$ and $g=u\prod_{j=2}^{l+1} t_j^{|v_{j}|-1}$ for some invertible function $u\in\os_{U_{P}}(U_{P})$. 
\item[ii)] $\delta=0$. We have $\pi_{P}^{*}f=v\prod_{j=1}^l z_j^{d(v_j,P_i)}$ and $g=u\prod_{j=1}^l z_j^{|v_{j}|-1}$ for some invertible functions $u,v\in\os_{U_{P}}(U_{P})$.
\end{enumerate}
\end{claim}

\noindent{\it Proof of Claim \ref{Newton claim}}. Let $\sigma\in \Sigma'$ be an $n$-dimensional cone such that $P\in \pi^{-1}_{\sigma}(0)$, and denote the primitive ray generators of $\sigma$ by $v_{1},\dots,v_{n}$. Recall that $U'_{\sigma}=\Spec \bbc[y_{1},\dots,y_{n}]$, with $y_{i}$ corresponding to $v_{i}$ for each $i$. By Lemma \ref{exceptional}, there exists a face $\tau$ of $\sigma$ such that $\relint(\tau)\cap(\bbr^{n})^{*}_{I}\neq \emptyset$ and $P\in O(\tau)$. By reindexing $v_{1},\dots,v_{n}$ (and $y_{1},\dots,y_{n}$ accordingly), we may assume that the primitive ray generators of $\tau$ are given by $v_{1},\dots,v_{l}$ for some $l\leq n$, so that 
$$
O(\tau)=0\times\bbg^{n-l}_{m},
$$
where $0\in\bba^{l}=\Spec\bbc[y_{1},\dots,y_{l}]$. In particular, we have
$$
y_{1}(P)=\cdots=y_{l}(P)=0.
$$
For $1\leq i\leq r$, we write $f=\sum_u c_{u}x^{u}$ and put
$$
\widetilde{f}(y_1,\dots,y_n) =  \sum_u c_{u} \prod_{j>l} y_j^{v_j(u)}  \prod_{j=1}^l y_j^{v_j(u)-d(v_j, P(f))}\in \bbc[y_1,\dots,y_n].
$$
By \eqref{toric monomial}, we have
$$
\pi_{\sigma}^{*}f = \sum_u c_{u} \prod_{j=1}^n y_j^{v_j(u)}
= \widetilde{f}(y_1,\dots,y_n)\prod_{j=1}^l y_j^{d(v_j, P(f))}.
$$
We define $\delta\in \{0,1\}$ as follows. If $\widetilde{f}(P)=0$, then define $\delta=1$. If $\widetilde{f}(P)\neq 0$, then define $\delta=0$.

We first assume that $\delta=0$. Since $\widetilde{f}(P)\neq 0$, there exists an open neighborhood $U_{P}$ of $P$ such that $v=\widetilde{f}(y_1,\dots,y_n)|_{U_{P}}$ is invertible. Thus, if we set $z_j=y_j|_{U_{P}}$ for $1\leq j\leq n$, then $z_{1},\dots,z_{n}$ are global coordinates on $U_{P}$ such that $(\pi_{\sigma}^{*}f)|_{U_{P}}=v\prod_{j=1}^l z_j^{d(v_j, P(f))}$. After possibly shrinking $U_{P}$, the assertion concerning the Jacobian $g$ follows from \eqref{jacobian}. 

Next, we assume that $\delta=1$. Fix some element $v\in \relint(\tau)\cap (\bbr^{n})^{*}_{I}$. Since $\pi_{\sigma}\colon \bbg^{n}_{m}\to\bbg^{n}_{m}$ is \'etale and $f$ is $\{x_{i}\}_{i\in I}$-globally Newton non-degenerate, the closed subscheme 
$$
V(\pi^{*}_{\sigma}f_{F(v,P(f))})\cap\bbg^{n}_{m}\subset\bbg^{n}_{m}
$$
is a smooth hypersurface. By Lemma \ref{Newton pull}, we have
$$
\pi^{*}_{\sigma}(f_{F(v,P(f))})=\widetilde{f}(0,\dots,0, y_{l+1},\dots,y_{n})\prod_{j=1}^l y_j^{d(v_j, P(f))}.
$$
Therefore, if we view $\widetilde{f}(0,\dots,0, y_{l+1},\dots,y_{n})\in \bbc[y_{l+1},\dots,y_{n}]$ as a regular function on $\bba^{n-l}$, then
$$
V(\widetilde{f}(0,\dots,0, y_{l+1},\dots,y_{n}))\cap\bbg^{n-l}_{m}\subset\bbg^{n-l}_{m}
$$
is a hypersurface, which implies that the ideal generated by $\widetilde{f}(y_1,\dots,y_n),y_{1},\dots,y_{l}$ defines a smooth subvariety in $\bba^{l}\times\bbg^{n-l}$ of codimension $1+l$. Thus there exists an open neighborhood $U_{P}$ of $P$ with global coordinates $z_{1},\dots,z_{n}\in\os_{U_{P}}(U_{P})$ such that $z_{1}=\widetilde{f}(y_1,\dots,y_n)$ and $z_{j}=y_{j-1}$ for $2\leq j\leq l+1$. After possibly shrinking $U_{P}$, the assertion concerning the Jacobian $g$ follows from \eqref{jacobian}. 
\qed

By Lemma \ref{minimal exponent hypersurface upper bound}, we have the upper bound
$$
\widetilde{\alpha}_{z}(f)\leq \widetilde{\alpha}(P(f)).
$$ 
Therefore, it suffices to prove the reverse inequality. Let $P \in (\pi\circ p\circ q)^{-1}(z)$ be a closed point. If we let $U_{P}$ be an open neighborhood of $P$ as in Claim \ref{Newton claim}, then by Corollary \ref{combinatorial lower bound first corollary} in case i) and Corollary \ref{combinatorial lower bound second corollary} in case ii), we have
$$
\widetilde{\alpha}(\pi_{P}^{*}f;g)\geq \widetilde{\alpha}( (d(v_{j},P(f)))^{l}_{j=1}; (|v_{j}|-1)^{l}_{j=1} ).
$$
For each cone $\tau\in \Sigma'$, we denote by $[\tau]$ the set of primitive ray generators of $\tau$, and we equip $[\tau]$ with some ordering. Then by Theorem \ref{minimal exponent lower bound}, we have
$$
\widetilde{\alpha}_{z}(f) \geq \min_{\tau\in\Sigma'} \widetilde{\alpha}( (d(v,P(f)))_{v\in [\tau]}; (|v|-1)_{v\in [\tau]} ),
$$
The desired lower bound $\widetilde{\alpha}_{z}(f)\geq \widetilde{\alpha}(P(f))$ then follows from Theorem \ref{minimal exponent fan}. \qed

\noindent{\it Proof of Theorem \ref{Khovanskii minimal exponent}}. By Remark \ref{finite minimal exponent several newton poly}, we have
$$
\widetilde{\alpha}(P(f_{1}),\dots,P(f_{r}))=1/ \min\bigg\{t>0 \ \big| \ (t,\dots,t) \in \bigcap_{i=1}^r \pi_i^{-1}(\pi_i(P(g)))\bigg\},
$$
hence it suffices to show that $\widetilde{\alpha}_{0}(Z)=\widetilde{\alpha}(P(f_{1}),\dots,P(f_{r}))$
By Corollary \ref{minimal exponent upper bound}, we have the upper bound $\widetilde{\alpha}_{0}(Z)\leq \widetilde{\alpha}(P(f_{1}),\dots,P(f_{r}))$. By Lemma \ref{minimal exponent in terms of hypersurface}, we have
$$
\widetilde{\alpha}_{0}(Z)=\min_{1\leq i\leq r}\min_{u_{i}\in U_{i}} \widetilde{\alpha}_{(0,u_{i})}(g_i).
$$
Since $g_{i}$ is $\{z_1,\dots,\widehat{z}_i,\dots,z_r\}$-globally Newton non-degenerate by Lemma \ref{Newton implications}, we obtain $\widetilde{\alpha}_{(0,u_{i})}(g_i)\geq \widetilde{\alpha}(P(g_i))$ for each $i$ and each $u_{i}\in U_{i}$ by Lemma \ref{Newton minimal exponent lower bound}, hence
$$
\widetilde{\alpha}_{0}(Z)\geq \min_{1\leq i\leq r}\widetilde{\alpha}(P(g_i))=\widetilde{\alpha}(P(f_{1}),\dots,P(f_{r})).\qed
$$
We record the following interesting consequence of

\begin{remark}
Consider the local Bernstein-Sato polynomial $b_{f,0}(s)$ of $f$, which is given by $b_{f,0}(s)=b_{f|_{U}}(s)$ for all sufficiently small open neighborhoods $U$ of $0\in\bba^{n}$. If $f$ is Newton non-degenerate, then by the proof of Lemma \ref{Newton minimal exponent lower bound}, Remark \ref{local setup}, and \cite[Theorem 1.1]{Lee26}, every root of $b_{f,0}(s)$ is either a negative integer or of the form 
$$
\frac{|v_Q|+\ell}{d(v_Q,P(f))}
$$
for some facet $Q$ of $P(f)$ not contained in any coordinate hyperplane and some integer $\ell\geq 0$, where $v_Q$ is the primitive ray generator of $\bbr_{\geq 0} \cdot L_Q$.
\end{remark}

\section{Khovanskii non-degenerate complete intersections with nested Newton polyhedra}\label{Nested non-degenerate complete intersections}

The goal of this section is to prove Theorem \ref{nested Khovanskii minimal exponent}. We begin with the following two lemmas. The second shows how condition i) of Theorem \ref{nested Khovanskii minimal exponent} will be used in the proof of the theorem.

\begin{lemma}\label{positive equivalence}
Let $P_1,\dots,P_r \subset \bbr^n$ be Newton polyhedra that do not contain the origin. Then the following are equivalent:
\begin{enumerate}
\item[i)] For all $i,j\in \{1,\dots,r\}$, there exists a real number $\epsilon>0$ such that $P_{i}\subset \epsilon\cdot P_{j}$.
\item[ii)] For all $i,j\in \{1,\dots,r\}$ and all facets $Q$ of $P_j$ that are not contained in any coordinate hyperplane, we have $d(L_Q,P_i)>0$. 
\end{enumerate}
\end{lemma}
\proof Fix $i,j\in \{1,\dots,r\}$. It suffices to show that there exists a real number $\epsilon>0$ such that $P_{i}\subset \epsilon\cdot P_{j}$ if and only if we have $d(L_Q,P_i)>0$ for all facets $Q$ of $P_j$ that are not contained in any coordinate hyperplane. If there exists a real number $\epsilon>0$ such that $P_{i}\subset \epsilon\cdot P_{j}$, then for all facets $Q$ of $P_j$ that are not contained in any coordinate hyperplane, we have $d(L_{Q},P_{i})\geq d(L_{Q},\epsilon\cdot P_{j})=\epsilon$. Conversely, assume that for all $i,j\in \{1,\dots,r\}$ and all facets $Q$ of $P_j$ that are not contained in any coordinate hyperplane, we have $d(L_Q,P_i)>0$. Set $\epsilon=\min_{Q}d(L_Q,P_i)$, where the minimum runs over the facets $Q$ of $P_{j}$ that are not contained in any coordinate hyperplane. Then $\epsilon>0$ and $d(L_{Q},P_{i})\geq \epsilon=d(L_{Q},\epsilon\cdot P_{j})$ for all facets $Q$ of $P_{j}$ that are not contained in any coordinate hyperplane, which implies that $d(L_{Q},\epsilon^{-1}\cdot P_{i})\geq 1$. Thus by \eqref{half space intersection}, we obtain $\epsilon^{-1}\cdot P_{i}\subset P_{j}$. \qed

\begin{lemma}\label{positive isomorphism}
Let $P_1,\dots,P_r \subset \bbr^n$ be Newton polyhedra that do not contain the origin such that for all $i,j\in \{1,\dots,r\}$, there exists a real number $\epsilon>0$ such that $P_{i}\subset \epsilon P_{j}$, and let $e_1,\dots,e_n$ denote the standard basis of $\bbr^{n}$ and let $e^{*}_1,\dots,e^{*}_n\in(\bbr^{n})^{*}$ denote the corresponding dual basis. If we denote the normal fan of the Minkowski sum $\sum_{i=1}^r P_i$ by $\Sigma$, then for every primitive ray generator $v$ of $\Sigma$ other than $e_{1}^{*},\dots,e_{n}^{*}$, we have $d(v,P_{i})>0$ for all $i$.
\end{lemma}
\proof Let $\Sigma_{i}$ denote the normal fan of $P_{i}$ for each $i$. By \eqref{cones normal fan sum}, there exists a cone $\sigma_{i}\in\Sigma_{i}$ for each $i$ such that $\bbr_{>0} \cdot v=\bigcap_{i=1}^{r}\relint\sigma_{i}$. If the primitive ray generators of $\sigma_{i}$ are contained in $\{e_{1}^{*},\dots,e_{n}^{*}\}$ for all $i$, then $v\in\{e_{1}^{*},\dots,e_{n}^{*}\}$, a contradiction. Thus there exists an index $i_{0}\in\{1,\dots,r\}$ such that $\sigma_{i_{0}}$ has a primitive ray generator other than $e_{1}^{*},\dots,e_{n}^{*}$. Let $v_{1},\dots,v_{l}$ be the primitive ray generators of $\sigma_{i_{0}}$ with $v_{1}\notin\{e_{1}^{*},\dots,e_{n}^{*}\}$. By \eqref{rays}, $P_{i_{0}}$ has a facet $Q$ not contained in any coordinate hyperplane such that $v_{1}=L_{Q}$, and by \eqref{relative interior}, we can write $v=\sum_{j=1}^{l}a_{j}v_{j}$ with each $a_{j}>0$. Thus by Lemma \ref{positive equivalence}, we obtain 
$$
d(v,P_{i})\geq d(a_{1}v_{1},P_{i})=a_{1}d(L_{Q},P_{i})>0
$$
for all $i$, as desired. \qed

In the following proof, we use a toric construction associated to a simplicial refinement of the Minkowski sum of the Newton polyhedra of the $f_{i}$'s. In the case of a smooth refinement, this construction appears in \cite[]{Kho77}, \cite[]{Var77}, \cite[]{Mor84}, \cite[]{Oka90}, and \cite[]{Oka97} (see also \cite[]{VZG08}), and in the hypersurface case  when $r=1$, this appears in \cite[]{Den95}, \cite[]{BN20}, and \cite[]{AQ24}.

\noindent{\it Proof of Theorem \ref{nested Khovanskii minimal exponent}}. Let $P_{i}=P(f_{i})\subset\bbr^{n}$ for each $i$. Recall that $\Sigma_{0}$ is a fan that refines the normal fan of $P_{i}$ for each $i$. Also, recall that $d(v,f) =\min\{v(u) \ |\ u\in \supp f\}$ for $v\in (\bbr^{n})^{*}_{\geq 0}$ and nonzero $f\in R=\bbc[x_{1},\dots,x_{n}]$. Since $d(v,f_{i})=d(v,P_{i})$ for $v\in (\bbr^{n})^{*}_{\geq 0}$  and $1\leq i\leq r$, by Theorem \ref{minimal exponent nested polyhedra}, we have
$$
\widetilde{\alpha}(P_{1},\dots,P_{r})=\min_{v}
\min\{ i + \tfrac{1}{d(v,f_{i})}(|v| - d(v,f_{1})-\cdots - d(v,f_{i})) \ | \ 1\leq i\leq r,d(v,f_{i})\neq 0 \},
$$
where the outer minimum runs over the primitive ray generators $v$ of $\Sigma_{0}$. Also, by Remark \ref{finite minimal exponent several newton poly}, we have
$$
\widetilde{\alpha}(P_{1},\dots,P_{r})=1/ \min\bigg\{t>0 \ \big| \ (t,\dots,t) \in \bigcap_{i=1}^r \pi_i^{-1}(\pi_i(P(g)))\bigg\}.
$$
Thus it suffices to show that $\widetilde{\alpha}_{0}(Z)=\widetilde{\alpha}(P_{1},\dots,P_{r})$.

Let $e_1,\dots,e_n$ denote the standard basis of $\bbr^{n}$ and let $e^{*}_1,\dots,e^{*}_n\in(\bbr^{n})^{*}$ denote the corresponding dual basis. We denote the normal fan of the Minkowski sum $\sum_{i=1}^r P_i$ by $\Sigma$ and the toric variety associated to $\Sigma$ by $X$. The induced toric morphism $X\to \bba^{n}$ is projective by \cite[Theorem 7.1.10]{CLS11}. By \cite[Proposition 11.1.7]{CLS11}, there exists a simplicial refinement $\Sigma'$ of $\Sigma$ such that $\Sigma'$ and $\Sigma$ have the same set of rays and the induced toric morphism $Y\to X$ is projective, where $Y$ is the toric variety associated to $\Sigma'$. Thus, if we denote the composition $Y \to X\to\bba^{n}$ by $\pi$, then $\pi$ is projective. Note that $\pi\colon Y\to\bba^{n}$ is the toric morphism associated to $\Sigma$, as defined in Section \ref{simplicial toric resolutions}.

For the remainder of this proof, we use the notation introduced in Section \ref{simplicial toric resolutions} for the simplicial fan $\Sigma'$. In particular, recall the setup
$$
U\xrightarrow{q} U/G\xrightarrow{p} Y\xrightarrow{\pi} \bba^{n},
$$
where the disjoint union $U=\coprod_{\sigma}U'_{\sigma}$ runs over the $n$-dimensional cones $\sigma$ of $\Sigma'$, and $p$ is surjective and \'etale. Also, for an $n$-dimensional cone $\sigma$ of $\Sigma$, recall that $\pi_{\sigma}$ denotes the morphism $U'_{\sigma}\to \bba^{n}$. 

We begin by observing that $\pi$ is an isomorphism over the complement of $Z$. To see this, let $v\in(\bbr^{n})^{*}_{\geq 0}$ be a primitive ray generator of $\Sigma'$ other than $e_{1}^{*},\dots,e_{n}^{*}$. By Lemma \ref{isom lemma}, it suffices to show that $d(v,P_{i})>0$ for each $i$. This follows from Lemma \ref{positive isomorphism} since $\Sigma$ has the same set of rays as $\Sigma$.

For the following claim, we say that a reindexing of $f_{1},\dots,f_{r}$ is Newton-preserving if it is induced by a permutation $\rho \in S_{r}$ such that $\deg f_{\rho(i)}= P_{i}$ for each $i$.

\begin{claim}\label{Khovanskii claim}
Given a closed point $P \in (\pi\circ p\circ q)^{-1}(0)$, there exists an $n$-dimensional cone $\sigma\in \Sigma'$ with primitive ray generators $v_{1},\dots,v_{n}$, an open neighborhood $U_{P}$ of $P$ with global coordinates $z_{1},\dots,z_{r'},t_{1},\dots,t_{n-r'}\in\os_{U_{P}}(U_{P})$ for some nonnegative integer $r'\leq r$, and a nonnegative integer $l\leq n- r'$, such that if we denote by $\pi_{P}$ the morphism $U_{P}\to \bba^{n}$ and by $g$ the Jacobian of $\pi_{P}$, then one of the following cases holds, depending on the value of $r'$:
\begin{enumerate}
\item[i)] $r'=r$. We have $\pi_{P}^{*}f_{i}=z_{i}\prod_{j=1}^l t_j^{d(v_j,P_i)}$ for $i\leq r$ and $g=u\prod_{j=1}^l t_j^{|v_{j}|-1}$ for some invertible function $u\in\os_{U_{P}}(U_{P})$. 
\item[ii)] $r'<r$. After a Newton-preserving reindexing of $f_{1},\dots,f_{r}$, if necessary, we have $\pi_{P}^{*}f_{i}=z_{i}\prod_{j=1}^l t_j^{d(v_j,P_i)}$ for $i\leq r'$, $\pi_{P}^{*}f_{r'+1}=v\prod_{j=1}^l t_j^{d(v_j,P_{r'+1})}$, and $g=u\prod_{j=1}^l t_j^{|v_{j}|-1}$ for some invertible functions $u,v\in\os_{U_{P}}(U_{P})$. Moreover, $\pi_{P}^{*}f_{i}$ is contained in the ideal generated by $\pi_{P}^{*}f_{r'+1}$ for $i> r'+1$.
\end{enumerate}
Furthermore, in both cases, $(\pi\circ p\circ q)^{-1}(0)\cap U_{P}$ is (set-theoretically) defined by the ideal generated by $t_{1},\dots,t_{l}$.
\end{claim}
\noindent{\it Proof of Claim \ref{Khovanskii claim}}. Let $\sigma\in \Sigma'$ be an $n$-dimensional cone such that $P\in \pi^{-1}_{\sigma}(0)$, and denote the primitive ray generators of $\sigma$ by $v_{1},\dots,v_{n}$. Recall that $U'_{\sigma}=\Spec \bbc[y_{1},\dots,y_{n}]$, with $y_{i}$ corresponding to $v_{i}$ for each $i$. By Lemma \ref{exceptional}, there exists a face $\tau$ of $\sigma$ such that $\relint(\tau)\cap(\bbr^{n})^{*}_{>0}\neq \emptyset$ and $P\in O(\tau)$. By reindexing $v_{1},\dots,v_{n}$ (and $y_{1},\dots,y_{n}$ accordingly), we may assume that the primitive ray generators of $\tau$ are given by $v_{1},\dots,v_{l}$ for some $l\leq n$, so that 
$$
O(\tau)=0\times\bbg^{n-l}_{m},
$$
where $0\in\bba^{l}=\Spec\bbc[y_{1},\dots,y_{l}]$. In particular, we have
$$
y_{1}(P)=\cdots=y_{l}(P)=0.
$$
For $1\leq i\leq r$, we write $f_{i}=\sum_u c_{iu}x^{u}$ and put
$$
\widetilde{f}_{i}(y_1,\dots,y_n) =  \sum_u c_{iu} \prod_{j>l} y_j^{v_j(u)}  \prod_{j=1}^l y_j^{v_j(u)-d(v_j, P_{i})}\in \bbc[y_1,\dots,y_n].
$$
By \eqref{toric monomial}, we have
$$
\pi_{\sigma}^{*}f_{i} = \sum_u c_{iu} \prod_{j=1}^n y_j^{v_j(u)}
= \widetilde{f}_{i}(y_1,\dots,y_n)\prod_{j=1}^l y_j^{d(v_j, P_{i})}.
$$
We define $r'$ as follows. Let $P_{0}=\bbr^{n}_{\geq 0}$, and note that $P_{0}\neq P_{i}$ for $1\leq i\leq r$ since $f_{i}(0)=0$. If $\widetilde{f}_1(P)=\cdots= \widetilde{f}_{r}(P)=0$, then define $r'=r$. If $\widetilde{f}_1(P)=\cdots= \widetilde{f}_{i}(P)=0$ and $\widetilde{f}_{i+1}(P)\neq 0$ for some $i<r$, then define $r'$ as the largest nonnegative integer $j\leq i$ such that $P_{j}\neq P_{i+1}$, in which case we also perform a Newton-preserving reindexing of $f_{1},\dots,f_{r}$ by swapping the indices $r'+1$ and $i+1$. Since $P_{1}\supset \cdots \supset P_{r}$, we have $d(v,P_{1})\leq \cdots\leq  d(v,P_{r})$ for all $v\in (\bbr^{n})^{*}_{\geq 0}$, which implies that $\prod_{j=1}^l y_j^{d(v_j, P_{i})}$ divides $\prod_{j=1}^l y_j^{d(v_j, P_{i+1})}$ for each $i\leq r-1$. Thus, if $r'<r$, then there exists an open neighborhood $B_{P}$ of $P$ such that $\pi_{P}^{*}f_{i}|_{B_{P}}$ is contained in the ideal generated by $\pi_{P}^{*}f_{r'+1}|_{B_{P}}$ for all $i> r'+1$. 

Next, fix some element $v\in \relint(\tau)\cap (\bbr^{n})^{*}_{>0}$. Since $\pi_{\sigma}\colon \bbg^{n}_{m}\to\bbg^{n}_{m}$ is \'etale, by the assumptions of the theorem, the closed subscheme
$$
V(\pi^{*}_{\sigma}f_{1,F(v,P_{1})},\dots,\pi^{*}_{\sigma}f_{r',F(v,P_{r'})})\cap\bbg^{n}_{m}\subset\bbg^{n}_{m}
$$
is a smooth subvariety of codimension $r'$. By Lemma \ref{Newton pull}, we have
$$
\pi^{*}_{\sigma}(f_{i,F(v,P_{i})})=\widetilde{f}_{i}(0,\dots,0, y_{l+1},\dots,y_{n})\prod_{j=1}^l y_j^{d(v_j, P_{i})}.
$$
Therefore, if we view $\widetilde{f}_{i}(0,\dots,0, y_{l+1},\dots,y_{n})\in \bbc[y_{l+1},\dots,y_{n}]$ as a regular function on $\bba^{n-l}$ for each $i$, then
$$
V(\widetilde{f}_{i}(0,\dots,0, y_{l+1},\dots,y_{n})\ |\ 1\leq i\leq r')\cap\bbg^{n-l}_{m}\subset\bbg^{n-l}_{m}
$$
is a smooth subvariety of codimension $r'$, which implies that the ideal generated by
$$
\widetilde{f}_1(y_1,\dots,y_n),\dots,\widetilde{f}_{r'}(y_1,\dots,y_n),y_{1},\dots,y_{l}
$$
defines a smooth subvariety in $\bba^{l}\times\bbg^{n-l}$ of codimension $r'+l$. Thus there exists an open neighborhood $U_{P}$ of $P$ with global coordinates $z_{1},\dots,z_{r'},t_{1},\dots,t_{n-r'}\in\os_{U_{P}}(U_{P})$ such that $z_{i}=\widetilde{f}_{i}(y_1,\dots,y_n)$ for $i\leq r'$ and $t_{j}=y_{j}$ for $j\leq l$, and with $U_{P}\subset B_{P}$ if $r'<r$. Finally, after possibly shrinking $U_{P}$, the assertions concerning the Jacobian $g$ in both cases of the claim follow from \eqref{jacobian}. 
\qed

By Corollary \ref{minimal exponent upper bound}, we have the upper bound
$$
\widetilde{\alpha}_{0}(Z)\leq \widetilde{\alpha}(P_{1},\dots,P_{r}).
$$ 
Therefore, it suffices to prove the reverse inequality. Let $P \in (\pi\circ p\circ q)^{-1}(0)$ be a closed point. If we let $U_{P}$ be an open neighborhood of $P$ as in Claim \ref{Khovanskii claim}, then by Corollary \ref{combinatorial lower bound first corollary} in case i) and Corollary \ref{combinatorial lower bound second corollary} in case ii), we have
$$
\widetilde{\alpha}(\pi_{P}^{*}f_{1},\dots,\pi_{P}^{*}f_{r};g)\geq \widetilde{\alpha}( (d(v_{j},P_1))^{l}_{j=1} ,\dots, (d(v_{j},P_r))^{l}_{j=1}; (|v_{j}|-1)^{l}_{j=1} ).
$$
For each cone $\tau\in \Sigma'$, we denote by $[\tau]$ the set of primitive ray generators of $\tau$, and we equip $[\tau]$ with some ordering. Then by Theorem \ref{minimal exponent lower bound}, we have
$$
\widetilde{\alpha}_{0}(Z) \geq \min_{\tau\in\Sigma'} \widetilde{\alpha}( (d(v,P_1))_{v\in [\tau]} ,\dots, (d(v,P_r))_{v\in [\tau]}; (|v|-1)_{v\in [\tau]} ),
$$
The desired lower bound $\widetilde{\alpha}_{0}(Z)\geq \widetilde{\alpha}(P_{1},\dots,P_{r})$ then follows from Theorem \ref{minimal exponent fan}. \qed

The following example is inspired by the example in \cite[Example 6.10]{BA07}. 

\begin{example}\label{nested newton polygons example}
Let $e_{1},e_{2},e_{3}$ denote the standard basis of $\bbr^{3}$. We identify $(\bbr^{3})^{*}$ with $\bbr^{3}$ via the standard evaluation pairing. Consider the polynomials 
$$
f_1=x^2+y^2+z^{2}
\ \ \ \ \text{and}
\ \ \ \
f_2 = x^6+2x^4y+x^2y^2+y^6+z^{6}
$$ 
and set $Z=V(f_{1},f_{2})\subset\bba^{3}$ and $P_{i}=P(f_{i})\subset  \bbr^{3}$ for $i=1,2$. Note that $Z$ is of codimension $2$. Indeed, $V(f_1)$ is irreducible, and $V(f_1)$ is not contained in $V(f_{2})$ since $f_{1}(\sqrt{-1},1,0)=0$ and $f_{2}(\sqrt{-1},1,0)=1$. It is easy to see that $f_{1}$ and the $2$-tuple $(f_{1},f_{2})$ are globally Khovanskii Newton non-degenerate. On the other hand, $f_2$ is not Newton non-degenerate, since if we put $Q=F((1,2,2),P_{2})$, then $f_{2,Q} = x^6+2x^4y+x^2y^2$ and
$$
(\sqrt{-1},1,1) \in V\bigg(f_{2,Q}, \frac{\partial f_{2,Q}}{\partial x}, \frac{\partial f_{2,Q}}{\partial y},\frac{\partial f_{2,Q}}{\partial z} \bigg).
$$
Let $\Sigma_{i}$ denote the normal fan of $P_{i}$ for $i=1,2$, and set $v=(1,2,1)$ and $w= (2,1,1)$. For $j=1,2,3$, let $\tau_{j}$ denote the $2$-dimensional cone of $\Sigma_{1}$ whose set of primitive ray generators is given by $\{(1,1,1),e_{j}\}$, and note that $w\in \relint\tau_{1}$ and $v\in \relint\tau_{2}$. Observe that $\Sigma_{2}$ has exactly four $3$-dimensional cones $\sigma_{1}$, $\sigma_{2}$, $\sigma_{3}$, and $\sigma_{4}$, whose sets of primitive ray generators are given by 
$$
\text{$\{v,w,e_{3}\}$, $\{v,w,e_{1},e_{2}\}$, $\{w,e_{1},e_{3}\}$, and $\{v,e_{2},e_{3}\}$,}
$$
respectively. Also, note that $(1,1,1)\in  \relint\sigma_{1}$. Let $\Sigma_{0}$ denote the fan obtained by taking the star subdivision of $\Sigma_{2}$ at $(1,1,1)$, as defined in \cite[page 515]{CLS11}. Then $\Sigma_{0}$ refines both $\Sigma_{1}$ and $\Sigma_{2}$, and the primitive ray generators of $\Sigma_{0}$ are given by $(1,1,1),v,w,e_{1},e_{2},e_{3}$. Since $d((1,1,1),P_{1})=2$, $d((1,1,1),P_{2})=4$, $d(v,P_{1})=d(w,P_{1})=2$, $d(v,P_{2})=d(w,P_{2})=6$, and $d(e_{j},P_{i})=0$ for all $i,j$, by Theorem \ref{nested Khovanskii minimal exponent}, we obtain
$$
\widetilde{\alpha}_{0}(Z)=\min\{1+\tfrac{3-2}{2},2+\tfrac{3-2-4}{4},1+\tfrac{4-2}{2},2+\tfrac{4-2-6}{6}\}=\tfrac{5}{4}.
$$
Finally, we explain why the calculation of $\widetilde{\alpha}_{0}(Z)$ does not naturally reduce to the calculation of the minimal exponent of a Newton non-degenerate hypersurface. Set 
$$
g_{1}=f_1+z_{2}f_2
\ \ \ \ \text{and}\ \ \ \
g_{2}=z_{1}f_1+f_2.
$$ 
By Theorem \ref{minimal exponent nested polyhedra}, we have
$$
\widetilde{\alpha}_{0}(Z)=\widetilde{\alpha}(P_{1},P_{2})=\min\{\widetilde{\alpha}(P(g_{1})),\widetilde{\alpha}(P(g_{2}))\}.
$$
Note that $P(g_{1})=\bbr_{\geq 0}\times P(f_{1})$ and $\widetilde{\alpha}(P(g_{1}))=\tfrac{3}{2}$. It is easy to see that $g_{1}$ is globally Newton non-degenerate. Thus, by Remark \ref{Newton minimal exponent}, we have $\widetilde{\alpha}(g_{1})=\widetilde{\alpha}(P(g_{1}))=\tfrac{3}{2}$, which implies that
$$
\widetilde{\alpha}_{0}(Z)=\widetilde{\alpha}(P(g_{2}))<\widetilde{\alpha}(P(g_{1})).
$$
By Lemma \ref{minimal exponent in terms of hypersurface}, it follows that $\widetilde{\alpha}_{0}(Z)=\widetilde{\alpha}_{(u_{2},0)}(g_{2})$ for some $u_{2}\in \bbc$ and $\widetilde{\alpha}_{0}(Z)<\widetilde{\alpha}_{(u_{1},0)}(g_{1})$ for all $u_{1}\in \bbc$. 
This means that the only natural way the value of $\widetilde{\alpha}_{0}(Z)$ could be explained by the minimal exponent of a Newton non-degenerate hypersurface is if $g_{2}$ is Newton non-degenerate. However, $g_{2}$ is not Newton non-degenerate. To see this, identify the standard basis of $\bbr^{4}$ with $z_{1},x,y,z$, and identify $(\bbr^{4})^{*}$ with $\bbr^{4}$ via the standard evaluation pairing. If we put $Q'=F((5,1,2,2),P(g_{2}))$, then $g_{2,Q'}(z_{1},x,y,z) = x^6+2x^4y+x^2y^2$, which contains the point $(1,\sqrt{-1},1,1)\in \bbg_{m}^{4}$ in its singular locus.
\end{example}

\appendix

\end{document}